\documentclass{amsart}
\usepackage{amssymb,amsmath}
\usepackage{float}
\usepackage{xcolor}
\usepackage{geometry}
\usepackage{caption}
\usepackage{tikz}
\usepackage[title,titletoc]{appendix}

\colorlet{linkequation}{blue}
\usepackage[unicode,naturalnames,colorlinks,plainpages=true,pdfpagelabels,hypertexnames=true,pdfstartview=FitV,linkcolor=blue,citecolor=red,urlcolor=black]{hyperref}

\theoremstyle{plain}
\newtheorem{theorem}{Theorem}[section]
\newtheorem{lemma}[theorem]{Lemma}
\newtheorem{proposition}[theorem]{Proposition}
\newtheorem{cor}[theorem]{Corollary}
\theoremstyle{definition}
\newtheorem{definition}[theorem]{Definition}
\theoremstyle{remark}
\newtheorem{remark}[theorem]{Remark}
\numberwithin{equation}{section}

\newcommand{\ep}{{ \epsilon  }}
\newcommand{\pa}{\partial}
\newcommand{\R}{{ \mathbb{R}  }}
\newcommand{\si}{\sigma}
\newcommand{\bke}[1]{\left( #1 \right)}
\newcommand{\norm}[1]{\left\Vert #1 \right\Vert}
\newcommand{\na}{\nabla}
\newcommand {\al}{\alpha}

\DeclareMathOperator{\supp}{supp}

\newcommand{\Rn}{{\mathbb R}^{n-1}}
\newcommand{\de}{\delta}
\newcommand{\Ga}{\Gamma}

\begin{document}
\title[Flow reversal for the Stokes system in the half-space]{Flow reversal for the Stokes system with localized boundary data in the half-space}

\author{TongKeun Chang}
\address{TongKeun Chang: Department of
Mathematics, Yonsei University, 50 Yonsei-ro, Seodaemun-gu, Seoul 03722, Republic of Korea} \email{chang7357@yonsei.ac.kr }

\author{KyungKeun Kang}
\address{KyungKeun Kang: Department of
Mathematics, Yonsei University, 50 Yonsei-ro, Seodaemun-gu, Seoul 03722, Republic of Korea} \email{kkang@yonsei.ac.kr }

\author{ChanHong Min}
\address{ChanHong Min: Department of
Mathematics, Yonsei University, 50 Yonsei-ro, Seodaemun-gu, Seoul 03722, Republic of Korea} \email{chopinmin1@yonsei.ac.kr }

% \thanks{T. Chang's work is supported by RS-2023-00244630.
% K. Kang's work is supported by NRF grant nos. RS-2024-00336346 and RS-2024-00406821.
% C. Min's work is supported by RS-2024-00336346, and the National Research Foundation of Korea Grant funded by the Korean Government RS-2023-00212227.}

\date{}

\makeatletter
\@namedef{subjclassname@2020}{%
\textup{2020} Mathematics Subject Classification} \makeatother


\begin{abstract}
We study the nonstationary Stokes system in the half-space with zero initial data and a localized inflow through the boundary whose time profile behaves like \((1-t)^a\), \(-1<a\leq1\), before it is switched off at \(t=1\).
By careful pointwise estimates of the representation formula of the solution to the Stokes system in the half-space, we show that the flow far from the inflow region reverses its direction.
For \(t>1\), the tangential velocity component pointing away from the source is negative near the wall and positive above it, and the normal velocity changes sign at a height comparable to the distance \(|x'|\) from the source.
The way the inflow is switched off determines when the reversal starts: if the inflow amplitude tends to zero (\(0<a\leq1\)), the reversal is already present shortly before \(t=1\), whereas if the amplitude diverges (\(-1<a<0\)), the tangential component pointing away from the source does not change sign shortly before \(t=1\), and neither does the normal component under a mild monotonicity condition on the profile.
For \(-1<a\leq-\frac12\) the reverse-flow region first appears next to the wall at \(t=1\), which yields separation points.
\end{abstract}
\maketitle

\section{Introduction}

Boundary-layer separation, and the reversed flow that accompanies it, are among the basic phenomena of viscous flow along a wall: near a separation point the fluid adjacent to the wall moves against the outer stream, and a region of reverse flow develops next to the boundary.
Experimental and theoretical studies indicate that flow reversal is closely associated with the onset of separation (see \cite{Dassrinivasan, klemp, pedley}).
The purpose of this paper is to show that flow reversal, and separation in a precise sign-based sense, already occur for the Stokes system, the linearization of the Navier--Stokes equations around the zero solution, and to describe them pointwise.
This is a first step toward understanding the mechanisms underlying flow reversal and, possibly, their connection to singular behavior near the boundary and to boundary-layer separation.

We consider the Stokes system in the half-space \(\R_+^n:=\R^{n-1}\times (0,\infty)\), \(n\geq 2\), with zero initial data,
\begin{equation}\label{StokesRn+}
\left\{
\begin{aligned}
&\partial_t w-\Delta w+\nabla p=0,\quad \nabla\cdot w=0 &&\text{in } \mathbb{R}^n_+\times(0,\infty),\\
&w(x',0,t)=g(x',t) &&\text{on } \partial\mathbb{R}^n_+\times(0,\infty),\\
&w(x,0)=0 &&\text{in } \mathbb{R}^n_+,
\end{aligned}
\right.
\end{equation}
where \( w:\R_+^n\times (0,\infty)\rightarrow \R^n\) is the velocity of the fluid, \(p:\R_+^n\times (0,\infty)\rightarrow \R\) is the associated pressure and \(g:\R^{n-1}\times (0,\infty)\rightarrow \R^{n}\) is the boundary data.

The boundary data describe a localized inflow, \(g(x',t)=\psi(x')\phi(t)\mathbf{e}_n\), where \(\psi\geq0\) is smooth and compactly supported and the time profile \(\phi\geq0\) behaves like \((1-t)^a\), \(-1<a\leq1\), before the inflow is switched off at \(t=1\) (see \eqref{0502-6}--\eqref{boundarydataspecific}).
Thus the inflow amplitude tends to zero at \(t=1\) if \(a>0\) and diverges while remaining integrable if \(a<0\).
Far from the inflow region we determine the sign of each velocity component as a function of the height \(x_n\), for times shortly before \(t=1\) and for all times after it, and we locate the heights at which the sign changes.
Our results can be summarized as follows.
\begin{enumerate}
\item For \(t>1\), the tangential component pointing away from the source is negative in a near-wall region and positive above it, while the normal component is negative below a height comparable to the distance \(|x'|\) from the source and positive above it.
The two components change sign at different heights, so that they have opposite signs in an intermediate region.
\item The way the inflow is switched off determines when the reversal starts.
If \(0<a\leq1\), the reversal is already present shortly before \(t=1\); if \(-1<a<0\), the tangential component directed away from the source is positive at every height shortly before \(t=1\), as is the normal component if \(\phi(s)\leq\phi(\frac12)\) for \(0<s<\frac12\), and the reversal occurs only at \(t=1\).
\item If \(-1<a\leq-\frac12\), the reverse-flow region first appears next to the wall at \(t=1\): its vertical extent tends to zero as \(t\downarrow1\), at an explicit rate, and \((x',0,1)\) is a separation point in the sense of Definition \ref{defsp}.
If \(-\frac12<a\leq1\), the vertical extent of the reverse-flow region remains bounded below as \(t\downarrow1\), and \((x',0,1)\) is not a separation point for the tangential component directed away from the source.
\end{enumerate}

\medskip
\noindent{\bf Separation and reversal in boundary-layer models.}
Rigorous descriptions of separation have so far been obtained mainly for asymptotic models that approximate the Navier--Stokes equations near the boundary, such as the Prandtl equations and the triple-deck model.
The two-dimensional Prandtl equations, introduced by Prandtl \cite{Prandtl}, read
\begin{align}\label{prandtl}
\begin{cases}
\partial_t u+u\partial_x u+v\partial_y u-\partial_y^2 u=-\partial_x p_E,\\
\partial_x u+\partial_y v=0, \\
u|_{t=0}=u_0,\quad u|_{y=0}=v|_{y=0}=0,\quad \displaystyle\lim_{y\rightarrow \infty}u(x,y,t)=u_E(x,t),
\end{cases}
\end{align}
where \((u_E, p_E)\) are related by the Bernoulli relation \(\partial_t u_E + u_E\partial_x u_E=-\partial_x p_E\).

In the steady case, the Prandtl equations are globally well-posed for all \(x>0\) under a favorable pressure gradient, \(\frac{dp_E}{dx} \leq 0\), whereas under an adverse pressure gradient, \(\frac{dp_E}{dx} > 0\), only local well-posedness holds and classical solutions may cease to exist beyond some point \(x^*>0\), the separation point (see \cite{Oleinik}).
Goldstein's formal analysis \cite{Goldstein} predicts that the wall shear stress \(\partial_y u(x,0)\) then vanishes like \((x^*-x)^{\frac12}\) as \(x \nearrow x^{*}\) and that the solution becomes singular there, the \emph{Goldstein singularity}; this was justified rigorously in the seminal work of Dalibard and Masmoudi \cite{Dali-Mas} (see also \cite{SWZ}).
Solutions \emph{beyond} the Goldstein singularity, for which \(u\) changes sign, have been constructed only recently: Iyer and Masmoudi \cite{masreversal} studied reversal near explicit reversed solutions of Falkner--Skan type, and Dalibard, Marbach and Rax \cite{dalireversal} developed a theory of the forward-backward parabolic problems that arise in the recirculation zone.

In the unsteady case, the computations of van Dommelen and Shen \cite{vanshen1,vanshen2} indicate the formation of a singularity in finite time while the tangential velocity remains bounded.
For the trivial outer flow, E and Engquist \cite{eng} proved finite-time blow-up of derivatives of the Prandtl solution under a symmetry assumption, and for the van Dommelen--Shen outer flow, Kukavica, Vicol and Wang \cite{kvw} proved finite-time blow-up of the boundary-layer thickness through a convexity argument for a Lyapunov functional related to the displacement thickness (see also \cite{unsteadyinvisicid,unsteadyviscous}).
The triple-deck model refines the Prandtl system near separation, and its well-posedness and ill-posedness have been studied in \cite{Iyervicol, iyermaekawa, Dietertvaret, Varetiyer} and the references therein.

\medskip
\noindent{\bf The Stokes system near the boundary.}
A second motivation comes from the regularity theory of the Stokes system.
In contrast to the heat equation, a solution of the Stokes system need not be regular near a flat part of the boundary on which it vanishes, owing to the nonlocal effect of the pressure.
Let \(B^+_{r}:=\{ x\in \R^n: |x|<r,\ x_n >0\}\) and \(Q_r^+:=B_r^+\times (1-r^2,1)\).
In \cite{Kang05}, a solution of the Stokes system in \(Q_1^+\) vanishing on the flat part of the boundary was constructed with \(\norm{\partial_{x_n}w}_{L^{\infty}(Q_{1/2}^+)}=\infty\), although \(\norm{w}_{H^1(Q_1^+)}+\norm{w}_{L^{\infty}(Q_1^+)}<\infty\).
Further solutions with singular normal derivatives were obtained in \cite{CK24,KM,KangTsai22, KangTsai22FE, CCK18}, and in \cite{CK26} for finite-energy solutions of the Stokes and Navier--Stokes equations with homogeneous initial and boundary data and non-smooth, spatially localized external forces; in \cite{CK25} the velocity itself is unbounded near the boundary, \(\norm{w}_{L^{\infty}(Q_{1/2}^+)}=\infty\), while the velocity and its gradient remain locally square integrable.
Likewise, Caccioppoli's inequality \(\norm{\nabla w}_{L^2(Q^+_{1/2})}\leq C\norm{w}_{L^2(Q^+_1)}\), which holds in the interior (see \cite{J,Wolf}, and \cite{Jin-Kang17} for non-Newtonian flows), fails near the boundary, and so does the Caccioppoli-type inequality \(\norm{\nabla u}^2_{L^2(Q^+_{1/2})}\le C\bke{\norm{u}^2_{L^2(Q^+_1)}+\norm{u}^3_{L^3(Q^+_1)}}\) for the Navier--Stokes equations \cite{CK20}.
In all these examples the singular behavior is created, away from the support of the boundary data, by localized non-smooth boundary data.
It is natural to ask how the direction of the flow behaves in such regions, and the heuristic that singular behavior near the boundary is related to a reversal of the flow direction is part of our motivation.

An advantage of the Stokes system over the Navier--Stokes equations is that its solutions with general initial and boundary data are given by an explicit representation formula via the Green tensor, due to Golovkin \cite{Golovkin} and Solonnikov \cite{So,So2} (see also \cite{KangTsai22} for the unrestricted Green tensor), so that the flow can be analyzed pointwise, including near the boundary.
In the same half-space setting, Chang and Kang \cite{CKseparation} proved that, for sufficiently regular localized boundary data, a separation point exists if a certain singular integral determined by the boundary data is negative and does not exist if it is positive, and they studied the motion of the separation point and the sign of the pressure gradient; see also \cite{KM2026}.
The present paper is complementary: for the profile \((1-t)^a\), which may be singular at \(t=1\), we describe the signs of all velocity components as functions of the height \(x_n\), before and after \(t=1\), locate their zeros, and show how the exponent \(a\) determines when the reversal starts and whether the reverse-flow region first appears next to the wall at \(t=1\).

\medskip
\noindent{\bf The mechanism.}
The reversal can be read off from the decomposition \eqref{1220-8} of a tangential component,
\begin{align*}
w_i=w^{L}_{ni}+w^{B}_i+w^{N}_i .
\end{align*}
The vector field \(w^N=(w_1^N,\ldots,w_n^N)\) is the instantaneous potential-flow contribution induced by the boundary data; its tangential component \(w_i^N\) points away from the source and is proportional to \(\phi(t)\).
The term \(w_i^B\) is a boundary correction. The kernel in \(w_i^B\) contains the factor \(e^{-x_n^2/(4(t-s))}\), so the effect of the boundary data at time \(s\) is concentrated within a distance of order \(\sqrt{t-s}\) from the wall and decays exponentially beyond that distance; see \eqref{1101-01}. The correction cancels \(w_i^N\) at the wall, \(w_i^B=-w_i^N\), thereby enforcing the no-slip condition, and has the opposite sign in the regime considered.
The remaining contribution \(w^{L}_{ni}\) is positive and has stronger far-field decay, with an additional factor of order \(|x'|^{-2}\) in the pointwise estimates for small \(x_n\); see \eqref{est-wi} and \eqref{0914-1}--\eqref{0914-2}.
For \(t>1\), \(w_i^N\) vanishes immediately while \(w_i^B\) remains, so the flow near the wall is reversed; farther from the wall, where \(w_i^B\) is exponentially small, the positive contribution \(w^{L}_{ni}\) prevails.
If \(\phi(t)\) tends to zero as \(t\to1\), the same imbalance is present shortly before \(t=1\); if \(\phi(t)\) diverges, \(w_i^N\) dominates and no reversal is possible before \(t=1\).
Finally, \(a=-\frac12\) is the threshold value: the integral
\begin{align*}
\int_0^1(1-s)^a(t-s)^{-\frac12}\,ds
\end{align*}
is unbounded as \(t\downarrow1\) for \(a\leq-\frac12\). In this range, the contribution \(w^{L}_{ni}\) becomes large just after \(t=1\), so the reverse flow is confined to a region that contracts toward the wall as \(t\downarrow1\).
For the normal component, \(w_n=-\sum_{k<n}w^{L}_{kk}+w_n^N\) by \eqref{1220-8w_n}, where the first term is negative below a height comparable to \(|x'|\) and positive above it, and \(w_n^N>0\) is again proportional to \(\phi(t)\).

\subsection{Main results}
We first introduce the sign-based definitions of a {\it separation point} and a {\it reversal point} used in this paper.

A separation point is a space-time point \((x_*',0,t_*)\) such that a tangential component of the velocity has one sign near the wall shortly before \(t_*\), while shortly after \(t_*\) it has the opposite sign in a narrow near-wall region and the original sign above it, with this region contracting toward the wall as \(t\downarrow t_*\).

\begin{definition}\label{defsp}
Let \(w=(w_1,\dots,w_n)\) be the solution of the Stokes system \eqref{StokesRn+}.
A point \(z^*=(x'_*, 0 ,t_*)\in \partial\mathbb{R}_+^{n}\times \mathbb{R}_+\) is called a {\bf separation point} of \(w\) if there exist \(i\in\{1,\dots, n-1\}\), positive constants \(\delta, \delta_1,\delta_2\), a function \(\alpha_i: (t_*-\delta_1, t_*)\rightarrow \mathbb{R}_{+}\) and nondecreasing functions \(\beta_{i1}\leq\beta_{i2}\) from \((t_*, t_*+\delta_2)\) to \((0,\delta)\) such that
\begin{itemize}
\item[(i)] \(w_i(x_*', x_n,t)>0\) (resp. \(<0\)) for all \(0<x_n<\alpha_i(t)\) and all \(t\in (t_*-\delta_1, t_*)\).

\item[(ii)] \(w_i(x_*', x_n,t)<0\) (resp. \(>0\)) for all \(0<x_n<\beta_{i1}(t)\) and all \(t\in (t_*, t_*+\delta_2)\).

\item[(iii)] \(w_i(x_*',x_n,t)>0\) (resp. \(<0\)) for all \(\beta_{i2}(t)<x_n<\delta\) and all \(t\in(t_*, t_*+\delta_2)\).

\item[(iv)] \(\lim_{t\rightarrow t_*^+}\beta_{i2}(t)=0\).
\end{itemize}
\end{definition}
Figure \ref{fig:figure 1} illustrates the conditions in Definition \ref{defsp}.
\begin{figure}[h]
\centering
\begin{tikzpicture}[x=0.72cm,y=0.80cm,>=latex,font=\small]
\fill[blue!18] (2.45,0) -- (2.45,3.45) .. controls (3.05,2.90) and (4.25,3.05) .. (5,2.20) -- (5,0) -- cycle;
\fill[orange!32] (5,0) .. controls (5.65,1.15) and (6.35,2.00) .. (7.00,2.55) .. controls (7.45,2.90) and (7.90,3.15) .. (8.20,3.30) -- (8.20,0) -- cycle;
\fill[blue!18] (5,0) .. controls (5.55,1.55) and (6.25,2.80) .. (7.00,3.45) .. controls (7.45,3.88) and (7.90,4.18) .. (8.20,4.35) -- (8.20,5.8) -- (5,5.8) -- cycle;
\draw[blue!65!black,line width=0.9pt] (2.45,3.45) .. controls (3.05,2.90) and (4.25,3.05) .. (5,2.20);
\draw[orange!85!black,line width=0.9pt] (5,0) .. controls (5.65,1.15) and (6.35,2.00) .. (7.00,2.55) .. controls (7.45,2.90) and (7.90,3.15) .. (8.20,3.30);
\draw[orange!85!black,line width=0.9pt] (5,0) .. controls (5.55,1.55) and (6.25,2.80) .. (7.00,3.45) .. controls (7.45,3.88) and (7.90,4.18) .. (8.20,4.35);
\draw[densely dotted] (1,5.8) -- (5,5.8);
\draw[densely dotted] (5,0) -- (5,5.8);
\draw[->,line width=0.8pt] (0,0) -- (10,0) node[right] {\(t\)};
\draw[->,line width=0.8pt] (1,-0.75) -- (1,6.55) node[above] {\(x_n\)};
\fill (5,0) circle (2.1pt);
\node[below left] at (1,0) {\(O\)};
\node[left] at (1,5.8) {\(\delta\)};
\node[below] at (2.45,0) {\(t_*-\delta_1\)};
\node[below] at (5,0) {\(t_*\)};
\node[below] at (8.20,0) {\(t_*+\delta_2\)};
\node[blue!65!black] at (3.45,3.45) {\(\alpha_i(t)\)};
\node[right,orange!85!black] at (8.20,4.35) {\(\beta_{i2}(t)\)};
\node[right,orange!85!black] at (8.20,3.30) {\(\beta_{i1}(t)\)};
\node[blue!70!black] at (3.55,1.35) {\(w_i>0\)};
\node[blue!70!black] at (6.80,4.75) {\(w_i>0\)};
\node[red!70!black] at (6.85,1.05) {\(w_i<0\)};
\end{tikzpicture}
\caption{Sign of \(w_i(x_*',x_n,t)\) near a separation point \((x_*',0,t_*)\) in the sense of Definition~\ref{defsp}.}\label{fig:figure 1}
\end{figure}
\begin{remark}\label{rmkdefsp}
For each \(t\in(t_*,t_*+\delta_2)\), the interval \(0<x_n<\beta_{i1}(t)\) on the vertical line through \(x_*'\) lies in the reverse-flow region.
No sign change is required for \(t_*-\delta_1<t<t_*\), and \(\alpha_i(t)\) need not vanish as \(t\rightarrow t_*\).
\end{remark}

A reversal point is an interior space-time point \((x'_*,x_n^*,t_*)\) such that \(x_n^*\) is a zero of a velocity component at time \(t_*\) and, for all times close to \(t_*\), this component changes sign as a function of \(x_n\) on the vertical line through \(x'_*\).

\begin{definition}\label{reversalpoint}
Let \(w\) be the solution of the Stokes system \eqref{StokesRn+} defined by \eqref{rep-bvp-stokes-w}.
For \(1\leq i\leq n\), a point \(z^*=(x'_*,x_n^*,t_*)\in\mathbb R^{n-1}\times\mathbb R_+\times\mathbb R_+\) is a {\bf reversal point (in the interior)} of \(w_i\) if there exist \(\delta_1>0\) and functions \(f_1,f_2,f_3:(t_*-\delta_1,t_*+\delta_1)\to(0,\infty)\) such that \(f_3(t_*)=x_n^*\) and, for every \(t\) in this interval,
\begin{align*}
w_i(x_*',f_3(t),t)=0,\qquad f_1(t)<f_3(t)<f_2(t),\qquad w_i(x_*',f_1(t),t)w_i(x_*',f_2(t),t)<0.
\end{align*}
\end{definition}

\begin{remark}
Thus \(w_i(x_*',\cdot\,,t)\) changes sign in \((f_1(t),f_2(t))\) for every \(t\) near \(t_*\), and \(x_n^*\) is a zero of \(w_i(x_*',\cdot\,,t_*)\) in \((f_1(t_*),f_2(t_*))\).
No regularity in \(t\) of \(f_1,f_2,f_3\) is assumed, and \(x_n^*\) itself need not be a point at which \(w_i(x_*',\cdot\,,t_*)\) changes sign.
\end{remark}

We now specify the boundary data \(g:\mathbb R^{n-1}\times \mathbb R_+\rightarrow \R^n\); their only nonzero component is the normal one:
\begin{align}\label{0502-6}
g(y', s) = g_n (y',s)\mathbf{e}_n= \psi (y')\phi(s)\mathbf{e}_n.
\end{align}
Here \(\psi \geq 0\), \(\psi\not\equiv0\) and \(\phi \geq 0\) are assumed to satisfy
\begin{align}\label{boundarydata}
\psi \in C^\infty_c (B'_{\frac12}), \quad  \supp\phi \subset [0,1], \quad   \phi   \in L^1({\mathbb R}) \cap C^\infty([0, 1)),\quad \phi(0)=0,
\end{align}
where \(B'_r:=\{y'\in\Rn: |y'|<r\}\), and
\begin{align}\label{boundarydataspecific}
\phi(s) = ( 1-s)^a\quad\text{for } s\in \left(\tfrac{1}{2},1\right),\qquad -1<a\leq 1.
\end{align}
We set
\begin{align}\label{1216-10}
M:=\int_0^{\frac12}\phi(s) ds,
\end{align}
and note that \(M>0\), since \(\phi\) is continuous on \([0,1)\) and \(\phi(\frac12)=2^{-a}>0\).
The function \(\psi\) is fixed throughout the paper, and \(\norm{\psi}_{L^1(\R^{n-1})}>0\).
The constants in our results depend on \(\psi\) through \(\norm{\psi}_{L^1(\R^{n-1})}\) and \(\norm{\psi}_{C^1(\R^{n-1})}\); this dependence is indicated in the main theorems and suppressed in the notation elsewhere.
We write \(A\approx B\) if \(cB\leq A\leq CB\) for positive constants \(c,C\) depending only on \(n\), \(a\), \(M\) and \(\psi\) (see Section \ref{sect3} for the general convention).

The heights at which the tangential components change sign are described by the following quantities.
For \(-1<a\leq-\frac12\), \(1<t<\frac98\) and \(x'\neq0\), let
\begin{align}\label{Xa}
X_a(x',t):=\begin{cases}
|x'|^2(t-1)^{-(a+\frac12)}, & -1<a<-\frac12,\\[1mm]
|x'|^2|\ln(t-1)|^{-1}, & a=-\frac12,
\end{cases}
\end{align}
and, for \(|x'|>1\), let
\begin{align}\label{Za}
Z_a(x',t):=\begin{cases}
\bigl(t\ln\frac{|x'|^2}{t}\bigr)^{\frac12}, & t\geq\frac98,\\[1mm]
(\ln|x'|)^{\frac12}, & 1<t<\frac98,\ \ -\frac12<a\leq1,\\[1mm]
\min\Bigl\{X_a(x',t)^{\frac{1}{1-2a}},\ \bigl(\ln(e+X_a(x',t))\bigr)^{\frac12}\Bigr\}, & 1<t<\frac98,\ \ -1<a\leq-\frac12,\\[1mm]
\bigl(\min\{\ln|x'|,\ |\ln(1-t)|\}\bigr)^{\frac12}, & \frac78<t<1,\ \ 0<a\leq1.
\end{cases}
\end{align}
For fixed \(x'\) and \(-1<a\leq-\frac12\), \(X_a(x',t)\to0\), and hence \(Z_a(x',t)\to0\), as \(t\downarrow1\).

Our first main result concerns the existence of separation points.
\begin{theorem}[\bf Existence of separation points]\label{thmseparationpoint}
Let \(w\) be the solution of the Stokes system \eqref{StokesRn+} defined by \eqref{rep-bvp-stokes-w} with the boundary data \(g=g_n\mathbf{e}_n\) given by \eqref{0502-6}--\eqref{boundarydataspecific}.
There exists a constant \(R_*>1\), depending on \(n,a,M\) and \(\psi\), such that the following holds for every \(x'\) with \(|x'|>R_*\) and every \(i\in\{1,\ldots,n-1\}\) with \(x_i\geq\frac12|x'|\).
\begin{itemize}
\item[(i)] If \(-1<a\leq -\frac12\), then conditions (i)--(iv) of Definition~\ref{defsp} hold for the component \(w_i\) at \(z^*=(x',0,1)\), that is, with \(t_*=1\), and the functions \(\beta_{i1}\leq\beta_{i2}\) can be chosen with
\begin{align*}
\beta_{ik}(t)\approx X_a(x',t)^{\frac{1}{1-2a}}=\begin{cases}|x'|^{\frac{2}{1-2a}}(t-1)^{\frac{a+\frac12}{2a-1}}, & -1<a<-\frac12,\\[1mm] |x'|\,|\ln(t-1)|^{-\frac12}, & a=-\frac12,\end{cases}
\qquad k=1,2.
\end{align*}
In particular, \((x',0,1)\) is a separation point of \(w\).
\item[(ii)] If \(-\frac12<a\leq1\), then \(w_i(x',x_n,t)<0\) for all \(0<x_n<\frac12\) and \(1<t<\frac98\); consequently, the conditions of Definition~\ref{defsp} at \((x',0,1)\) fail for \(w_i\), in either of the two alternatives.
\end{itemize}
\end{theorem}

The second main result describes the signs of the components of \(w\) and the location of their zeros, shortly before \(t=1\) and after it.

\begin{theorem}\label{thmasymptotic}
Let \(w\) be the solution of the Stokes system \eqref{StokesRn+} defined by \eqref{rep-bvp-stokes-w} with the boundary data \(g=g_n\mathbf{e}_n\) given by \eqref{0502-6}--\eqref{boundarydataspecific}, and let \(Z_a\) be given by \eqref{Za}.
There exist constants \(c_*>1\), \(R_*>1\), \(0<\tau_0<\frac18\), \(0<\mu_1\leq\mu_2\), \(0<\gamma_1\leq\gamma_2\) and \(\eta_0>0\), depending only on \(n,a,M\) and \(\psi\), such that the following holds.
\begin{itemize}
\item[(i)] 
Let \(t>1\) and \(|x'|>c_*(1+\sqrt t)\).
\begin{itemize}
\item[(a)] For every \(i\in\{1,\dots,n-1\}\) with \(x_i\geq\frac12|x'|\) there are \(0<x_{ni2}^*\leq x_{ni1}^*\) such that
\begin{align*}
w_i(x',x_n,t)<0\quad\text{for } 0<x_n<x_{ni2}^*,\qquad w_i(x',x_n,t)>0\quad\text{for } x_n>x_{ni1}^*,
\end{align*}
and
\begin{align*}
x_{nik}^*\approx Z_a(x',t),\qquad k=1,2;
\end{align*}
in particular, \(w_i(x',\cdot\,,t)\) has a zero, and all its zeros lie in \([x_{ni2}^*,x_{ni1}^*]\).
If \(-\frac12<a\leq1\) and \(t<\frac98\), then moreover \(x_{ni2}^*\geq\frac12\).
\item[(b)] \(w_n(x',x_n,t)<0\) for \(0<x_n<\mu_1|x'|\) and \(w_n(x',x_n,t)>0\) for \(x_n>\mu_2|x'|\).
If \(0<a\leq1\), this also holds at \(t=1\).
\end{itemize}

\item[(ii)]
Let \(|x'|>R_*\).
\begin{itemize}
\item[(a)] Let \(-1<a<0\).
If \(1-\tau_0<t<1\), then \(w_i(x',x_n,t)>0\) for every \(x_n>0\) and every \(i\in\{1,\dots,n-1\}\) with \(x_i\geq\frac12|x'|\).
If, in addition, \(\phi(s)\leq\phi(\frac12)=2^{-a}\) for \(0<s<\frac12\), then \(w_n(x',x_n,t)>0\) for every \(x_n>0\) and every \(\frac12<t<1\).
\item[(b)] Let \(0<a\leq1\) and \(1-\tau_0<t<1\).
For every \(i\in\{1,\dots,n-1\}\) with \(x_i\geq\frac12|x'|\) there are \(\frac12\leq x_{ni2}^*\leq x_{ni1}^*\) such that
\begin{align*}
w_i(x',x_n,t)<0\quad\text{for } 0<x_n<x_{ni2}^*,\qquad w_i(x',x_n,t)>0\quad\text{for } x_n>x_{ni1}^*,
\end{align*}
and \(x_{nik}^*\approx Z_a(x',t)\), \(k=1,2\).
Moreover, \(w_n(x',x_n,t)<0\) for \(0<x_n\leq1\) and \(w_n(x',x_n,t)>0\) for \(x_n>\gamma_2|x'|\), and if \((1-t)^a|x'|\leq\eta_0\), then \(w_n(x',x_n,t)<0\) for all \(0<x_n<\gamma_1|x'|\).
\end{itemize}
\end{itemize}
The case \(a=0\) is not covered by (ii).
\end{theorem}

The sign changes of the tangential and the normal components occur at different heights.

\begin{cor}\label{cor-sign-changes}
In the setting of Theorem \ref{thmasymptotic}, after enlarging \(c_*\) and \(R_*\), the following holds for every \(i\in\{1,\dots,n-1\}\) with \(x_i\geq\frac12|x'|\).
\begin{itemize}
\item[(i)] If \(t>1\) and \(|x'|>c_*(1+\sqrt t)\), then \(x_{ni1}^*<\mu_1|x'|\).
Hence \(w_i\) and \(w_n\) are both negative for \(0<x_n<x_{ni2}^*\), both positive for \(x_n>\mu_2|x'|\), and \(w_i>0>w_n\) for \(x_{ni1}^*<x_n<\mu_1|x'|\).
\item[(ii)] If \(0<a\leq1\), \(|x'|>R_*\), \(1-\tau_0<t<1\) and \((1-t)^a|x'|\leq\eta_0\), then the same holds with \(\gamma_1,\gamma_2\) in place of \(\mu_1,\mu_2\).
\end{itemize}
\end{cor}

\begin{remark}\label{rmk-sepfail}
\begin{itemize}
\item[(i)] The constant \(\frac12\) in the condition \(x_i\geq\frac12|x'|\) may be replaced by any fixed \(\kappa\in(0,1]\), the constant \(c_0\) of Lemma \ref{estimateofkernelt<1} being taken correspondingly large (see Remark \ref{rmk-shift}) and the constants in the theorems depending on \(\kappa\).
If \(\mathcal R_i\) denotes the reflection \(x_i\mapsto-x_i\), then \(\psi\circ \mathcal R_i\) satisfies the same hypotheses as \(\psi\), with the same norms, and \(w_i[\psi\circ \mathcal R_i](x,t)=-w_i[\psi](\mathcal R_ix,t)\); hence the case \(x_i\leq-\kappa|x'|\) reduces to \(x_i\geq\kappa|x'|\), with the sign pattern of \(w_i\) reversed, which is the alternative indicated by ``resp.'' in Definition \ref{defsp}.
Taking \(\kappa=(n-1)^{-\frac12}\) and \(i\) with \(|x_i|=\max_{j<n}|x_j|\), we obtain, for \(-1<a\leq-\frac12\), that \((x',0,1)\) is a separation point of \(w\) for every \(|x'|>R_*\), with \(R_*\) depending also on this choice.
\item[(ii)] For \(0<a\leq1\), \(w_i\) is also negative near the wall shortly before \(t=1\), by Theorem \ref{thmasymptotic}(ii)(b).
We expect that a different choice of \(\phi\) on \((0,\frac12)\), while retaining \(\phi(s)=(1-s)^a\) on \((\frac12,1)\), gives \(w_i(x,t)>0\) for \(t\ll1\) and \(x_n<\sqrt t\), and we expect that \((x',0,t_*)\) is then a separation point of \(w\) for some \(t_*\in(0,1)\).
\end{itemize}
\end{remark}

\begin{remark}[Comparison with \cite{KM2026}]\label{rmk-KM2026}
\begin{enumerate}
\item[(i)] Theorem \ref{thmseparationpoint} gives separation points only in the singular regime \(-1<a\leq-\frac12\), where separation occurs exactly at the switch-off time \(t=1\) and is caused by the singularity of the temporal boundary data there.
On the other hand, separation for regular temporal boundary data, with \(\phi\in C^{1,\beta}([0,\infty))\) for some \(\beta\in(0,1)\) and \(\phi=0\) on \([1,\infty)\), so that \(\phi(1)=\phi'(1)=0\), is studied in the work \cite{KM2026}; such profiles are not of the form \eqref{boundarydataspecific} for any \(-1<a\leq1\).
There, if \(\phi\) is in addition unimodal, the wall shear stress \(\partial_{x_n}w_i(x',0,t)\) is continuous and changes sign at a time \(t_i^*(x')\in(0,1)\) for every \(|x_i|>1\), and \((x',0,t_i^*(x'))\) is a separation point in the sense recalled in (ii).
Thus, for such data, separation occurs before the switch-off, while the boundary data still act, consistent with the expectation in Remark \ref{rmk-sepfail}(ii).

\item[(ii)] The notion of separation point in \cite{KM2026} is weaker than Definition \ref{defsp}: it only requires \(w_i\) to take both signs on sets of positive measure in every small cylinder \(\{|z'-x'|<r,\ 0<z_n<\varepsilon,\ |s-t_*|<r^2\}\), and it imposes no condition on the zero set of \(w_i\).
Since \(w_i\) is continuous in \(\R^n_+\times((0,1)\cup(1,\infty))\), conditions (i) and (ii) of Definition \ref{defsp} already imply it, so the separation points in Theorem \ref{thmseparationpoint} are also separation points in the sense of \cite{KM2026}.
The converse fails, because conditions (iii) and (iv) of Definition \ref{defsp} require the flow of the original sign to reach down to the wall as \(t\to t_*^+\).
For instance, if \(-\frac12<a<0\), \(|x'|>R_*\) and \(x_i\geq\frac12|x'|\), then Theorems \ref{thmseparationpoint}(ii) and \ref{thmasymptotic}(ii)(a) show that \((x',0,1)\) is a separation point in the sense of \cite{KM2026}, but not in the sense of Definition \ref{defsp}.
\end{enumerate}
\end{remark}
The third main result establishes the existence of reversal points.

\begin{theorem}[\bf Existence of reversal points]\label{rp-flow}
Let \(w\) be the solution of the Stokes system \eqref{StokesRn+} defined by \eqref{rep-bvp-stokes-w} with the boundary data \(g=g_n\mathbf{e}_n\) given by \eqref{0502-6}--\eqref{boundarydataspecific}, and let \(Z_a\), \(\tau_0\) and \(\eta_0\) be as in \eqref{Za} and Theorem \ref{thmasymptotic}.
There exists \(c_*>0\), depending only on \(n\), \(a\), \(M\) and \(\psi\), such that the following holds.

\begin{itemize}
\item[(i)] {\bf (Normal component.)}
For every \(t_*>1\) and every \(x'\) with \(|x'|>c_*(1+\sqrt{t_*})\) there is \(x_n^*\approx|x'|\) such that \((x',x_n^*,t_*)\) is a reversal point of \(w_n\).

If \(0<a\leq1\), then for every \(t_*\in(1-\tau_0,1]\) and every \(x'\) with \(|x'|>c_*\) there is \(x_n^*\) with \(1<x_n^*\lesssim|x'|\) such that \((x',x_n^*,t_*)\) is a reversal point of \(w_n\); if in addition \((1-t_*)^a|x'|<\eta_0\), then \(x_n^*\) can be chosen with \(x_n^*\approx|x'|\).

If \(-1<a<0\) and \(\phi(s)\leq\phi(\frac12)=2^{-a}\) for \(0<s<\frac12\), then for \(|x'|>c_*\), \(w_n\) has no reversal point \((x',x_n,t)\) with \(\frac12<t<1\).

\item[(ii)] {\bf (Tangential components, \(t>1\).)}
Let \(1\leq i\leq n-1\) and \(t_*>1\), and let \(x'\) satisfy \(x_i\geq\frac12|x'|\) and \(|x'|>c_*(1+\sqrt{t_*})\).
Then there is \(x_n^*\approx Z_a(x',t_*)\) such that \((x',x_n^*,t_*)\) is a reversal point of \(w_i\).

\item[(iii)] {\bf (Tangential components, \(t<1\).)}
Let \(1\leq i\leq n-1\) and let \(x'\) satisfy \(x_i\geq\frac12|x'|\) and \(|x'|>c_*\).
If \(0<a\leq1\), then for every \(t_*\in(1-\tau_0,1)\) there is \(x_n^*\approx Z_a(x',t_*)\) such that \((x',x_n^*,t_*)\) is a reversal point of \(w_i\).
If \(-1<a<0\), then \(w_i\) has no reversal point \((x',x_n,t)\) with \(1-\tau_0<t<1\).
\end{itemize}
The case \(a=0\) is not covered by the assertions for \(t\leq1\) in (i) and (iii).
\end{theorem}

Table \ref{tab:summary} summarizes the signs of \(w_i\) in the different regimes.
\begin{table}[H]
\centering
\small
\renewcommand{\arraystretch}{1.25}
\begin{tabular}{l|l|l|l|l}
 & \(w_i\), \(t<1\) & \(w_i\), \(t>1\) & \(w_n\), \(t<1\) & \(w_n\), \(t>1\)\\ \hline
\(-1<a\leq-\frac12\) & positive & changes sign; height \(\to0\) as \(t\downarrow1\) & positive & changes sign at height \(\approx|x'|\)\\
\(-\frac12<a<0\) & positive & changes sign at height \(\approx(\ln|x'|)^{\frac12}\) & positive & changes sign at height \(\approx|x'|\)\\
\(0<a\leq1\) & changes sign & changes sign at height \(\approx(\ln|x'|)^{\frac12}\) & changes sign & changes sign at height \(\approx|x'|\)
\end{tabular}
\caption{Behavior of \(w_i\) (\(x_i\geq\frac12|x'|\)) and of \(w_n\) for \(|x'|\) large and \(t\) close to \(1\); see Theorem \ref{thmasymptotic} for the precise statements and hypotheses. A separation point at \((x',0,1)\) occurs only in the first row. The case \(a=0\), \(t<1\), is not covered.}\label{tab:summary}
\end{table}

\begin{remark}
When the outer tangential flow points in the positive \(x_i\)-direction, \(\partial_{x_i}p>0\) is an adverse pressure gradient and \(\partial_{x_i}p<0\) is a favorable pressure gradient.
The explicit computations of the pressure gradient in the same half-space Stokes setting in \cite{CKseparation,KM2026} suggest, for boundary data of the present type, an adverse gradient in the regimes where reversal occurs and a favorable one in the regime without reversal before \(t=1\).
\end{remark}

The paper is organized as follows.
In Section~\ref{sect3} we fix the notation, recall the representation formula for \eqref{StokesRn+} in terms of the Poisson kernel of the Stokes system, and prove the kernel estimates on which the analysis is based.
Section~\ref{sign-w} treats \(t>1\): the tangential components in Proposition \ref{1stmainprop} and the normal component in Proposition \ref{2ndmainprop}.
Section~\ref{sect4} treats \(t<1\), in Propositions \ref{3rdmainprop} and \ref{0716}, and contains the proofs of Theorems \ref{thmseparationpoint}, \ref{thmasymptotic} and \ref{rp-flow} and of Corollary \ref{cor-sign-changes}.
The proofs of the auxiliary integral estimates and elementary inequalities are collected in the Appendix.

\section{Preliminaries}\label{sect3}

We first introduce the following convention.
Let \(f\) be a function of the displayed variables, and let \(g_1,\dots,g_I\), \(h_1,\dots,h_J\) and \(\Theta_1,\dots,\Theta_J\) be nonnegative functions of the same variables.
We write
\begin{align*}
f\approx \sum_{i=1}^{I} g_i +  \sum_{j=1}^{J} h_j e^{-\Theta_j}
\end{align*}
if there exist positive constants \(c_i,C_i,d_j,D_j,\alpha_j,\beta_j\), independent of the displayed variables, such that
\begin{align*}
\sum_{i=1}^I c_ig_i + \sum_{j=1}^J d_j h_j e^{-\alpha_j\Theta_j}\leq f \leq \sum_{i=1}^I C_i g_i + \sum_{j=1}^J D_j h_j e^{-\beta_j\Theta_j}.
\end{align*}
Thus the constants in the exponents may differ between the lower and the upper bound.
In particular, an explicit constant in an exponent carries no information, and a factor that is bounded above and below by powers of \(\Theta_j\) may be absorbed into \(e^{-\Theta_j}\) once \(\Theta_j\) is bounded below by a positive constant.
Similarly, we write \(f\lesssim \sum_{i=1}^I g_i + \sum_{j=1}^J h_j e^{-\Theta_j}\) if the upper bound holds.

When some terms enter with a minus sign, the lower bound uses the upper constants for those terms and the upper bound uses the lower constants; for instance, \(f\approx g_1-g_2\) means that \(c_1g_1-C_2g_2\leq f\leq C_1g_1-c_2g_2\) for some constants \(0<c_1\leq C_1\) and \(0<c_2\leq C_2\).
Unless explicitly stated otherwise, comparison constants may depend on fixed parameters such as \(n,a,M\) and on the fixed function \(\psi\) (through \(\norm{\psi}_{L^1(\Rn)}\) and \(\norm{\psi}_{C^1(\Rn)}\)), but not on \(x\) or \(t\).

Next, we recall the representation formula for the solution to the Stokes system in the half-space.
The solution of \eqref{StokesRn+} with general boundary data \(g=(g_1,\dots, g_n)\) is given by
\begin{align}\label{rep-bvp-stokes-w}
w_i(x,t) & = \sum_{j=1}^{n}\int_0^t \int_{\Rn} K_{ij}( x-y^{\prime},t-s)g_j(y^{\prime},s) dy^{\prime}ds,
\end{align}
where \(x-y':=(x'-y',x_n)\) and the Poisson kernel \(K_{ij} \) of the Stokes system in \(\R_+^n\times (0, \infty)\) is given as follows (see \cite{So}):
\begin{align}\label{Poisson-tensor-K}
\begin{split}
K_{ij}(x,t) &  =  -2 \delta_{ij} \partial_{x_n}\Ga(x,t) +4L_{ij} (x,t) +2  \de_{jn} \de(t)  \partial_{x_i} N(x),\quad i,j=1,2,\dots,n,
\end{split}
\end{align}
where
\begin{align}\label{L-tensor}
L_{ij} (x,t) & =  \partial_{x_j} \int_{\Rn\times(0,x_n)}   \partial_{z_n} \Ga(z,t) \partial_{x_i}   N(x-z)  dz,\quad i,j=1,2,\dots, n.
\end{align}
Here \(\Ga\) and \(N\) are the \(n\)-dimensional heat kernel and Newtonian kernel defined respectively by
\begin{align}\label{H-L-10}
\Ga(x,t) =\begin{cases}
(4\pi t)^{-\frac{n}{2}} e^{-\frac{ |x|^2}{4t}} &\quad t > 0,\\
0 &\quad t \leq 0,
\end{cases}
\quad 
N(x) =\begin{cases}
    \frac1{2\pi} \ln |x| &  \quad n =2,\\
    \displaystyle -\frac1{n(n -2)\omega_n} |x|^{-n +2} &\quad n \geq 3,\\
\end{cases}
\end{align}
where \(\omega_n\) is the volume of the \(n\)-dimensional unit ball.
We also denote by \(\Gamma'=\Gamma'(x',t)\) and \(\Gamma_1=\Gamma_1(x_n,t)\) the \((n-1)\)-dimensional and the one-dimensional heat kernels, so that \(\Gamma(x,t)=\Gamma_1(x_n,t)\Gamma'(x',t)\).

It is shown in \cite{So} that \(L_{ij}\) is smooth for \(x\neq 0\), \(t>0\), and satisfies the following relations:
\begin{align}\label{1006-3}
\sum_{i=1}^{n} L_{ii} = \frac12 \partial_{x_n} \Ga, \quad L_{ij} = L_{ji} \,\,(  1 \leq i,j \leq n-1),  \quad L_{in} = L_{ni}  + B_{in} \,\, ( i \neq n),
\end{align}
where
\begin{equation}\label{B-tensor}
B_{in}(x,t): = \text{p.v.}\int_{\Rn}\partial_{x_n} \Ga(x -z^\prime , t) \partial_{z_i} N( z^{\prime},0) dz^\prime.
\end{equation}

\subsection{Integral estimates}
In this subsection, we provide integral estimates that will be used throughout the paper.
The main kernel to be estimated is \(L_{ij}\), defined in \eqref{L-tensor}; we use Lemmas \ref{technicallemma} and \ref{lemma0709-1} for this purpose.

For \(x\in\mathbb{R}_+^n\), \(t>0\) and nonnegative integers \(m\) and \(k\), we introduce the following integral:
\begin{equation}\label{I-mk}
I_{m,k}(x,t):=\int_0^{x_n}\frac{x_n-z_n}{t^{\frac{3}{2}}}e^{-\frac{(x_n-z_n)^2}{4t}}\frac{z_n^k}{(|x'|^2+z_n^2)^{\frac{m}{2}}}dz_n.
\end{equation}
We begin with an estimate of \(I_{m,k}\), whose proof is given in Appendix \ref{proofoftechnicallemma}.
\begin{lemma}\label{technicallemma}
Let \( I_{m,k}\) be the integral given in \eqref{I-mk}.
For \(x\in \mathbb{R}_+^n\) and \(t>0\) such that \(|x'|\geq \sqrt{t}\), we have
\begin{align*}
I_{m,k}(x,t)&\approx t^{-\frac{1}{2}}x_n^k|x|^{-m}\min\left\{1,x_n^2t^{-1}\right\}.
\end{align*}
\end{lemma}

We now provide convolution-type estimates of the heat and Newtonian kernels.
The main term of each convolution turns out to be the derivative of the Newtonian kernel, while the error terms are of lower order.
\begin{lemma}\label{lemma0709-1}
Let \( (x',t)\in (\mathbb{R}^{n-1}\setminus\left\{0\right\})\times \R_+\) satisfy \(|x'|\geq \max\{1, \sqrt{t}\}\) and \(1\leq i,j\leq n-1\).
Then the following holds.
\begin{itemize}
\item[(i)] For \( x_n > 0\),
\begin{align}\label{0515-1}
\int_{\Rn}   \partial_{x_i} \Ga'(x'-z',t)     \partial_{x_n} N( z',x_n) dz' =   \partial_{x_i } \partial_{x_n}  N( x)  +     \frac12 \partial_{x_i}  \Ga'(x',t)  \mathbf{1}_{x_n \leq  |x'|}  + J_1(x, t),
\end{align}
where \(|J_1(x, t) | \lesssim \frac{tx_n}{|x|^{n+2}|x'|}\).

\item[(ii)] For \( x_n > 0\),
\begin{align}\label{0306-1}
\int_{\Rn}  \partial_{x_i} \Ga'(x'-z',t)   \partial_{z_j}  N( z', x_n) dz' =\partial_{x_j} \partial_{x_i} N(x)  + J_2(x,t),
\end{align}
where \(|J_2(x,t)|\lesssim \frac{t}{|x|^{n+2}}\).
\item[(iii)]
\begin{align}\label{0515-2}
\textup{p.v.}\int_{\Rn}    \Ga'(x'-z',t)    \partial_{z_i }  N( z',0) dz' =   \partial_{x_i }  N( x', 0 )  + J_3(x', t),
\end{align}
where \(|J_3(x', t) | \lesssim \frac{ t}{|x'|^{n+1 }}\).
\end{itemize}
\end{lemma}

The proof of Lemma \ref{lemma0709-1} is given in Appendix \ref{proofoflemma}.

The following two lemmas, which are the main estimates of this section, provide pointwise estimates for the kernels \({L}_{ij}\) and \(B_{in}\) that appear in the solution formulas for \(w_i\) and \(w_n\) (see \eqref{wis}, \eqref{1220-8} and \eqref{1220-8w_n}).

\begin{lemma}\label{estimateofkernelt<1}
Let \( 1 \leq i \leq n-1\) and \(L_{ij}\) be given in \eqref{L-tensor}.
Fix \(c_0\geq3\).
There exist \(R_0>1\) and \(0<\ep_0<1\), depending only on \(n\) and \(c_0\), such that for \(|x'|\geq R_0\), \(x_n>0\) and \(\sqrt{t} < \ep_0 |x'|\), the following holds.
There exist \(\kappa_1,\dots,\kappa_4>0\), depending only on \(n\), such that
\begin{align}\label{240707}
\frac{\kappa_1|x'|^2- \kappa_2x_n^2}{t^{\frac12}|x|^{n+2}}\min\left\{1,\frac{x_n^2}{t}\right\}\leq \sum_{k=1}^{n-1}{L}_{kk}(x,t)\leq \frac{\kappa_3|x'|^2- \kappa_4x_n^2}{t^{\frac12}|x|^{n+2}}\min\left\{1,\frac{x_n^2}{t}\right\}.
\end{align}
If moreover \( |x'| \leq c_0 x_i\), then
\begin{align}\label{L^Lestimate}
{L}_{ni}(x,t) &\approx   \left(\frac{x_n}{|x|^{n+2}}+\frac{1}{t^{\frac{n+1}{2}}}e^{-\frac{c|x|^2}{t}}\right)\frac{x_i}{t^{\frac{1}{2}}} \min\left\{1,\frac{x_n^2}{t}\right\}.
\end{align}
\end{lemma}
The proof of Lemma \ref{estimateofkernelt<1} is given in Appendix \ref{proofofkernel}.
\begin{lemma}\label{lem-0317-01}
Let \(1\leq i\leq n-1\) and \(B_{in}\) be given in \eqref{B-tensor}.
Let \(c_0\) be as in Lemma \ref{estimateofkernelt<1}.
There exist \(R_0>1\) and \(0<\epsilon_0<1\), depending only on \(n\) and \(c_0\), such that for \(|x'|\geq R_0\), \(x_n>0\), \(|x'|\leq c_0x_i\) and \(\sqrt{t}<\epsilon_0|x'|\), the following holds.
\begin{align}\label{1101-01}
B_{in}(x,t)\approx - \frac{1}{t^{\frac32}}\frac{x_ix_n}{|x'|^n}e^{-\frac{x_n^2}{4t}}.
\end{align}
\end{lemma}
\begin{proof}
By \eqref{B-tensor} and \eqref{0515-2},
\begin{align*}
B_{in}(x,t)=\partial_{x_n}\Gamma_1(x_n,t)\,\Bigl(\text{p.v.}\int_{\Rn}\Gamma'(x'-z',t)\,\partial_{z_i}N(z',0)\,dz'\Bigr) =\partial_{x_n}\Gamma_1(x_n,t)\bigl(\partial_{x_i}N(x',0)+J_3(x',t)\bigr).
\end{align*}
Here \(\partial_{x_n}\Gamma_1(x_n,t)=-(4\pi)^{-\frac12}\,\frac{x_n}{2t^{\frac32}}\,e^{-\frac{x_n^2}{4t}}\).
Moreover, \(\partial_{x_i}N(x',0)=\frac{x_i}{n\omega_n|x'|^{n}}\geq\frac{1}{c_0n\omega_n}|x'|^{1-n}\) because \(|x'|\leq c_0x_i\), while, by Lemma \ref{lemma0709-1}(iii) and \(t<\epsilon_0^2|x'|^2\),
\begin{align*}
|J_3(x',t)|\leq C\frac{t}{|x'|^{n+1}}\leq C\epsilon_0^2|x'|^{1-n}\leq\frac12\,\partial_{x_i}N(x',0)
\end{align*}
after decreasing \(\epsilon_0\).
Hence \(\partial_{x_i}N(x',0)+J_3(x',t)\approx x_i|x'|^{-n}\), and \eqref{1101-01} follows.
\end{proof}
\begin{remark}\label{rmk-shift}
The kernel estimates are applied at the shifted points \(x-y'=(x'-y',x_n)\) with \(y'\in\supp\psi\subset B'_{\frac12}\).
If \(x_i\geq\frac12|x'|\) and \(|x'|\geq4\), then \(|x'-y'|\leq|x'|+\frac12\leq3\bigl(\frac12|x'|-\frac12\bigr)\leq3(x_i-y_i)\), so the condition \(|x'-y'|\leq c_0(x_i-y_i)\) of Lemmas \ref{estimateofkernelt<1} and \ref{lem-0317-01} holds for every such \(y'\).
\end{remark}

We now decompose the solution of \eqref{StokesRn+} with the boundary data \eqref{0502-6}.
Let us denote
\begin{align}\label{wis}
\begin{split}
w^L_{ij}(x,t) & :=    4\int_0^t   \int_{\Rn} L_{ij}(x- y', t -s) \psi(y') \phi(s) dy' ds, \quad 1 \leq i, j \leq n,\\
w^B_i(x,t) & :=  4\int_0^t   \int_{\Rn} B_{i n}(x- y',  t -s) \psi(y') \phi(s)dy' ds, \quad 1 \leq i \leq n-1, \\
w_i^N(x,t) & :=   2\phi(t)\int_{\Rn} \partial_{x_i} N(x -y')\psi(y')  dy', \quad 1\leq i \leq n,\\
w^G (x,t) &  :=  -2\int_0^t   \int_{\Rn}\partial_{x_n}\Ga(x- y', t -s)\psi(y') \phi(s) dy' ds,
\end{split}
\end{align}
where, as before, \(x-y'=(x'-y',x_n)\).

Then from \eqref{rep-bvp-stokes-w}, \eqref{Poisson-tensor-K} and \eqref{1006-3}, we have
\begin{align}
w_i(x,t) & =   w^{L}_{ni}(x,t) + w^{B}_{i}(x,t)+ w^{N}_i(x,t), \quad   1 \leq i \leq n-1,\label{1220-8}\\
w_n(x,t) &=    w_{nn}^L(x,t) + w^{N}_n(x,t) + w^G(x,t)\nonumber\\
& =     -\sum_{i =1}^{n-1}  w^{L}_{ii}(x,t) + w^{N}_n(x,t) .\label{1220-8w_n}
\end{align}

The next lemma locates the solution of a transcendental equation that arises in the estimates of \(w_i\); it is used to locate the sign changes of \(w_i\).

\begin{lemma}\label{lemma0630}
Let \(h(\theta)=\theta^b\ln\theta\) for \(\theta>0\), where \(b>0\).
Let \(R>2\) and \(c>0\).
Define
\begin{align*}
A_-:=\min\left\{4,\left(\frac{bc}{2b+3}\right)^{1/b}\right\},\qquad A_+:=2^{1/b}\max\left\{4,(bc)^{1/b}\right\},
\end{align*}
and
\begin{align*}
\theta_-^*:=A_-R^{1/b}(\ln R)^{-1/b},\qquad \theta_+^*:=A_+R^{1/b}(\ln R)^{-1/b}.
\end{align*}
Then \(h(\theta)<cR\) for \(2<\theta<\theta_-^*\), whereas \(h(\theta)>cR\) for \(\theta>\theta_+^*\).
\end{lemma}
The proof is given in Appendix \ref{proofoflemma0630}.

We also recall the following elementary inequality, which is useful in locating the zeros.
\begin{lemma}\label{lemma-log-power}
For fixed \(\epsilon>0\) and all \(x>1\),
\begin{align}\label{lemma0630-2}
\epsilon \ln x \mathbf{1}_{x\leq e^{\frac{1}{\epsilon}}}  +  \frac{e-1}{e}x^{\epsilon}\mathbf{1}_{x>e^{\frac{1}{\epsilon}}}&\leq x^\epsilon-1\leq (e-1)\epsilon \ln x\mathbf{1}_{x\leq e^{\frac{1}{\epsilon}}}  +  x^\epsilon \mathbf{1}_{x>e^{\frac{1}{\epsilon}}}.
\end{align}
\end{lemma}
\begin{proof}
    Write \(y=\epsilon \ln x\), and use \(y\leq e^y-1\leq (e-1)y\) for \(0<y\leq 1\) and \((1-e^{-1})e^y\leq e^y-1\leq e^y\) for \(y>1\).
\end{proof}

Finally, we introduce notation for pairs of heights that bracket a sign change of a function.
\begin{definition}\label{1202}
Let \(0\leq r_1<r_2\leq \infty\) and \(f:[0,\infty)\rightarrow \mathbb{R}\) be a measurable function.
We define
\begin{align*}
&S^{-+}(f;r_1,r_2):=\left\{(y_{1}, y_{2}) \mid r_1\leq y_1\leq y_2\leq r_2,\ f(s)<0 \textrm{ if } r_1<s<y_{1},\ f(s)>0 \textrm{ if } y_{2}<s<r_2\right\},\\
&S^{+-}(f;r_1,r_2):=\left\{(y_{1},y_{2}) \mid r_1\leq y_1\leq y_2\leq r_2,\ f(s)>0 \textrm{ if } r_1<s<y_{1},\ f(s)<0 \textrm{ if } y_{2}<s<r_2\right\}.
\end{align*}
\end{definition}
\begin{remark}\label{1125}
Let \(f\) be continuous on \((r_1,r_2)\), and let \((y_1^*,y_2^*)\in S^{-+}(f;r_1,r_2)\cup S^{+-}(f;r_1,r_2)\) with \(r_1<y_1^*\leq y_2^*<r_2\).
Then \(f\) takes both signs in \((r_1,r_2)\), so it has a zero there by the intermediate value theorem, and every zero of \(f\) in \((r_1,r_2)\) lies in \([y_1^*,y_2^*]\).
\end{remark}

\section{Signs of the velocity components and location of their zeros for \texorpdfstring{\(t>1\)}{t>1}}
\label{sign-w}

In this section, we determine the signs of \(w_j(x,t)\) (\(1\leq j\leq n\)) for \(t>1\) and locate their zeros.
The main results are Propositions~\ref{1stmainprop} and~\ref{2ndmainprop}, which correspond to the tangential and normal components of the velocity, respectively.
\subsection{Tangential components of the velocity field}

\subsubsection{{\bf Estimates of \texorpdfstring{\(w_i\)}{w\_i}}}
We begin by introducing several integrals that appear in the estimates of \(w_i\).
For \(x_n>0\), \(r>0\), \(t>1\), \(a>-1\), \(c>0\) and \(k \geq 0\), we define the following integrals.
\begin{align}
\mathcal{G}_{a,k}(x_n,t)&:=\int_{\frac{1}{2}}^1\frac{(1-s)^a}{(t-s)^k}\min\left\{1,\frac{x_n^2}{t-s}\right\}ds,\label{G-ak}\\
\mathcal{H}_{a,k}(r,t)&:=\int_{\frac{1}{2}}^1\frac{(1-s)^a}{(t-s)^k}e^{-\frac{cr^2}{t-s}}ds,\label{H-ak}\\
\mathcal {K}_{a,k}(x,t)& :=     \int_{\frac12}^1 \frac{(1-s)^a }{(t-s)^{k}}e^{-\frac{c|x|^2}{t-s}}   \min\left\{ 1,\frac{x_n^2}{t-s} \right\}  ds.\label{K_ak}
\end{align}
Here \(c>0\) is a constant, and we suppress the dependence on \(c\); in accordance with the convention of Section \ref{sect3}, it may change from line to line.

In the following lemma, we give pointwise estimates of \(w_i\) in terms of \(\mathcal{G}_{a,k}\), \(\mathcal{H}_{a,k}\) and \(\mathcal{K}_{a,k}\).
\begin{lemma}[{{\bf  Estimates of \(w_i\)}}]\label{w1L-w1B}
There exists \(R_1\geq 1\), depending only on \(n\), such that the following holds.
Let \(t>1\), \( |x'| \geq R_1(1+\sqrt{t})\), \(x_i \geq \frac12 |x'|\) and \(i\neq n\), and let \(M\) be given by \eqref{1216-10}.
Then for \(w_i\) given in \eqref{1220-8}, the following estimate holds.
\begin{align}\label{est-wi}
\begin{split}
w_i(x,t) &  \approx \frac{Mx_n x_i }{|x|^{n+2}} \frac{1}{\sqrt{t}}\min\left\{1,\frac{x_n^2}{t}\right\}  +   \frac{Mx_i }{t^{\frac{n+2}{2}}}e^{-\frac{c|x|^2}{t}}  \min\left\{1,\frac{x_n^2}{t}\right\} +  \frac{x_n x_i }{|x|^{n+2}}  \mathcal{G}_{a,\frac{1}{2}}(x_n,t)  +  x_i {\mathcal K}_{a,\frac{n+2}{2}}(x,t)\\
& \quad - \frac{x_i x_n}{|x'|^n}\Big( M t^{-\frac32 } e^{-\frac{cx_n^2}t}  + \mathcal{H}_{a,\frac{3}{2}}(x_n,t)  \Big).
\end{split}
\end{align}
\end{lemma}

\begin{proof}
From \eqref{1220-8}, we have, for \(t>1\),
\begin{align*}
w_i(x,t) = w_{ni}^L(x,t)+w_i^B(x,t)+w_i^N(x,t)= w_{ni}^L(x,t)+w_i^B(x,t),
\end{align*}
since \(\phi(t)=0\) for \(t>1\).
We first estimate \(w_{ni}^{L}\).
Using \(\eqref{wis}_1\), the estimate \eqref{L^Lestimate} at the shifted points \(x-y'\) (see Remark \ref{rmk-shift}), and \(|x'-y'|\approx |x'|\), \(x_i-y_i\approx x_i\) for \(y'\in \supp\psi\), we obtain
\begin{align*}
w_{ni}^{L}(x,t) &\approx \norm{\psi}_{L^1}\int_0^1  \left(\frac{x_n}{|x|^{n+2}}+\frac{1}{(t-s)^{\frac{n+1}{2}}}e^{-\frac{c|x|^2}{t-s}}\right)\frac{x_i}{(t-s)^{\frac{1}{2}}} \min\left\{1,\frac{x_n^2}{t-s}\right\} \phi(s)    ds\\
&=\norm{\psi}_{L^1}\left(\int_0^{\frac12}\cdots ds+\int_{\frac12}^{1}\cdots ds\right) =: \norm{\psi}_{L^1}(I + J).
\end{align*}
We estimate \(I\) and \(J\) separately.
For \(I\), we use \(t-s\approx t\) for \(t>1\) and \(0<s<\frac12\) to obtain
\begin{align*}
I&=   \frac{x_n x_i }{|x|^{n+2}} \int_0^\frac12 \frac{1}{\sqrt{t-s}}\min\left\{1,\frac{x_n^2}{t-s}\right\}\phi(s)ds + x_i \int_0^\frac12  \frac{1}{(t-s)^{\frac{n+2}{2}}}e^{-\frac{c|x|^2}{t-s}}  \min\left\{1,\frac{x_n^2}{t-s}\right\} \phi(s)ds\\
&  \approx    \frac{Mx_n x_i }{|x|^{n+2}} \frac{1}{\sqrt{t}}\min\left\{1,\frac{x_n^2}{t}\right\} +   \frac{Mx_i }{t^{\frac{n+2}{2}}}e^{-\frac{c|x|^2}{t}}  \min\left\{1,\frac{x_n^2}{t}\right\}.
\end{align*}
For \(J\), we find that
\begin{align*}
\begin{split}
J  &=  \frac{x_n x_i }{|x|^{n+2}} \int_{\frac12}^1 \frac{1}{\sqrt{t-s}}\min\left\{1,\frac{x_n^2}{t-s}\right\}\phi(s)ds + x_i \int_{\frac12}^1  \frac{1}{(t-s)^{\frac{n+2}{2}}}e^{-\frac{c|x|^2}{t-s}}  \min\left\{1,\frac{x_n^2}{t-s}\right\} \phi(s)ds\\
&  \approx     \frac{x_n x_i }{|x|^{n+2}}  \mathcal{G}_{a,\frac{1}{2}}(x_n,t)  + x_i {\mathcal K}_{a,\frac{n+2}{2}}(x,t),
\end{split}
\end{align*}
where \(\mathcal{G}_{a,k}\) and \(\mathcal{K}_{a,k}\) are given in \eqref{G-ak} and \eqref{K_ak} respectively.

On the other hand, by \(\eqref{wis}_2\) and Lemma \ref{lem-0317-01}, we find that
\begin{align*}
\begin{split}
w_i^B(x,t)  &\approx-\frac{x_ix_n}{|x'|^n}\int_0^1\frac{1}{(t-s)^{\frac{3}{2}}}e^{-\frac{cx_n^2}{(t-s)}} \phi(s)ds  \approx-\frac{x_i x_n}{|x'|^n} \Big( M t^{-\frac32 } e^{-\frac{cx_n^2}t}  + \mathcal{H}_{a,\frac{3}{2}}(x_n,t)  \Big).
\end{split}
\end{align*}
Summing all the estimates gives the result.
\end{proof}

In the following auxiliary lemmas, we give pointwise estimates of the integrals \eqref{G-ak}--\eqref{K_ak}.
The constants in \(\approx\) may depend on the fixed parameters \(a\) and \(k\); in the applications below, \(k\) ranges over a finite set, so this dependence is harmless.
The proofs of these lemmas are given in Appendices \ref{proofofelementarylemma}, \ref{proofofelementarylemma2} and \ref{proofofelementarylemma3}, respectively.

\begin{lemma}\label{elementarylemma}
Let \(x_n>0\), \(t>1\), \(a>-1\), \(k \geq 0\) and \(\mathcal{G}_{a,k}\) be the integral given in \eqref{G-ak}.
Then the following holds.
\begin{itemize}
\item[(i)] If \( t\geq\frac98\), we have \(\mathcal{G}_{a,k}(x_n,t)\approx \frac{1}{a+1}(t-1)^{-k} \min\left\{1,\frac{x_n^2}{t-1}\right\}\).
\item[(ii)] If \( 1 < t < \frac98\) and \(x_n\leq \sqrt{2(t-1)}\), we have
\begin{align*}
\mathcal{G}_{a,k}(x_n,t) & \approx  x_n^2  \left( \frac{2^{k-a} - (t-1)^{a-k} }{ a-k }\mathbf{1}_{a\neq k} +   |\ln (2(t-1))|\mathbf{1}_{a=k} + \frac1{1 +a}   (t-1)^{a-k}\right).
\end{align*}
\item[(iii)] If \( 1 < t < \frac98\) and \(x_n\geq \frac12\), we have
\begin{align*}
\mathcal{G}_{a,k}(x_n,t) & \approx  \frac{2^{k-a-1} - (t-1)^{-k+a+1} }{1-k+a } \mathbf{1}_{a\neq k-1} + |\ln (2(t-1))|\mathbf{1}_{a=k-1}   +\frac1{1 +a} (t-1)^{-k+a+1}.
\end{align*}
\item[(iv)] If \( 1 < t < \frac98\) and \(\sqrt{2(t-1)}\leq x_n\leq \frac12\), we have
\begin{align*}
\mathcal{G}_{a,k}(x_n,t) & \approx \frac{(t-1)^{a-k+1}}{1+a}\mathbf{1}_{-1<a<-1+\frac{k}{2}}+\frac{x_n^{2a-2k+2}-(t-1)^{a+1-k}}{a+1-k}\mathbf{1}_{-1+\frac{k}{2}\leq a<k-\frac12,a\neq k-1}\\
&\quad+\ln\left(\frac{x_n^2}{t-1}\right)\mathbf{1}_{a=k-1}+x_n\mathbf{1}_{a=k-\frac12}+x_n^2\frac{2^{k-a}-x_n^{2a-2k}}{a-k}\mathbf{1}_{k-\frac12<a<k+\frac12,a\neq k}\\
&\quad+x_n^2|\ln(2x_n^2)|\mathbf{1}_{a=k}+\frac{x_n^2}{2^a a}\mathbf{1}_{a\geq k+\frac12}.
\end{align*}
\end{itemize}
\end{lemma}
\begin{lemma}\label{elementarylemma2}
Let \(r>0\), \(t>1\), \(a>-1\), \(k\geq 0\) and \(\mathcal{H}_{a,k}\) be the integral given in \eqref{H-ak}.
Then the following holds.
\begin{itemize}
\item[(i)] If \(t\geq\frac98\), then \( \mathcal{H}_{a,k}(r,t)\approx \frac{1}{a+1}(t-1)^{-k}e^{-\frac{cr^2}{t-1}}\).
\item[(ii)] Let \(1<t<\frac98\).
\begin{itemize}
\item[(a)] If \(r\leq \sqrt{2(t-1)}\), then
\begin{align*}
\mathcal{H}_{a,k}(r,t)\approx \frac{1-(t-1)^{a+1-k}}{a+1-k}\mathbf{1}_{a\neq k-1}+|\ln(t-1)|\mathbf{1}_{a=k-1} +  \frac{1}{a+1}(t-1)^{a+1-k}e^{-\frac{cr^2}{t-1}}.
\end{align*}

\item[(b)] If \(\sqrt{2(t-1)}\leq r\leq \frac12\), then
\begin{align*}
\mathcal{H}_{a,k}(r,t)\approx \frac{1-r^{2(a-k+1)}}{a+1-k}\mathbf{1}_{a\neq k-1}+|\ln r|\mathbf{1}_{a=k-1}  +  \frac{1}{a+1}(t-1)^{a+1-k}e^{-\frac{cr^2}{t-1}}.
\end{align*}

\item[(c)] If \(r\geq \frac12\), then \(\mathcal{H}_{a,k}(r,t)\approx r^{-2}e^{-cr^2} + \frac{1}{a+1}(t-1)^{a+1-k}e^{-\frac{cr^2}{t-1}}\).
\end{itemize}
\end{itemize}
\end{lemma}
\begin{lemma}\label{elementarylemma3}
Let \(|x'|\geq 2\), \(x_n>0\), \(t>1\), \(a>-1\), \(k\geq 0\) and \(\mathcal{K}_{a,k}\) be the integral given in \eqref{K_ak}.
Then the following holds.
\begin{itemize}
\item[(i)] If \(t\geq \frac98\), then \(\mathcal{K}_{a,k}(x,t)\approx \frac{1}{a+1}(t-1)^{-k}e^{-\frac{c|x|^2}{t-1}}\min\left\{1,\frac{x_n^2}{t-1}\right\}\).
\item[(ii)] If \(1<t<\frac98\), then \(\mathcal{K}_{a,k}(x,t)\approx |x|^{-2}e^{-c|x|^2}\min\left\{1, x_n^2\right\}+\frac{1}{1+a}(t-1)^{a+1-k}e^{-\frac{c|x|^2}{t-1}}\min\left\{1,\frac{x_n^2}{t-1}\right\}\).
\end{itemize}
\end{lemma}

In Lemmas \ref{elementarylemma} and \ref{elementarylemma2}, expressions such as \(\frac{1-(t-1)^{\beta}}{\beta}\) and \(\frac{2^{-\beta}-(t-1)^{\beta}}{\beta}\) are comparable with constants depending on \(\beta\), hence on \(a\) and \(k\) (see Remark \ref{1217}).

The following lemma shows that under certain conditions on \((x,t)\) and \(a\), the integral \(\mathcal{K}_{a,\frac{n+2}{2}}\) is controlled by the integral \(\mathcal{G}_{a,\frac12}\).
\begin{lemma}\label{1211}
Let \(1<t<\frac98\), \(|x'|\geq2\), and let \(\mathcal{G}_{a,k}\) and \(\mathcal{K}_{a,k}\) be the integrals given in \eqref{G-ak} and \eqref{K_ak} respectively.
If either \(\sqrt{2(t-1)}\leq x_n<\frac12\) and \(-1<a\leq 0\), or \(x_n\geq\frac12\), then
\begin{align*}
\mathcal{K}_{a,\frac{n+2}{2}}(x,t)\lesssim \frac{x_n}{|x|^{n+2}}\mathcal{G}_{a,\frac12}(x_n,t).
\end{align*}
\end{lemma}
\begin{proof}
We first consider the case \(\sqrt{2(t-1)}\leq x_n<\frac12\) and \(-1<a\leq 0\).
We claim that
\begin{align}\label{0317-06}
\frac{|x|^{n+2}}{x_n}\mathcal{K}_{a, \frac{n+2}{2}}(x,t)\lesssim x_n\lesssim \mathcal{G}_{a,\frac12}(x_n,t).
\end{align}
For the right inequality of \eqref{0317-06}, note that \(\mathcal{G}_{a,\frac12}(x_n,t)\) is a decreasing function in \(a\).
Thus from Lemma \ref{elementarylemma}, we find that for \(-1<a\leq 0\),
\begin{align*}
\mathcal{G}_{a,\frac12}(x_n,t)\geq \mathcal{G}_{0,\frac12}(x_n,t)\approx x_n.
\end{align*}
We now show the left inequality of \eqref{0317-06}.
Splitting the \(s\)-integral at \(s=2-t\), using \(1-s\approx t-s\) and \(\min\{1,\frac{x_n^2}{t-s}\}\leq\frac{x_n^2}{t-s}\) for \(\frac12<s<2-t\), and \(t-s\approx t-1\) for \(2-t<s<1\), and then changing variables, we obtain
\begin{align*}
\frac{|x|^{n+2}}{x_n}\mathcal{K}_{a,\frac{n+2}{2}}(x,t) &\lesssim x_n |x|^{n+2} \Bigg( \int_{2(t-1)}^{t-\frac12} u^{a-\frac{n+4}{2}} e^{-\frac{c|x|^2}{u}} \, du + \frac{e^{-\frac{c|x|^2}{t-1}}}{(t-1)^{\frac{n+4}{2}}} \int_{2-t}^1 (1-s)^a \, ds \Bigg) \\
&= x_n |x|^{n+2} \int_{\frac{c|x|^2}{t-\frac12}}^{\frac{c|x|^2}{2(t-1)}} \left(\frac{c|x|^2}{v}\right)^{a-\frac{n+4}{2}} e^{-v}\frac{c|x|^2}{v^2} \, dv+ \frac{x_n}{a+1} \frac{|x|^{n+2}}{(t-1)^{\frac{n+2}{2}-a}} e^{-\frac{c|x|^2}{t-1}} \\
&\lesssim x_n |x|^{2a} \int_0^{\infty} v^{\frac n2-a} e^{-v} \, dv + x_n |x|^{2a} \lesssim x_n ,
\end{align*}
where the implicit constants depend on \(a\) and \(n\), and in the last line we used \(\sup_{y>0}y^{\frac{n+2}{2}-a}e^{-cy}<\infty\) with \(y=\frac{|x|^2}{t-1}\), and \(|x|^{2a}\leq1\) for \(a\leq0\) and \(|x|\geq1\).

We now consider the case \(x_n\geq\frac12\).
Here we find that
\begin{align*}
\frac{|x|^{n+2}}{x_n}\mathcal{K}_{a,\frac{n+2}{2}}(x,t)&=\frac{|x|^{n+2}}{x_n}\int_{\frac12}^1 \frac{(1-s)^a}{(t-s)^{\frac{n+2}{2}}}e^{-\frac{c|x|^2}{t-s}}\min\left\{1, \frac{x_n^2}{t-s}\right\} ds\\
&\lesssim \frac{1}{x_n}\int_{\frac12}^1 (1-s)^a e^{-\frac{cx_n^2}{t-s}}\min\left\{1, \frac{x_n^2}{t-s}\right\} ds \\
&\leq \frac{(t-1/2)^{\frac12}}{x_n}\int_{\frac12}^1 \frac{(1-s)^a}{(t-s)^{\frac12}}e^{-\frac{cx_n^2}{t-s}}\min\left\{1, \frac{x_n^2}{t-s}\right\} ds\\
&\lesssim \int_{\frac12}^1 \frac{(1-s)^a}{(t-s)^{\frac12}}\min\left\{1, \frac{x_n^2}{t-s}\right\} ds = \mathcal{G}_{a,\frac12}(x_n,t).
\end{align*}
This finishes the proof.
\end{proof}

The following corollary, a direct consequence of Lemma \ref{1211}, simplifies the expression \eqref{est-wi} for \(w_i\) when \(t\) is close to \(1\).
\begin{cor}\label{1216}
Let \(1<t<\frac98\), let \(x'\) and \(i\) be as in Lemma \ref{w1L-w1B}, and let \(w_i\) be given in \eqref{1220-8}.
Then the following hold.
\begin{itemize}
\item[(i)] If \(\sqrt{2(t-1)}\leq x_n <\frac12\) and \(-1<a\leq 0\), then
\begin{align*}
w_i(x,t)\approx \frac{Mx_n^3x_i}{|x|^{n+2}} + \frac{x_nx_i}{|x|^{n+2}}\mathcal{G}_{a,\frac12}(x_n,t) -\frac{x_ix_n}{|x'|^{n}}\left(M + \mathcal{H}_{a,\frac32}(x_n,t)\right).
\end{align*}
\item[(ii)] If \(\sqrt{2(t-1)}\leq x_n<\frac12\) and \(a>0\), then
\begin{align*}
w_i(x,t)\approx  x_i \mathcal{K}_{a,\frac{n+2}{2}}(x,t) - \frac{x_ix_n}{|x'|^n}\left(M + \mathcal{H}_{a,\frac32}(x_n,t)\right).
\end{align*}
\item[(iii)] If \(\frac12\leq x_n\leq |x'|\), then
\begin{align*}
w_i(x,t)\approx \frac{x_ix_n}{|x'|^{n+2}}\left(M+\mathcal{G}_{a,\frac12}(x_n,t)-|x'|^2 \left(Me^{-x_n^2}+\mathcal{H}_{a,\frac32}(x_n,t)\right)\right).
\end{align*}
\end{itemize}
\end{cor}
\begin{proof}
(i) and (iii) are straightforward from \eqref{est-wi} and Lemma \ref{1211}.

We now prove (ii).
Observe that for \(a>0\), \(\mathcal{G}_{a,\frac12}(x_n,t)\approx x_n^2 \mathcal{H}_{a,\frac32}(x_n,t)\).
Then we note that
\begin{align*}
\frac{Mx_i}{t^{\frac{n+2}{2}}}e^{-\frac{c|x|^2}{t}}\min\left\{1, \frac{x_n^2}{t}\right\} + \frac{x_nx_i}{|x|^{n+2}}\mathcal{G}_{a,\frac12}(x_n,t) &\approx Mx_i e^{-c|x'|^2}x_n^2 + \frac{x_nx_i}{|x'|^{n+2}}\mathcal{G}_{a,\frac12}(x_n,t)\\
&\lesssim \frac{ x_ix_n^2}{|x'|^{n+2}} \left(M|x'|^{n+2}e^{-c|x'|^2} + \mathcal{H}_{a,\frac32}(x_n,t)\right)\\
&\lesssim \frac{1}{|x'|^2}\frac{x_ix_n^2}{|x'|^{n}}\left(M+\mathcal{H}_{a,\frac32}(x_n,t)\right)\\
&\lesssim \frac{1}{|x'|^2}\frac{x_ix_n^2}{|x'|^n}\left(Mt^{-\frac32}e^{-\frac{x_n^2}{t}}+ \mathcal{H}_{a,\frac32}(x_n,t)\right).
\end{align*}
The result then follows from the inequality \(\frac{Mx_i x_n^3}{|x'|^{n+2}}\lesssim \frac{1}{|x'|^2}\frac{x_ix_n}{|x'|^n}\).
\end{proof}

\subsubsection{{\bf Asymptotics of \texorpdfstring{\(x_{ni}^*\)}{x\_ni*} and sign of \texorpdfstring{\(w_i\)}{w\_i}}}
Before stating our result on the asymptotics of \(x_{ni}^*\), we introduce certain sets whose defining conditions involve the symbols ``\(\lesssim\)'' and ``\(\gtrsim\)''.
We begin by clarifying the meanings of these symbols.
For example, for a given \(f:\mathbb{R}^{n-1}\times \R_+\rightarrow \R_+\) and \(A, B>0\), we define
\begin{align*}
U:=\left\{(x',t)\in \mathbb{R}^{n-1} \times \R_+ \mid c_1A\leq f(x',t)\leq c_2B\right\}
\end{align*}
for some constants \(c_1\) and \(c_2\) depending on \(n\), \(a\) and \(M\).
For notational convenience, we henceforth denote the above set by
\begin{align*}
U=\{(x',t)\in \R^{n-1}\times \R_+ \mid A\lesssim f(x',t)\lesssim B\}.
\end{align*}

Now we introduce the sets which will appear in the next proposition: for \(0\leq\alpha<\beta\),
\begin{align*}
&U_1(\alpha,\beta):=\left\{(x',t)\in \R^{n-1}\times \R_+   \mid  \alpha\lesssim |a+1||x'|^2\lesssim \beta\right\},\\
&U_2(\alpha,\beta):=\left\{(x',t)\in \R^{n-1}\times \R_+  \, \Big| \, \alpha\lesssim \left|a+\frac12\right||x'|^2\lesssim \beta\right\},\\
&U_3(\alpha,\beta):=\left\{(x',t)\in \R^{n-1}\times \R_+  \, \bigg|\,  \alpha\lesssim |x'|^2 \left(\ln \frac{|x'|^2}{t-1}\right)^{-1}\lesssim \beta\right\}.
\end{align*}
We then define the sets \(A_a^l\subset \R^{n-1}\times \R_+\), \(l\in \{1,2\}\), as follows.
\begin{align}\label{0804-2}
A_a^1:=\begin{cases} U_1(t-1,(t-1)^{a+\frac12}), &\textrm{ if } -1<a< -\frac34,\\
U_2(e^{-\frac{1-2a}{1+2a}}(t-1), (t-1)^{a+\frac12}), &\textrm{ if } -\frac34\leq a<-\frac12,\\
\emptyset, &\textrm{ if } a\geq -\frac12,
\end{cases}
\end{align}
\begin{align}\label{0804-3}
A_a^2:=\begin{cases} U_3(t-1, 1),&\textrm{ if } a=-\frac12,\\
\emptyset, &\textrm{ if } a\neq -\frac12.
\end{cases}
\end{align}
In addition, for any measurable function \(f:\R_+\rightarrow \R_+\), we define
\begin{align*}
U(f):=\left\{(x',t)\in \R^{n-1}\times \R_+\,\Big|\, 1\lesssim \frac{|x'|^2}{f(t)}\lesssim e^{c|x'|^2}\right\},
\end{align*}
where, in the upper bound, the constant \(c>0\) in the exponent is allowed to change as well (it is the constant \(c_3\) of Case 3 in the proof of Proposition \ref{1stmainprop}).
Since the function \(f_a\) defined below is bounded below by a positive constant depending on \(a\), the upper bound in \(U(f_a)\) is automatically satisfied once \(|x'|\) is sufficiently large; it is retained only to match the conditions in Case 3 of that proof.
We then define the set
\begin{align}\label{0804-4}
A_a^3:=U(f_a),
\end{align}
where \(f_a=f_a(t)\) is given by
\begin{align*}
f_a(t):=&\frac{(t-1)^{a+\frac12}}{a+1}\mathbf{1}_{-1<a< -\frac34}+|\ln(t-1)|\mathbf{1}_{a=-\frac12}\\
&+\left(\frac{1}{|a+\frac12|}(t-1)^{-|a+\frac12|}\mathbf{1}_{t< 1+\frac14 e^{-\frac{1}{|a+\frac12|}}}+|\ln(t-1)|\mathbf{1}_{t\geq 1+\frac14 e^{-\frac{1}{|a+\frac12|}}}\right)\\
&\times\left(\mathbf{1}_{-\frac34\leq a<-\frac12}+(t-1)^{|a+\frac12|}\mathbf{1}_{-\frac12<a\leq -\frac14}\right) + \mathbf{1}_{a>-\frac14}.
\end{align*}

We now state the main result of this subsection, which gives the asymptotics of \(x_{ni}^*\) for \(t>1\).

\begin{proposition}\label{1stmainprop}
Let \(w\) be the solution of the Stokes system \eqref{StokesRn+} defined by \eqref{rep-bvp-stokes-w} with the boundary data \(g=g_n\mathbf{e}_n\) given by \eqref{0502-6}--\eqref{boundarydataspecific}, and let \(i\neq n\).
The implicit constants in \eqref{0804-2}--\eqref{0804-4} can be chosen, depending only on \(n\), \(a\) and \(M\), and there exists \(c_*>0\), depending on the same quantities, such that if \(|x'|>c_*(1+\sqrt{t})\), \(x_i\geq\frac12|x'|\) and \(t>1\), then for the sets \(A_a^l\), \(l\in\{1,2,3\}\), given in \eqref{0804-2}--\eqref{0804-4}, there is a pair \((x_{ni2}^*,x_{ni1}^*)\in S^{-+}(w_i(x',\cdot\,,t);0,\infty)\), where \(S^{-+}\) is defined in Definition \ref{1202}, satisfying the following: for \(k=1,2\),

\begin{itemize}
\item[(i)] If \( t \geq \frac98\), then \(x^*_{nik}\approx \left( t\ln\left( \frac{|x'|^2}{t}\right)\right)^{\frac12}\).

\item[(ii)] Let \( 1<t< \frac98\).
\begin{itemize}

\item[(a)] For \( -1 < a< -\frac34\),
\begin{align*}
x_{nik}^*\approx                \left(\frac{(a+1)|x'|^2}{(t-1)^{a+\frac12}}\right)^{\frac{1}{1-2a}}\mathbf{1}_{A_a^1}+\Big( \ln \frac{(a +1)|x'|^2}{(t -1)^{a +\frac12}}  \Big)^{\frac12}\mathbf{1}_{A_a^3}.
\end{align*}

\item[(b)] For \( -\frac34 \leq a< -\frac12\),
\begin{align*}
x_{nik}^*\approx &   \left(\frac{|a+\frac12||x'|^2}{(t-1)^{a+\frac12}}\right)^{\frac{1}{1-2a}}\mathbf{1}_{A_a^1}\\
&+\left( \left(\ln \frac{|x'|^2}{|\ln (t-1)|}\right)^{\frac12} \mathbf{1}_{t\geq 1+\frac14e^{\frac{1}{a+\frac12}}} +\left(\ln \frac{|a+\frac12||x'|^2}{(t-1)^{a+\frac12}}\right)^{\frac12}\mathbf{1}_{t<1+\frac14 e^{\frac{1}{a+\frac12}}}\right)\mathbf{1}_{A_a^3}.
\end{align*}

\item[(c)] For \( a =-\frac12\), \( x_{nik}^* \approx |x'|\left(\ln \frac{|x'|^2}{t-1} \right)^{-\frac12}\mathbf{1}_{A_a^2} +\Big(\ln \frac{|x'|^2 }{ |\ln(t-1)|} \Big)^\frac12\mathbf{1}_{A_a^3}\).

\item[(d)] For \(-\frac12 < a \leq -\frac14\),
\begin{align*}
x_{nik}^*& \approx \left(\left(\ln \frac{|x'|^2}{(t-1)^{a+\frac12}|\ln(t-1)|}\right)^{\frac12}\mathbf{1}_{t\geq1+\frac14 e^{-\frac{1}{a+\frac12}}}+\left(\ln \left(\left(a+\frac12\right)|x'|^2\right)\right)^{\frac12}\mathbf{1}_{t<1+\frac14 e^{-\frac{1}{a+\frac12}}}\right)\mathbf{1}_{A_a^3}.
\end{align*}

\item[(e)] For \(-\frac14< a\leq1\), \(x_{nik}^* \approx ( \ln |x'| )^{\frac12}\mathbf{1}_{A_a^3}.\)
\end{itemize}

\item[(iii)] The implicit constants in \eqref{0804-2}--\eqref{0804-4} can be chosen so that \(A_a^1\cup A_a^2\) and \(A_a^3\) are disjoint.
If \(1<t<\frac98\) and \((x',t)\notin A_a^1\cup A_a^2\cup A_a^3\), then \(-1<a\leq-\frac12\) and \(x_{nik}^*\approx1\).
\end{itemize}
Here \(x_{nik}^*\in (\sqrt{2(t-1)},\frac12)\) if \((x',t)\in A_a^1\cup A_a^2\), and \(x_{nik}^*\in (\frac12, |x'|)\) if \((x',t)\in A_a^3\).
\end{proposition}
\begin{remark}\label{1217}
\begin{itemize}
\item[(i)] The values \(-\frac34\) and \(-\frac14\) are only convenient partition points: \(a=-\frac34\) is assigned to (ii)(b), and \(a=-\frac14\) to (ii)(d) (Subcase 3-4 of the proof), where the asymptotics agree with those of (ii)(e), both being \(\approx(\ln|x'|)^{\frac12}\).
\item[(ii)] All constants in Proposition \ref{1stmainprop}, including \(c_*\), may depend on \(a\), and no uniformity as \(a\to-1\) or \(a\to-\frac12\) is asserted.
Hence the factors \(a+1\), \(|a+\frac12|\) and their reciprocals in \(U_1\), \(U_2\), \(f_a\), in Proposition \ref{1stmainprop} and in Lemmas \ref{elementarylemma}--\ref{elementarylemma3} can be absorbed into the constants; they are retained only to show where the formulas degenerate.
The same applies to the distinction at \(a=-1+\frac k2\) in Lemma \ref{elementarylemma}(iv) (see Appendix \ref{proofofelementarylemma}).
For fixed \(a\), all asymptotics in (ii) are \(\approx1\) near the common boundary of \(A_a^1\cup A_a^2\) and \(A_a^3\), in accordance with (iii).
\item[(iii)] Case 2 of the proof also yields, formally, zeros of \(w_i\) in \((\sqrt{2(t-1)},\frac12)\) when \(-\frac34\leq a<-\frac12\) and \(x_n\lesssim\sqrt{t-1}\), and when \(-\frac12<a<0\).
They require \(|x'|^2(\ln\frac{|x'|^2}{t-1})^{-1}\lesssim t-1\), respectively \(|x'|^2\lesssim t-1\) or \((a+\frac12)|x'|^2\lesssim1\), each of which bounds \(|x'|\) by a constant depending on \(a\); so they do not occur for \(|x'|>c_*\), and the corresponding sets are omitted from \eqref{0804-2}--\eqref{0804-3} (see Subcases 2-2 and 2-4).
\end{itemize}
\end{remark}
\begin{proof}
Before entering the case-by-case analysis, we explain the common structure of the proof.
In each subcase, \(w_i(x,t)\) is bounded from above and below by two explicit comparison functions, say \(g_1\) and \(g_2\).
These functions have the same algebraic structure and differ only in their constants.

Depending on the parameter regime for \((x,t)\), the sign of \(w_i\) is determined in one of two ways: either both comparison functions have a definite sign on the relevant interval of \(x_n\), which gives the sign of \(w_i\) there, or they vanish at points \(x_{ni2}^*\leq x_{ni1}^*\), which yields a sign change of \(w_i\) in \([x_{ni2}^*,x_{ni1}^*]\).

All subcases below follow this scheme; only the explicit expressions of the bounds differ.

Throughout the proof, the constants \(c_i=c_i(n,a,M), d_i=d_i(n,a,M)>0\), \(i=1,2,3,4\), may change from case to case.
In each subcase, the constants will satisfy \(c_1\leq d_1\), \(c_2\geq d_2\) and \(c_3\leq d_3\).

\noindent
{\bf \(\bullet\) (Case 1: \(1<t<\frac98\), \(x_n<\sqrt{2(t-1)}\))}\quad

By \eqref{est-wi}, Lemma \ref{elementarylemma}(ii), Lemma \ref{elementarylemma2}(ii)(a) and Lemma \ref{elementarylemma3}(ii), we have
\begin{align*}
\begin{split}
w_i(x,t)&\approx \frac{x_i x_n}{|x'|^{n+2}}\left\{ ( x_n^2- |x'|^2  ) \left(  \frac{(t-1)^{a-\frac12}-1}{\frac12-a}\mathbf{1}_{a\neq \frac12}+  |\ln(t-1)|\mathbf{1}_{a=\frac12}  +  \frac{(t-1)^{a-\frac12}}{a+1}\right)\right.\\
&\left.\quad+x_n\left(|x'|^{n}e^{-c|x'|^2}+\frac{|x'|^{n+2}(t-1)^{a-\frac{n}{2}-1}}{a+1}e^{-\frac{c|x'|^2}{t-1}}\right) + M\left(x_n^2+x_n|x'|^{n+2}e^{-c|x'|^2}-|x'|^2\right)\right\}\\
&\approx \frac{x_i x_n}{|x'|^{n+2}}\left\{ - |x'|^2  \left(  \frac{(t-1)^{a-\frac12}-1}{\frac12-a}\mathbf{1}_{a\neq \frac12}+  |\ln(t-1)|\mathbf{1}_{a=\frac12}  +  \frac{(t-1)^{a-\frac12}}{a+1}\right)\right.\\
&\left.\quad+x_n\left(|x'|^{n}e^{-c|x'|^2}+\frac{|x'|^{n+2}(t-1)^{a-\frac{n}{2}-1}}{a+1}e^{-\frac{c|x'|^2}{t-1}}\right) + M\left(x_n^2+x_n|x'|^{n+2}e^{-c|x'|^2}-|x'|^2\right)\right\}\\
&\lesssim\frac{x_i x_n}{|x'|^{n+2}}\left\{ - |x'|^2  \left(  \frac{(t-1)^{a-\frac12}-1}{\frac12-a}\mathbf{1}_{a\neq \frac12}+  |\ln(t-1)|\mathbf{1}_{a=\frac12}  +  \frac{(t-1)^{a-\frac12}}{a+1}\right)\right.\\
&\left. \quad + \left(1+M+\frac{1}{a+1}\right)|x'|^{n+2}e^{-c|x'|^2}  + M\right\}<0
\end{split}
\end{align*}
by taking \(|x'|\) sufficiently large.
Here, after factoring out \(\frac{x_ix_n}{|x'|^{n+2}}\), the terms containing \(x_n|x'|^{n}e^{-c|x'|^2}\) come from \(x_i\mathcal{K}_{a,\frac{n+2}{2}}\), and the term \(Mx_n|x'|^{n+2}e^{-c|x'|^2}\) comes from the second term of \eqref{est-wi}.
In the last step we used \(x_n<\frac12\) and the bound \((t-1)^{a-\frac n2-1}e^{-\frac{c|x'|^2}{t-1}}\leq C e^{-c|x'|^2}\) for \(0<t-1<\frac18\), after decreasing \(c\).

\noindent
{\bf \(\bullet\) (Case 2: \( 1 < t<\frac98\), \( \sqrt{2(t -1)} \leq x_n < \frac12\))}

{\bf \underline{Subcase 2-1}: \(-1 < a< -\frac34\).} \quad From Corollary \ref{1216}(i), Lemma \ref{elementarylemma}(iv) and Lemma \ref{elementarylemma2}(ii)(b), we have
\begin{align*}
\begin{split}
\frac{x_ix_n}{|x'|^{n+2}}\left(  c_1 \frac{(t-1)^{a+\frac12}}{a+1}  -c_2 |x'|^2 x_n^{2a-1} \right)  \leq w_i(x,t)\leq \frac{x_ix_n}{|x'|^{n+2}}\left(d_1 \frac{(t-1)^{a+\frac12}}{a+1}  -d_2|x'|^2 x_n^{2a-1} \right).
\end{split}
\end{align*}

If \(\frac{d_1}{d_2}\frac{2^{\frac12-a}}{a+1}(t-1)\leq |x'|^2\leq \frac{c_1}{c_2}\frac{2^{2a-1}}{a+1}(t-1)^{a+\frac12}\), then let us denote
\begin{align*}
&x_{ni1}^*:=\left(  \frac{c_1}{ c_2}\right)^{\frac1{2a-1}} |x'|^{\frac{2}{1-2a}} \Big( \frac{(t-1)^{a+\frac12}}{a+1}\Big)^{\frac1{2a -1}},\quad x_{ni2}^*: =\left( \frac{d_1}{ d_2}\right)^{\frac1{2a-1}} |x'|^{\frac{2}{1-2a}} \Big( \frac{ (t-1)^{a+\frac12}}{a+1}\Big)^{\frac1{2a -1}},
\end{align*}
so that \(x_{ni1}^*\geq x_{ni2}^*\) and \((x_{ni2}^*, x_{ni1}^*)\in S^{-+}(w_i(x',\cdot, t); \sqrt{2(t-1)},\frac12)\).

On the other hand, if \(|x'|^2\geq \frac{d_1}{d_2}\frac{2^{2a-1}}{a+1}(t-1)^{a+\frac12}\), then \(w_i(x,t)<0\) for all \(\sqrt{2(t-1)}\leq x_n\leq \frac12\), while if \(|x'|^2<\frac{c_1}{c_2}\frac{2^{\frac12-a}}{a+1}(t-1)\), then \(w_i(x,t)>0\) for all \(\sqrt{2(t-1)}\leq x_n<\frac12\).

{\bf \underline{Subcase 2-2}: \(-\frac34 \leq a < -\frac12\).} \quad From Corollary \ref{1216}(i), Lemma \ref{elementarylemma}(iv) and Lemma \ref{elementarylemma2}(ii)(b), we have
\begin{align}\label{0626-11}
\begin{split}
w_i(x,t)&\approx \frac{x_ix_n}{|x'|^{n+2}}\left(  \frac{x_n^{2a +1} - (t-1)^{a+\frac12}}{a+\frac12}  - |x'|^2 x_n^{2a-1} \right).
\end{split}
\end{align}

1. Suppose \(x_n\leq e^{-\frac{1}{2a+1}}\sqrt{t-1}\).
By \eqref{lemma0630-2}, we have \( \frac{x_n^{2a +1} - (t-1)^{a+\frac12}}{a+\frac12} \approx x_n^{2a +1} \ln \frac{x_n^2}{t-1}\).
From \eqref{0626-11}, we have
\begin{align*}
\frac{x_ix_n}{|x'|^{n+2}} \Big( c_1 x_n^{2a+1}\ln \frac{x_n^2}{t-1}-c_2|x'|^2x_n^{2a-1} \Big) \leq w_i(x,t) \leq  \frac{x_ix_n}{|x'|^{n+2}}\Big( d_1 x_n^{2a+1}\ln \frac{x_n^2}{t-1}-d_2|x'|^2 x_n^{2a-1}\Big).
\end{align*}

Then by Lemma \ref{lemma0630}, applied with \(b=1\), \(R=\frac{|x'|^2}{t-1}\) and \(\theta=\frac{x_n^2}{t-1}\), we have
\begin{align*}
w_i(x,t)>0 &  \text{ for } x_n\geq \sqrt{2}\max\left\{2,\left(\frac{c_2}{c_1}\right)^{\frac12}\right\}   |x'|\left(\ln\frac{|x'|^2}{t-1}\right)^{-\frac12}=:x_{ni1}^*,\\
w_i(x,t)<0 & \text{ for } x_n\leq \min\left\{2, \left(\frac{d_2}{5d_1}\right)^{\frac12}\right\}   |x'|\left(\ln\frac{|x'|^2}{t-1}\right)^{-\frac12}=: x_{ni2}^*.
\end{align*}
Note that \(x^*_{ni2} \leq x^*_{ni1} \).
If
\begin{align}\label{0716-1}
\frac{\sqrt{2(t-1)} }{\min\left\{2, \left(\frac{d_2}{5d_1}\right)^{\frac12}\right\}}  \leq |x'|\left(\ln \frac{|x'|^2}{t-1}\right)^{-\frac12}\leq \frac{\min\{e^{-\frac{1}{2a+1}}\sqrt{t-1},\frac12\}}{\sqrt{2}\max\left\{2,\left(\frac{c_2}{c_1}\right)^{\frac12}\right\} },
\end{align}
then \(\sqrt{2(t-1)}\leq x^*_{nik}\leq \min\{e^{-\frac{1}{2a+1}}\sqrt{t-1},\frac12\}\), so that
\begin{align*}
(x^*_{ni2}, x^*_{ni1}) \in S^{-+}\Bigl(w_i(x', \cdot\,, t); \sqrt{2(t-1)}, \min\bigl\{e^{-\frac{1}{2a+1}}\sqrt{t-1},\tfrac12\bigr\} \Bigr).
\end{align*}

On the other hand, if
\begin{align}\label{0930-01}
\min\{e^{-\frac{1}{2a+1}}\sqrt{t-1},\frac12\} \leq \min\left\{2,\left(\frac{d_2}{5d_1}\right)^{\frac12}\right\} |x'|\left(\ln \frac{|x'|^2}{t-1}\right)^{-\frac12},
\end{align}
then we have \( w_i (x,t) < 0\) for \( \sqrt{2(t-1)} \leq x_n\leq\min\{e^{-\frac{1}{2a+1}}\sqrt{t-1},\frac12\}\).

For \(|x'|>c_*\), with \(c_*\) large enough depending on \(a\), only this last alternative occurs.
Indeed, put \(y:=\frac{|x'|^2}{t-1}\), so that \(y>8c_*^2\) and \(|x'|^2\bigl(\ln\frac{|x'|^2}{t-1}\bigr)^{-1}=(t-1)\frac{y}{\ln y}\).
Since \(y\mapsto \frac{y}{\ln y}\) is increasing for \(y\geq e\), \eqref{0930-01} holds as soon as
\begin{align*}
\min\left\{4,\frac{d_2}{5d_1}\right\}\frac{8c_*^2}{\ln(8c_*^2)}\geq e^{-\frac{2}{2a+1}} .
\end{align*}
Consequently, for \(|x'|>c_*\), the function \(w_i(x',\cdot\,,t)\) has no zero in \([\sqrt{2(t-1)},\min\{e^{-\frac{1}{2a+1}}\sqrt{t-1},\frac12\}]\), and \eqref{0716-1} is never satisfied.

2. Suppose \(x_n\geq e^{-\frac{1}{2a+1}}\sqrt{t-1}\).
From \eqref{lemma0630-2}, we have \( \frac{x_n^{2a +1} - (t-1)^{a+\frac12}}{a+\frac12} \approx \frac{(t-1)^{a +\frac12} }{|a+\frac12|}\).
From \eqref{0626-11}, we have
\begin{align*}
\frac{x_ix_n}{|x'|^{n+2}}  \Big( \frac{e-1}{e}c_1\frac{(t-1)^{a+\frac12}}{|a+\frac12|}-c_2|x'|^{2}x_n^{2a-1} \Big) \leq w_i(x,t) \leq \frac{x_ix_n}{|x'|^{n+2}}  \Big(d_1\frac{(t-1)^{a+\frac12}}{|a+\frac12|}-d_2|x'|^2x_n^{2a-1} \Big).
\end{align*}
Then, we have
\begin{align*}
&w_i(x,t)>0 \text{ for } x_n\geq \left(\frac{e-1}{e}\frac{c_1}{c_2}\right)^{\frac{1}{2a-1}}\left|a+\frac12\right|^{\frac{1}{1-2a}}|x'|^{\frac{2}{1-2a}}(t-1)^{\frac{a+\frac12}{2a-1}}=:x_{ni1}^*,\\
&w_i(x,t)<0 \text{ for } x_n\leq \left(\frac{d_1}{d_2}\right)^{\frac{1}{2a-1}}\left|a+\frac12\right|^{\frac{1}{1-2a}}|x'|^{\frac{2}{1-2a}}(t-1)^{\frac{a+\frac12}{2a-1}}=:x_{ni2}^*.
\end{align*}
Note that \(x_{ni1}^*\geq x_{ni2}^*\).
If
\begin{align*}
\frac{d_1}{d_2} e^{-\frac{1-2a}{1+2a}} (t-1)\leq \left|a+\frac12\right||x'|^2\leq 2^{2a-1} \frac{e-1}{e}\frac{c_1}{c_2} (t-1)^{a+\frac12},
\end{align*}
then \( \min \{e^{-\frac{1}{2a+1}}\sqrt{t-1},\frac12 \}\leq x^*_{nik} \leq \frac12\), so that
\begin{align*}
(x^*_{ni2}, x^*_{ni1}) \in S^{-+}\Bigl(w_i(x', \cdot\,, t); \min\bigl\{e^{-\frac{1}{2a+1}}\sqrt{t-1},\tfrac12\bigr\},\tfrac12 \Bigr).
\end{align*}

On the other hand, if \( 2^{2a-1} \frac{d_1}{d_2} (t-1)^{a+\frac12} \leq \left|a+\frac12\right||x'|^2 \), then \(x_{ni2}^*\geq\frac12\), and hence \( w_i (x,t) < 0\) for \( \min\{e^{-\frac{1}{2a+1}}\sqrt{t-1},\frac12\} \leq x_n\leq \frac12\).

{\bf \underline{Subcase 2-3}: \(a=-\frac12\).} \quad From Corollary \ref{1216}(i), Lemma \ref{elementarylemma}(iv) and Lemma \ref{elementarylemma2}(ii)(b), we have
\begin{align*}
\frac{x_ix_n}{|x'|^{n+2}}x_n^{-2} \Big( c_1 x_n^2\ln \frac{x_n^2}{t-1}-c_2 |x'|^2  \Big) \leq w_i(x,t)\leq  \frac{x_ix_n}{|x'|^{n+2}}x_n^{-2} \Big( d_1 x_n^2 \ln \frac{x_n^2}{t-1}-d_2 |x'|^2  \Big).
\end{align*}

We now search for \(x_{nik}^*\).
By Lemma \ref{lemma0630}, applied as in Subcase 2-2 with \(b=1\), we have
\begin{align*}
x_{ni1}^*&:=\sqrt{2}\max\left\{2,\left(\frac{c_2}{c_1}\right)^{\frac12}\right\}|x'| \left(\ln\frac{|x'|^2}{t-1}\right)^{-\frac12},\\
x_{ni2}^*&:=\min\left\{2,\left(\frac{d_2}{5d_1}\right)^{\frac12}\right\}|x'| \left(\ln\frac{|x'|^2}{t-1}\right)^{-\frac12}.
\end{align*}
Then \(w_i(x,t)>0\) for \(x_n\geq x_{ni1}^*\) and \(w_i(x,t)<0\) for \(x_n\leq x_{ni2}^*\).
If
\begin{align*}
8\max\left\{4,\frac{c_2}{c_1}\right\}|x'|^2 \leq \ln\frac{|x'|^2}{t-1} \leq \frac12\min\left\{4,\frac{d_2}{5d_1}\right\}\frac{|x'|^2}{t-1},
\end{align*}
then \(\sqrt{2(t-1)}\leq x^*_{ni2}\leq x^*_{ni1}\leq \frac12\), and hence \((x_{ni2}^*, x_{ni1}^*)\in S^{-+}(w_i(x',\cdot,t);\sqrt{2(t-1)},\frac12)\).
If instead \(|x'|^2\geq \frac{d_1}{4d_2}|\ln(4(t-1))|\), then \(w_i(x,t)<0\) for \(\sqrt{2(t-1)}\leq x_n\leq\frac12\).

{\bf \underline{Subcase 2-4}: \(-\frac12 < a<0\).} \quad From Corollary \ref{1216}(i), Lemma \ref{elementarylemma}(iv) and Lemma \ref{elementarylemma2}(ii)(b), we have
\begin{align}\label{0626-12}
\begin{split}
w_i(x,t)&\approx \frac{x_ix_n}{|x'|^{n+2}}\left(  \frac{x_n^{2a +1}- (t-1)^{a+\frac12}}{a+\frac12}  - |x'|^2 x_n^{2a-1} \right).
\end{split}
\end{align}

1. Suppose that \(x_n<e^{\frac{1}{2a+1}}\sqrt{t-1}\).
Then by \eqref{lemma0630-2}, \(\frac{x_n^{2a+1}-(t-1)^{a+\frac12}}{a+\frac12}\approx (t-1)^{a+\frac12}\ln \frac{x_n^2}{t-1}\) and it follows from \eqref{0626-12} that
\begin{align*}
\frac{x_ix_n}{|x'|^{n+2}} \Big( c_1(t-1)^{a+\frac12}\ln \frac{x_n^2}{t-1}-c_2|x'|^2 x_n^{2a-1} \Big) \leq w_i(x,t) \leq \frac{x_ix_n}{|x'|^{n+2}} \Big(d_1 (t-1)^{a+\frac12}\ln \frac{x_n^2}{t-1}-d_2|x'|^2 x_n^{2a-1} \Big).
\end{align*}

Then by Lemma \ref{lemma0630}, applied with \(b=\frac{1-2a}{2}\), \(R=\frac{|x'|^2}{t-1}\) and \(\theta=\frac{x_n^2}{t-1}\), we have
\begin{align*}
&w_i(x,t)>0 \text{ for } x_n\geq 2^{\frac{1}{1-2a}}A_1|x'|^{\frac{2}{1-2a}}(t-1)^{-\frac{a+\frac12}{1-2a}}\left(\ln \frac{|x'|^2}{t-1}\right)^{-\frac{1}{1-2a}}=:x_{ni1}^*,\\
&w_i(x,t)<0 \text{ for } x_n\leq A_2|x'|^{\frac{2}{1-2a}}(t-1)^{-\frac{a+\frac12}{1-2a}}\left(\ln \frac{|x'|^2}{t-1}\right)^{-\frac{1}{1-2a}}=:x_{ni2}^*,
\end{align*}
where \(A_1:=\max\left\{2, \left(\frac{1-2a}{2}\frac{c_2}{c_1}\right)^{\frac{1}{1-2a}}\right\}\) and \(A_2:=\min\left\{2, \left(\frac{1-2a}{8-4a}\frac{d_2}{d_1}\right)^{\frac{1}{1-2a}}\right\}\).
Clearly we have \(A_1\geq A_2\).
If
\begin{align}\label{0716-3}
\frac{2^{\frac12-a}}{A_2^{1-2a}}(t-1)\leq |x'|^2\left(\ln\frac{|x'|^2}{t-1}\right)^{-1}\leq \frac{1}{2A_1^{1-2a}}\min\left\{e^{\frac{1-2a}{1+2a}}(t-1),2^{2a-1}(t-1)^{a+\frac12}\right\},
\end{align}
then \(\sqrt{2(t-1)}\leq x_{ni2}^*\leq x_{ni1}^*\leq \min\{e^{\frac{1}{2a+1}}\sqrt{t-1},\frac12\}\), so that
\begin{align*}
(x^*_{ni2}, x^*_{ni1}) \in S^{-+}\Bigl(w_i (x',\cdot\,,t); \sqrt{2(t-1)}, \min\bigl\{e^{\frac{1}{2a+1}}\sqrt{t-1},\tfrac12\bigr\}\Bigr).
\end{align*}

On the other hand, if \(|x'|^2\geq \frac{2e^{\frac{1-2a}{1+2a}}}{1+2a}\frac{d_1}{d_2}(t-1)\), then \(w_i(x,t)<0\) for \(\sqrt{2(t-1)}\leq x_n\leq e^{\frac{1}{2a+1}}\sqrt{t-1}\).
Since \(t-1<\frac18\), this condition holds for every \(|x'|>c_*\) once \(c_*^2\geq \frac{e^{\frac{1-2a}{1+2a}}}{4(1+2a)}\frac{d_1}{d_2}\); in particular, \eqref{0716-3} is never satisfied for \(|x'|>c_*\).

2. Suppose \( x_n>e^{\frac{1}{2a+1}}\sqrt{t-1} \).
Then by \eqref{lemma0630-2}, \(\frac{x_n^{2a+1}-(t-1)^{a+\frac12}}{a+\frac12}\approx \frac{x_n^{2a+1}}{a+\frac12}\) and we have from \eqref{0626-12} that
\begin{align*}
\frac{x_ix_n}{|x'|^{n+2}}\left({c}_1\frac{x_n^{2a+1}}{a+\frac12}-c_2|x'|^2 x_n^{2a-1}\right) \leq w_i (x,t) \leq  \frac{x_ix_n}{|x'|^{n+2}}\left(d_1\frac{x_n^{2a+1}}{a+\frac12}-d_2|x'|^2 x_n^{2a-1}\right).
\end{align*}
Let \( x_{ni1}^*:=\sqrt{\frac{c_2}{{c}_1}\left(a+\frac12\right)}|x'|,\quad x_{ni2}^*:=\sqrt{\frac{d_2}{d_1}\left(a+\frac12\right)}|x'|\) so that \(x_{ni1}^*\geq x_{ni2}^*\).
To ensure that \(x_{nik}^*\) (\(k=1,2\)) satisfy the condition \(e^{\frac{1}{2a+1}}\sqrt{t-1}<x_n<\frac12\), we require that
\begin{align*}
e^{\frac{1}{2a+1}}\sqrt{\frac{d_1}{d_2}}\sqrt{t-1}\leq \left|a+\frac12\right|^{\frac12}|x'|\leq \sqrt{\frac{{c}_1}{4c_2}}.
\end{align*}
Then it follows that \((x_{ni2}^*, x_{ni1}^*) \in S^{-+}\left(w_i(x',\cdot, t); \min\{e^{\frac{1}{2a+1}}\sqrt{t-1},\frac12\},\frac12\right)\).

On the other hand, if \(|x'|^2\geq \frac{2}{a+\frac12}\frac{d_1}{d_2}\), then \(w_i(x,t)<0\) for \(e^{\frac{1}{2a+1}}\sqrt{t-1}<x_n\leq \frac12\).
This condition holds for every \(|x'|>c_*\) once \(c_*^2\geq\frac{2}{a+\frac12}\frac{d_1}{d_2}\).
Thus, for \(-\frac12<a<0\) and \(|x'|>c_*\), we have \(w_i(x,t)<0\) for all \(\sqrt{2(t-1)}\leq x_n<\frac12\), and no set \(A_a^1\) or \(A_a^2\) is needed for \(a>-\frac12\).

{\bf \underline{Subcase 2-5}: \(a=0\).} \quad From Corollary \ref{1216}(i), Lemma \ref{elementarylemma}(iv) and Lemma \ref{elementarylemma2}(ii)(b), we have for sufficiently large \(|x'|\),
\begin{align*}
\begin{split}
w_i(x,t)&\approx \frac{x_i}{|x'|^{n+2}}(Mx_n+1)(x_n^2-|x'|^2)<0.
\end{split}
\end{align*}

{\bf \underline{Subcase 2-6}: \(a>0\).} \quad From Corollary \ref{1216}(ii), Lemma \ref{elementarylemma3}(ii) and Lemma \ref{elementarylemma2}(ii)(b), we have for sufficiently large \(|x'|\),
\begin{align*}
\begin{split}
w_i(x,t)&\lesssim x_ix_n\left[-\left(M+\frac{x_n^{2a-1}-1}{\frac12-a}\mathbf{1}_{0<a<1, a\neq \frac12}  +|\ln(2x_n^2)|\mathbf{1}_{a=\frac12} + \mathbf{1}_{a\geq 1}\right)\frac{1}{|x'|^n} + x_ne^{-|x'|^2}\right] <0.\\
\end{split}
\end{align*}

\noindent
{\bf \(\bullet\) (Case 3: \( 1 < t<\frac98\), \( \frac12<x_n<|x'|\))}

From Corollary \ref{1216}(iii), Lemma \ref{elementarylemma}(iii) and Lemma \ref{elementarylemma2}(ii)(c), we have
\begin{align}\label{0727-1}
w_i(x,t)&\approx\frac{x_ix_n}{|x'|^{n+2}}\left( 1+\frac{2^{-a-\frac12}-(t-1)^{a+\frac12}}{a+\frac12}\mathbf{1}_{a\neq -\frac12}+|\ln(2(t-1))|\mathbf{1}_{a=-\frac12}+\frac{(t-1)^{a+\frac12}}{a+1}-    |x'|^2 e^{- x_n^2}\right).
\end{align}

{\bf \underline{Subcase 3-1}: \( -1 < a< -\frac34\).}

From \eqref{0727-1}, we have
\begin{align*}
\frac{x_ix_n}{|x'|^{n+2}}\left(  c_1 \frac{(t-1)^{a+\frac12}}{a +1} - c_2 |x'|^2 e^{-c_3x_n^2}  \right)\leq w_i(x ,t)   \leq \frac{x_ix_n}{|x'|^{n+2}}\left(  d_1 \frac{(t-1)^{a+\frac12}}{a +1} - d_2 |x'|^2 e^{-d_3x_n^2}  \right).
\end{align*}
Let \(x_{ni1}^*:=\left(\frac{1}{c_3}\ln\left(\frac{(a+1)c_2}{c_1}\frac{|x'|^2}{(t-1)^{a+\frac12}}\right)\right)^{\frac12}\), \(x_{ni2}^*:=\left(\frac{1}{d_3}\ln\left(\frac{(a+1)d_2}{d_1}\frac{|x'|^2}{(t-1)^{a+\frac12}}\right)\right)^{\frac12}\) so that \(x_{ni1}^*\geq x_{ni2}^*\) and \( w_i (x,t)>0\) for \(x_n\geq x_{ni1}^*\), while \(w_i (x,t)<0\) for \(x_n\leq x_{ni2}^*\).
This implies that if \( \frac{d_1}{d_2}e^{\frac{d_3}{4}}<\frac{(a+1)|x'|^2}{(t-1)^{a+\frac12}}<\frac{c_1}{c_2}e^{c_3|x'|^2}\), then \((x_{ni2}^*,x_{ni1}^*)\in S^{-+}(w_i(x',\cdot, t); \frac12, |x'|)\).

On the other hand, if \(|x'|^2\leq \frac{c_1e^{\frac{c_3}{4}}}{c_2}\frac{(t-1)^{a+\frac12}}{a+1}\), then \(w_i(x,t)>0\) for all \(\frac12\leq x_n\leq |x'|\).

{\bf \underline{Subcase 3-2}: \( -\frac34 \leq a <-\frac12 \).} \quad From \eqref{0727-1}, we have
\begin{align}\label{0729-1}
\begin{split}
w_i(x ,t)   \approx \frac{x_ix_n}{|x'|^{n+2}}\left(   \frac{2^{-a-\frac12}-(t-1)^{a+\frac12}}{a+\frac12}-   |x'|^2 e^{-d_3x_n^2}  \right).
\end{split}
\end{align}

1. Suppose that \(t\geq 1+\frac14 e^{\frac{2}{2a+1}}\).
By \eqref{lemma0630-2}, we have \(\frac{2^{-a-\frac12}-(t-1)^{a+\frac12}}{a+\frac12} \approx |\ln(t-1)| \).
Thus \eqref{0729-1} gives
\begin{align*}
\frac{x_ix_n}{|x'|^{n+2}}\left(c_1 |\ln(t-1)|   -c_2|x'|^2e^{-c_3x_n^2}\right) \leq w_i(x,t) \leq \frac{x_ix_n}{|x'|^{n+2}}\left(d_1  |\ln(t-1)| -d_2|x'|^2e^{-d_3x_n^2}\right).
\end{align*}
Let \(x_{ni1}^* = \left(\frac{1}{c_3}\ln\left(\frac{c_2}{c_1}\frac{|x'|^2}{|\ln(t-1)|}\right)\right)^{\frac12}\) and \( x_{ni2}^* =\left(\frac{1}{d_3}\ln\left(\frac{d_2}{d_1}\frac{|x'|^2}{|\ln(t-1)|}\right)\right)^{\frac12}\) so that \(x_{ni1}^*\geq x_{ni2}^*\) and \( w_i (x,t)>0\) for \(x_n\geq x_{ni1}^*\), while \(w_i (x,t)<0\) for \(x_n\leq x_{ni2}^*\).
This implies that if \(\frac{d_1}{d_2}e^{\frac{d_3}{4}}\leq \frac{|x'|^2}{|\ln(t-1)|}\leq \frac{c_1}{c_2}e^{c_3|x'|^2}\), then \((x_{ni2}^*, x_{ni1}^*)\in S^{-+}(w_i(x',\cdot, t); \frac12, |x'|)\).

On the other hand, if \(\frac{|x'|^2}{|\ln(t-1)|}<\frac{c_1}{c_2}e^{\frac{c_3}{4}}\), then \(w_i(x,t)>0\) for all \(\frac12\leq x_n\leq|x'|\), whereas if \(\frac{|x'|^2}{|\ln(t-1)|}>\frac{d_1}{d_2}e^{d_3|x'|^2}\), then \(w_i(x,t)<0\) for all \(\frac12\leq x_n\leq|x'|\).

2. Suppose that \(t< 1+\frac14 e^{\frac{2}{2a+1}}\).
By \eqref{lemma0630-2}, we have \(\frac{2^{-a-\frac12}-(t-1)^{a+\frac12}}{a+\frac12} \approx \frac{(t-1)^{a+\frac12}}{|a+\frac12|} \).
Thus \eqref{0729-1} gives
\begin{align*}
\frac{x_ix_n}{|x'|^{n+2}}\left( c_1 \frac{e-1}{e}\frac{(t-1)^{a+\frac12}}{|a+\frac12|} -c_2|x'|^2e^{-c_3x_n^2}\right)\leq w_i(x,t) \leq \frac{x_ix_n}{|x'|^{n+2}}\left( d_1 \frac{(t-1)^{a+\frac12}}{|a+\frac12|}  -d_2|x'|^2e^{-d_3x_n^2}\right).
\end{align*}
Let \(x_{ni1}^*= \left(\frac{1}{c_3}\ln\left(\frac{e|a+\frac12|c_2}{(e-1)c_1}\frac{|x'|^2}{(t-1)^{a+\frac12}}\right)\right)^{\frac12},\quad x_{ni2}^*= \left(\frac{1}{d_3}\ln\left(\frac{|a+\frac12|d_2}{d_1}\frac{|x'|^2}{(t-1)^{a+\frac12}}\right)\right)^{\frac12}\) so that \(x_{ni1}^*\geq x_{ni2}^*\) and \( w_i (x,t)>0\) for \(x_n\geq x_{ni1}^*\), while \(w_i (x,t)<0\) for \(x_n\leq x_{ni2}^*\).
This implies that if \(\frac{d_1}{d_2}e^{\frac{d_3}{4}}\leq \bigl|a+\frac12\bigr| \frac{|x'|^2}{(t-1)^{a+\frac12}}\) and \(\frac{e|a+\frac12|}{e-1} \frac{|x'|^2}{(t-1)^{a+\frac12}}\leq \frac{c_1}{c_2}e^{c_3|x'|^2}\), then \((x_{ni2}^*, x_{ni1}^*)\in S^{-+}(w_i(x',\cdot, t); \frac12, |x'|)\).

On the other hand, if \( \frac{e|a+\frac12|}{e-1} \frac{|x'|^2}{(t-1)^{a+\frac12}}< \frac{c_1}{c_2}e^{\frac{c_3}{4}}\), then \(w_i(x,t) > 0\) for all \(\frac12<x_n<|x'|\), and if \(\bigl|a+\frac12\bigr| \frac{|x'|^2}{(t-1)^{a+\frac12}}>\frac{d_1}{d_2}e^{d_3|x'|^2}\), then \(w_i(x,t) < 0\) for all \(\frac12<x_n<|x'|\).

{\bf \underline{Subcase 3-3}: \( a= -\frac12\).} \quad From \eqref{0727-1}, we have
\begin{align*}
\frac{x_ix_n}{|x'|^{n+2}}\left(c_1 |\ln(t-1)| -  c_2|x'|^2 e^{-c_3x_n^2} \right) \leq w_i(x,t)   \leq \frac{x_ix_n}{|x'|^{n+2}}\left(d_1 |\ln(t-1)| -  d_2|x'|^2 e^{-d_3x_n^2} \right).
\end{align*}
Let \( x_{ni1}^*:=\left(\frac{1}{c_3}\ln\left(\frac{c_2}{c_1}\frac{|x'|^2}{|\ln(t-1)|}\right)\right)^{\frac12}, x_{ni2}^*:=\left(\frac{1}{d_3}\ln\left(\frac{d_2}{d_1}\frac{|x'|^2}{|\ln(t-1)|}\right)\right)^{\frac12}\), so that \(x_{ni1}^*\geq x_{ni2}^*\) and \( w_i (x,t)>0\) for \(x_n\geq x_{ni1}^*\), while \(w_i (x,t)<0\) for \(x_n\leq x_{ni2}^*\).
This implies that if \(\frac{d_1}{d_2}e^{\frac{d_3}{4}}\leq \frac{|x'|^2}{|\ln(t-1)|}\leq \frac{c_1}{c_2}e^{c_3|x'|^2}\), then \((x_{ni2}^*, x_{ni1}^*)\in S^{-+}(w_i(x',\cdot,t);\frac12, |x'|)\).

On the other hand, if \(|x'|^2\leq \frac{c_1e^{\frac{c_3}{4}}}{c_2}|\ln(t-1)|\), then \(w_i(x,t)>0\) for all \(\frac12\leq x_n\leq |x'|\).

{\bf \underline{Subcase 3-4}: \( -\frac12 < a \leq -\frac14\).} \quad Relation \eqref{0729-1} also holds in this range.
Indeed, since \(0<t-1<\frac18\), we have \((t-1)^{a+\frac12}<8^{-(a+\frac12)}<2^{-(a+\frac12)}\), so the quotient \(\frac{2^{-a-\frac12}-(t-1)^{a+\frac12}}{a+\frac12}\) is bounded below by a positive constant depending on \(a\), and the bounded terms \(1\) and \(\frac{(t-1)^{a+\frac12}}{a+1}\) in \eqref{0727-1} can be absorbed into it.
Hence, as in Subcase 3-2, if \(|x'|^2\leq \frac{c_1e^{\frac{c_3}{4}}}{c_2}\frac{2^{-a-\frac12}-(t-1)^{a+\frac12}}{a+\frac12}\), then \(w_i(x,t)>0\) for all \(\frac12\leq x_n\leq |x'|\), and if \(|x'|^2e^{-d_3|x'|^2}\geq \frac{d_1}{d_2}\frac{2^{-a-\frac12}-(t-1)^{a+\frac12}}{a+\frac12}\), then \(w_i(x,t)<0\) for all \(\frac12\leq x_n\leq |x'|\).

From \eqref{lemma0630-2}, we have \(\frac{2^{-a-\frac12}-(t-1)^{a+\frac12}}{a+\frac12}\approx (t-1)^{a+\frac12}|\ln(t-1)|\mathbf{1}_{t\geq 1+\frac14 e^{-\frac{1}{a+\frac12}}}+\frac{1}{|a+\frac12|}\mathbf{1}_{t< 1+\frac14 e^{-\frac{1}{a+\frac12}}}\).
Then \eqref{0729-1} gives
\begin{align*}
w_i(x,t)&\geq \frac{x_ix_n}{|x'|^{n+2}} \biggl[c_1\left((t-1)^{a+\frac12}|\ln(t-1)|\mathbf{1}_{t\geq 1+\frac14 e^{-\frac{1}{a+\frac12}}}+\frac{e-1}{e}\frac{1}{|a+\frac12|}\mathbf{1}_{t< 1+\frac14 e^{-\frac{1}{a+\frac12}}}\right)\\
&\qquad\qquad\qquad-c_2|x'|^2e^{-c_3x_n^2}\biggr],\\
w_i(x,t)&\leq \frac{x_ix_n}{|x'|^{n+2}}\biggl[(e-1)d_1\left((t-1)^{a+\frac12}|\ln(t-1)|\mathbf{1}_{t\geq 1+\frac14 e^{-\frac{1}{a+\frac12}}}+\frac{1}{e|a+\frac12|}\mathbf{1}_{t< 1+\frac14 e^{-\frac{1}{a+\frac12}}}\right)\\
&\qquad\qquad\qquad-d_2|x'|^2e^{-d_3x_n^2}\biggr].
\end{align*}
Let
\begin{align*}
x_{ni1}^*&=\left(\frac{1}{c_3}\ln\left(\frac{c_2}{c_1}\frac{|x'|^2}{(t-1)^{a+\frac12}|\ln(t-1)|}\right)\right)^{\frac12}\mathbf{1}_{t\geq1+\frac14 e^{-\frac{1}{a+\frac12}}}\\
&\qquad+\left(\frac{1}{c_3}\ln\left(\frac{e|a+\frac12|c_2}{(e-1)c_1}|x'|^2\right)\right)^{\frac12}\mathbf{1}_{t< 1+\frac14 e^{-\frac{1}{a+\frac12}}},\\
x_{ni2}^*&=\left(\frac{1}{d_3}\ln\left(\frac{d_2}{(e-1)d_1}\frac{|x'|^2}{(t-1)^{a+\frac12}|\ln(t-1)|}\right)\right)^{\frac12}\mathbf{1}_{t\geq 1+\frac14 e^{-\frac{1}{a+\frac12}}}\\
&\qquad+\left(\frac{1}{d_3}\ln\left(\frac{e|a+\frac12|d_2}{(e-1)d_1}|x'|^2\right)\right)^{\frac12}\mathbf{1}_{t< 1+\frac14 e^{-\frac{1}{a+\frac12}}}.
\end{align*}
We note that \(x_{ni1}^*\geq x_{ni2}^*\) and \( w_i (x,t)>0\) for \(x_n\geq x_{ni1}^*\), while \(w_i (x,t)<0\) for \(x_n\leq x_{ni2}^*\).
This implies that \((x_{ni2}^*, x_{ni1}^*)\in S^{-+}(w_i(x',\cdot,t);\frac12, |x'|)\) provided that \((e-1)\frac{d_1}{d_2}e^{\frac{d_3}{4}}\leq \frac{|x'|^2}{(t-1)^{a+\frac12}|\ln(t-1)|}\leq \frac{c_1}{c_2}e^{c_3|x'|^2}\) when \(t\geq 1+\frac14 e^{-\frac{1}{a+\frac12}}\), and \(\frac{d_1}{d_2}e^{\frac{d_3}{4}}\leq\frac{e|a+\frac12||x'|^2}{e-1}\leq \frac{c_1}{c_2}e^{c_3|x'|^2}\) when \(t< 1+\frac14 e^{-\frac{1}{a+\frac12}}\).

On the other hand, if
\begin{align*}
|x'|^2\leq \frac{c_1 e^{\frac{c_3}{4}}}{c_2} \left((t-1)^{a+\frac12}|\ln(t-1)|\mathbf{1}_{t\geq 1+\frac14 e^{-\frac{1}{a+\frac12}}} +\frac{e-1}{e}\frac{1}{|a+\frac12|}\mathbf{1}_{t< 1+\frac14 e^{-\frac{1}{a+\frac12}}}\right),
\end{align*}
then \(w_i(x,t)>0\) for \(\frac12<x_n<|x'|\).

{\bf \underline{Subcase 3-5}: \(a> -\frac14\).} \quad From \eqref{0727-1}, we have
\begin{align*}
\frac{x_ix_n}{|x'|^{n+2}}\left(   c_1- c_2 |x'|^2 e^{-c_3x_n^2}  \right) \leq w_i(x ,t) & \leq \frac{x_ix_n}{|x'|^{n+2}}\left(   d_1- d_2 |x'|^2 e^{-d_3x_n^2}  \right).
\end{align*}

Let \(x_{ni1}^*:=\left(\frac{1}{c_3}\ln\frac{c_2|x'|^2}{c_1}\right)^{\frac12}, x_{ni2}^*:=\left(\frac{1}{d_3}\ln\frac{d_2|x'|^2}{d_1}\right)^{\frac12}\) so that \(x_{ni1}^*\geq x_{ni2}^*\) and \( w_i (x,t)>0\) for \(x_n\geq x_{ni1}^*\), while \(w_i (x,t)<0\) for \(x_n\leq x_{ni2}^*\).
This implies that if \(e^{\frac{d_3}{4}}\frac{d_1}{d_2}\leq |x'|^2\leq \frac{c_1}{c_2}e^{c_3|x'|^2}\), then \((x_{ni2}^*, x_{ni1}^*)\in S^{-+}(w_i(x',\cdot,t);\frac12, |x'|)\).

On the other hand, if \(|x'|^2e^{-\frac{c_3}{4}}<\frac{c_1}{c_2}\), then \(w_i(x,t)>0\) for \(\frac12<x_n<|x'|\).

\noindent
{\bf \(\bullet\) (Case 4: \( 1 < t<\frac98\), \( x_n>|x'|\))}

From \eqref{est-wi}, Lemma \ref{elementarylemma}(iii) and Lemma \ref{elementarylemma2}(ii)(c), we have
\begin{align*}
w_i(x,t)\geq \frac{x_i x_n}{x_n^{n+2}|x'|^n}\left(c_1|x'|^n-c_2x_n^{n+2}e^{-c_3x_n^2}\right)>0,
\end{align*}
by taking \(|x'|\) sufficiently large depending on \(n\) and \(M\).

{\bf \(\bullet\) (Case 5: \(t\geq\frac98\), \(x_n<\sqrt{t}\) or \(x_n>|x'|\))} From \eqref{est-wi}, Lemma \ref{elementarylemma}(i) and Lemma \ref{elementarylemma2}(i), we find that for \(t\geq\frac98\), if \(x_n<\sqrt{t}\), then
\begin{align*}
w_i(x,t)&\lesssim (M+\frac{1}{a+1}) \frac{x_nx_i}{t^{\frac32}}\left(\frac{t}{|x|^{n+2}} + \frac{e^{-\frac{c|x|^2}{t}}}{t^{\frac{n}{2}}}- \frac{1}{|x'|^n}\right)\lesssim -(M+\frac{1}{a+1}) \frac{x_ix_n}{t^{\frac32}|x'|^n}<0
\end{align*}
for sufficiently large \(|x'|\), and if \(x_n>|x'|\), then \(w_i(x,t)\geq \frac{x_ix_n}{x_n^{n+2}t^{\frac32}}\left(c_1 t-c_2\frac{x_n^{n+2}}{|x'|^n}e^{-\frac{c_3x_n^2}{t}}\right)>0\) for sufficiently large \(|x'|\).

{\bf \(\bullet\) (Case 6: \(t\geq\frac98\), \( \sqrt{t}<x_n<|x'|\))} From \eqref{est-wi}, Lemma \ref{elementarylemma}(i) and Lemma \ref{elementarylemma2}(i), we have
\begin{align*}
w_i(x,t)&\geq \left(M+\frac{1}{a+1}\right)\frac{x_nx_i}{t^{\frac12}|x'|^{n+2}} \left(c_1-c_2\frac{|x'|^2}{t}e^{-\frac{c_3x_n^2}{t}}\right),\\
w_i(x,t)&\leq \left(M+\frac{1}{a+1}\right)\frac{x_nx_i}{t^{\frac12}|x'|^{n+2}} \left(d_1-d_2\frac{|x'|^2}{t}e^{-\frac{d_3x_n^2}{t}}\right).
\end{align*}
By weakening the constants if necessary, we choose them so that \(c_1<d_1\), \(c_2>d_2\) and \(c_3<d_3\).
Let \( x_{ni1}^*:=\sqrt{\frac{t}{c_3}\ln\left(\frac{c_2}{c_1}\frac{|x'|^2}{t}\right)}\), \(x_{ni2}^*:=\sqrt{\frac{t}{d_3}\ln\left(\frac{d_2}{d_1}\frac{|x'|^2}{t}\right)}\) so that \(x_{ni1}^*\geq x_{ni2}^*\) and \(w_i(x,t)>0\) for \(x_n\geq x_{ni1}^*\), while \(w_i(x,t)<0\) for \(x_n\leq x_{ni2}^*\).
This implies that if \(\frac{d_1}{d_2}e^{d_3}t< |x'|^2< \frac{c_1}{c_2}te^{\frac{c_3|x'|^2}{t}}\), then \((x_{ni2}^*, x_{ni1}^*)\in S^{-+}(w_i(x',\cdot,t);\sqrt{t},|x'|)\).

It remains to justify that the local sign information obtained in the preceding cases is compatible across the interfaces in \(x_n\).
We use a single pair of comparison functions on the whole interval \(0<x_n\leq |x'|\).
Put
\begin{align*}
P(x_n,t):=\int_0^1\frac{\min\{1,x_n^2/(t-s)\}}{\sqrt{t-s}}\phi(s)\,ds, \qquad Q(x_n,t):=\int_0^1\frac{e^{-x_n^2/(4(t-s))}}{(t-s)^{3/2}}\phi(s)\,ds.
\end{align*}
Then \(P(\cdot,t)\) is nondecreasing and \(Q(\cdot,t)\) is nonincreasing.
Moreover, using the factorization
\begin{align*}
B_{in}(x,u)=\partial_{x_n}\Gamma_1(x_n,u) \bigl(\Gamma'(\cdot,u)*\partial_iN(\cdot,0)\bigr)(x')
\end{align*}
together with Lemma~\ref{lemma0709-1}(iii), and using Lemma~\ref{estimateofkernelt<1} for the \(L_{ni}\)-term, we obtain, after enlarging \(c_*\),
\begin{align}\label{global-monotone-pair}
c_1\frac{x_ix_n}{|x'|^{n+2}}\bigl(P-K|x'|^2Q\bigr) &\leq w_i(x,t)\leq C_1\frac{x_ix_n}{|x'|^{n+2}}\bigl(P+E-\kappa|x'|^2Q\bigr),
\end{align}
where
\begin{align*}
E=E(x',t):= C\|\phi\|_{L^1(0,1)}|x'|^{-1}e^{-c|x'|^2/t}
\end{align*}
is independent of \(x_n\).
Here \(0<c_1\leq C_1\) and \(0<\kappa\leq K\) depend only on \(n\), \(a\) and \(M\). Moreover, if \(c_2\leq C_2\) are the constants in the two-sided bound \(w_i^B\approx-\frac{x_ix_n}{|x'|^{n}}Q\), then \(\kappa=c_2/C_1\) and \(K=C_2/c_1\).
Indeed, the exponentially decaying term in \(L_{ni}\) is bounded by the displayed \(E\) because
\begin{align*}
\min\{1,x_n^2/u\}\leq x_nu^{-1/2},\qquad u^{-(n+3)/2}e^{-c|x'|^2/u}\leq C|x'|^{-n-3}e^{-c'|x'|^2/t}.
\end{align*}
Both brackets in \eqref{global-monotone-pair} are therefore continuous, nondecreasing functions of \(x_n\).
Since \(E\geq0\) and \(\kappa\leq K\), the upper bracket dominates the lower one pointwise.
Both brackets change sign on \((0,|x'|]\).
Indeed, since \(t>1\), we have \(P(x_n,t)\leq x_n^2\int_0^1(t-s)^{-\frac32}\phi(s)\,ds\to0\) as \(x_n\to0\), while \(Q(x_n,t)\to\int_0^1(t-s)^{-\frac32}\phi(s)\,ds\geq t^{-\frac32}\norm{\phi}_{L^1(0,1)}\); hence, as \(x_n\to0\), the upper bracket in \eqref{global-monotone-pair} tends to a limit which is at most
\begin{align*}
C\norm{\phi}_{L^1(0,1)}|x'|^{-1}e^{-\frac{c|x'|^2}{t}}-\kappa|x'|^2t^{-\frac32}\norm{\phi}_{L^1(0,1)}<0,
\end{align*}
because \(t\leq c_*^{-2}|x'|^2\) and \(c_*\) may be enlarged.
On the other hand, at \(x_n=|x'|\) we have \(\min\{1,|x'|^2/(t-s)\}=1\) for \(0<s<1\), since \(|x'|^2\geq t\), and therefore \(P(|x'|,t)\geq t^{-\frac12}\norm{\phi}_{L^1(0,1)}\).
Moreover, \(u\mapsto u^{-\frac32}e^{-|x'|^2/(4u)}\) is increasing on \((0,\frac16|x'|^2)\supset(0,t]\), so that \(Q(|x'|,t)\leq t^{-\frac32}e^{-\frac{|x'|^2}{4t}}\norm{\phi}_{L^1(0,1)}\).
Hence the lower bracket in \eqref{global-monotone-pair} at \(x_n=|x'|\) is at least
\begin{align*}
t^{-\frac12}\norm{\phi}_{L^1(0,1)}\Bigl(1-K\frac{|x'|^2}{t}e^{-\frac{|x'|^2}{4t}}\Bigr)>0
\end{align*}
after enlarging \(c_*\).
Let \(z_+\) be the supremum of the points \(x_n\in(0,|x'|]\) at which the upper bracket is negative, and let \(z_-\) be the infimum of the points at which the lower bracket is positive.
By monotonicity and continuity, \(0<z_+\leq z_-<|x'|\), the upper bracket is negative on \((0,z_+)\) and the lower bracket is positive on \((z_-,|x'|]\).
Consequently
\begin{align*}
w_i(x',x_n,t)<0\quad(0<x_n<z_+),\qquad w_i(x',x_n,t)>0\quad(z_-<x_n\leq|x'|).
\end{align*}

It remains to compare \(z_\pm\) with the thresholds found in the case analysis.
Fix a subcase in which a pair \((x_{ni2}^*,x_{ni1}^*)\) was found, and let \(I\) be the corresponding range of \(x_n\).
On \(I\), Lemmas~\ref{elementarylemma}--\ref{elementarylemma3} give explicit functions \(p_\pm\), \(q_\pm\), differing only in their constants, such that
\begin{align*}
p_-\leq P\leq p_+,\qquad q_-\leq Q\leq q_+ .
\end{align*}
Hence, on \(I\), we have
\begin{align*}
P+E-\kappa|x'|^2Q&\leq p_++E-\kappa|x'|^2q_-=:G_2,\\
P-K|x'|^2Q&\geq p_--K|x'|^2q_+=:G_1 .
\end{align*}
Up to the factor \(\frac{x_ix_n}{|x'|^{n+2}}\), the functions \(G_1\) and \(G_2\) have the form of the comparison functions \(g_1\) and \(g_2\) of that subcase, with the term \(E\) absorbed in the same way.
The computation of the subcase uses only this form and the ordering \(c_1\leq d_1\), \(c_2\geq d_2\), \(c_3\leq d_3\); applied to \(G_1\) and \(G_2\), it gives
\begin{align*}
G_2(x_{ni2}^*)<0,\qquad G_1(x_{ni1}^*)>0,
\end{align*}
with thresholds \(x_{ni2}^*\leq x_{ni1}^*\) in \(I\) satisfying the asymptotics found in the subcase.
Since both brackets in \eqref{global-monotone-pair} are nondecreasing in \(x_n\), we conclude that
\begin{align*}
x_{ni2}^*\leq z_+\leq z_-\leq x_{ni1}^* .
\end{align*}
Thus \(z_\pm\) lie in \(I\) and satisfy the same asymptotics, and we may take
\begin{align*}
x_{ni2}^*:=z_+,\qquad x_{ni1}^*:=z_- .
\end{align*}
This gives (i) and (ii), including the asserted intervals for \(x_{ni1}^*\) and \(x_{ni2}^*\) at the end of the proposition.

Finally, Cases 4 and 5 give \(w_i(x',x_n,t)>0\) for \(x_n>|x'|\), so that
\begin{align*}
(x_{ni2}^*,x_{ni1}^*)\in S^{-+}(w_i(x',\cdot\,,t);0,\infty).
\end{align*}

{\bf \(\bullet\) (Transition region: proof of (iii))}\quad We fix the implicit constants in \eqref{0804-2}--\eqref{0804-4} as follows.
Choose the lower constant in the definition of \(A_a^3\) to be the constant in the conditions of Case 3 under which a pair \((x_{ni2}^*,x_{ni1}^*)\) exists in \((\frac12,|x'|)\).
Choose the upper constants in the definitions of \(A_a^1\) and \(A_a^2\) to be those in the conditions of Subcases 2-1, 2-2 and 2-3 under which such a pair exists in \((\sqrt{2(t-1)},\frac12)\), decreasing them if necessary so that \(A_a^1\cup A_a^2\) and \(A_a^3\) are disjoint.
This is possible because, for \(|x'|>c_*\), both families of conditions are expressed through the same ratio \(\rho\) introduced below (for \(a=-\frac12\), up to a factor \(\approx1\)); and decreasing an upper constant only shrinks the set on which a sufficient condition is required, so the conclusions of those subcases remain valid.

If \(a>-\frac12\), then \(f_a\) is bounded above and below on \((1,\frac98)\) by positive constants depending on \(a\), so every \((x',t)\) with \(|x'|>c_*\) and \(1<t<\frac98\) lies in \(A_a^3\), after enlarging \(c_*\).
The same holds if \(-\frac34\leq a<-\frac12\) and \(t\geq1+\frac14e^{\frac{2}{2a+1}}\), since then \(f_a(t)=|\ln(t-1)|\) is bounded.

Let now \(-1<a\leq-\frac12\) and \(1<t<\frac98\), with \(t<1+\frac14e^{\frac{2}{2a+1}}\) if \(-\frac34\leq a<-\frac12\), and suppose that \((x',t)\notin A_a^1\cup A_a^2\cup A_a^3\).
Put
\begin{align*}
F_a(t):=\begin{cases} \frac{(t-1)^{a+\frac12}}{a+1}, &\text{if } -1<a<-\frac34,\\[1mm]
\frac{(t-1)^{a+\frac12}}{|a+\frac12|}, &\text{if } -\frac34\leq a<-\frac12,\\[1mm]
|\ln(t-1)|, &\text{if } a=-\frac12 .
\end{cases}
\end{align*}
By the definition of the sets \(A_a^l\) (for \(a=-\frac12\) note that \(\ln\frac{|x'|^2}{t-1}\approx|\ln(t-1)|\) when \(|x'|^2\lesssim|\ln(t-1)|\) and \(|x'|\) is large), there are \(0<\nu_-<\nu_+\), depending only on \(n\), \(a\) and \(M\), such that
\begin{align*}
\nu_-<\rho:=\frac{|x'|^2}{F_a(t)}<\nu_+ .
\end{align*}
The upper bounds in Subcase 2-1, in part 2 of Subcase 2-2 and in Subcase 2-3 can be written as
\begin{align*}
w_i(x,t)\leq\frac{x_ix_n}{|x'|^{n+2}}F_a(t)\,x_n^{2a-1}\bigl(d_1x_n^{1-2a}-d_2\rho\bigr),
\end{align*}
valid for \(\sqrt{2(t-1)}\leq x_n<\frac12\), and in addition \(x_n\geq e^{-\frac{1}{2a+1}}\sqrt{t-1}\) if \(-\frac34\leq a<-\frac12\); for \(a=-\frac12\) we have used \(\ln\frac{x_n^2}{t-1}\leq|\ln(t-1)|\) for \(x_n<1\).
Likewise, the lower bounds in Subcase 3-1, in part 2 of Subcase 3-2 and in Subcase 3-3 can be written, after decreasing \(c_1\), as
\begin{align*}
w_i(x,t)\geq\frac{x_ix_n}{|x'|^{n+2}}F_a(t)\bigl(c_1-c_2\rho\, e^{-c_3x_n^2}\bigr),\qquad \tfrac12<x_n<|x'| .
\end{align*}
Let
\begin{align*}
\Lambda_-:=\min\Bigl\{\frac14,\Bigl(\frac{d_2\nu_-}{2d_1}\Bigr)^{\frac{1}{1-2a}}\Bigr\},\qquad \Lambda_+:=\max\Bigl\{1,\Bigl(\frac{1}{c_3}\ln\frac{2c_2\nu_+}{c_1}\Bigr)^{\frac12}\Bigr\}.
\end{align*}
Then \(d_1x_n^{1-2a}-d_2\rho<0\) for \(x_n\leq\Lambda_-\), and \(c_1-c_2\rho e^{-c_3x_n^2}\geq\frac{c_1}{2}\) for \(x_n\geq\Lambda_+\).
Together with Case 1, with part 1 of Subcase 2-2 (which gives \(w_i<0\) for \(\sqrt{2(t-1)}\leq x_n\leq e^{-\frac{1}{2a+1}}\sqrt{t-1}\) when \(|x'|>c_*\)), with Case 4, and with the positivity of the lower bracket in \eqref{global-monotone-pair} at \(x_n=|x'|\), this shows that \(w_i(x',x_n,t)<0\) for \(0<x_n\leq\Lambda_-\) and \(w_i(x',x_n,t)>0\) for \(x_n\geq\Lambda_+\), where \(\Lambda_+<|x'|\) after enlarging \(c_*\).
Hence \((\Lambda_-,\Lambda_+)\in S^{-+}(w_i(x',\cdot\,,t);0,\infty)\), and we may take \(x_{ni2}^*:=\Lambda_-\) and \(x_{ni1}^*:=\Lambda_+\), both of which are \(\approx1\).
This proves (iii) and completes the proof.
\end{proof}

\subsection{Normal component of the velocity field}\label{sect5}

In this subsection, we treat the normal component of the velocity field and show that it changes sign when \(t>1\).

The following lemma provides a pointwise estimate of \(w_n\).
\begin{lemma}\label{w_nformula}
There exists \(R_1\geq 1\), depending only on \(n\), such that the following holds.
Let \(t>1\) and \(|x'|\geq R_1\sqrt{t}\).
Then there exist \(\lambda_i=\lambda_i(n)>0\) (\(i=1,2,3,4\)) such that
\begin{align}\label{est-wn}
\begin{split}
w_n(x,t)\leq &\frac{\lambda_1x_n^2-\lambda_2|x'|^2}{|x|^{n+2}}\norm{\psi}_{L^1(\Rn)}\int_0^1 \frac{1}{\sqrt{t-s}}\min\left\{1,\frac{x_n^2}{t-s}\right\} \phi(s)ds,\\
w_n(x,t)\geq &\frac{\lambda_3x_n^2-\lambda_4|x'|^2}{|x|^{n+2}}\norm{\psi}_{L^1(\Rn)}\int_0^1 \frac{1}{\sqrt{t-s}}\min\left\{1,\frac{x_n^2}{t-s}\right\} \phi(s)ds.
\end{split}
\end{align}
\end{lemma}
\begin{proof}
Since \(t>1\), \(w_n^N(x,t)=0\) by \(\eqref{wis}_3\).
Thus by \eqref{1220-8w_n}, \(\eqref{wis}_1\) and \eqref{240707},
\begin{align*}
w_n(x,t)& = - \sum_{k=1}^{n-1}w_{kk}^L(x,t)\approx\frac{x_n^2-|x'|^2}{|x|^{n+2}}\norm{\psi}_{L^1}\int_0^1\frac{1}{\sqrt{t-s}}\min\left\{1,\frac{x_n^2}{t-s}\right\} \phi(s)ds.
\end{align*}
\end{proof}

\begin{proposition}\label{2ndmainprop}
Let \(w\) be the solution of the Stokes system \eqref{StokesRn+} defined by \eqref{rep-bvp-stokes-w} with the boundary data \(g=g_n\mathbf{e}_n\) given by \eqref{0502-6}--\eqref{boundarydataspecific}.
There are \(c_*>0\) and \(0<\mu_1\leq\mu_2\), depending on \(n,a\) and \(M\), such that if \( |x'|>c_*(1+\sqrt t)\) and \(t>1\), then \(w_n(x',x_n,t)<0\) for \(0<x_n<\mu_1|x'|\) and \(w_n(x',x_n,t)>0\) for \(x_n>\mu_2|x'|\).
In particular, \(w_n(x',\cdot\,,t)\) has a zero, and every zero \(x_{nn}^*\) satisfies \(\mu_1|x'|\leq x_{nn}^*\leq\mu_2|x'|\).
\end{proposition}

\begin{proof}
Since the integral in \eqref{est-wn} is positive for \(x_n>0\), the upper bound there is negative when \(\lambda_1x_n^2<\lambda_2|x'|^2\), and the lower bound is positive when \(\lambda_3x_n^2>\lambda_4|x'|^2\).
Letting \(x_n\to\infty\) and \(x_n\to0\) in \eqref{est-wn} shows \(\lambda_3\leq\lambda_1\) and \(\lambda_2\leq\lambda_4\), so \(\mu_1:=(\lambda_2/\lambda_1)^{\frac12}\leq(\lambda_4/\lambda_3)^{\frac12}=:\mu_2\).
The claim follows, and Remark \ref{1125} gives the existence of a zero.
\end{proof}

\section{Signs of the velocity components and location of their zeros for \texorpdfstring{\(t<1\)}{t<1}}\label{sect4}

In this section, we determine the signs of \(w_i(x,t)\) (\(1\leq i\leq n-1\)) and of \(w_n(x,t)\) for \(t<1\) and locate their zeros.
The main results are Propositions \ref{3rdmainprop} and \ref{0716}.

\subsection{Tangential components of the velocity field }\label{subsect4.1}
\subsubsection{{\bf Estimates of \texorpdfstring{\(w_i\)}{w\_i}}}
We recall from \eqref{1220-8} that \(w_i(x,t)=w^{L}_{ni}(x,t)+w_i^B(x,t)+w_i^N(x,t)\).
We first estimate \(w^{L}_{ni}\) (Lemma \ref{lemma1003-2}) and \(w_i^B+w_i^N\) (Lemma \ref{thoe1021-1}).
We begin with an auxiliary lemma.

\begin{lemma}\label{lemma0928}
Let \( \frac78 < t < 1\), \(\alpha\in\mathbb{R}\) and \(\theta>0\). Define
\begin{align*}
A(\alpha, \theta, t): =\int_{1-t}^{\frac12}  s^{\alpha}   e^{-\frac{\theta^2}{4s}} ds,  \quad B(\alpha, \theta, t) : =  \int_0^{1-t}  s^{\alpha }e^{-\frac{\theta^2}{4s}} ds.
\end{align*}
Then \(A(\alpha, \theta, t)\) and \(B(\alpha, \theta, t)\) satisfy the following bounds.
\begin{itemize}
\item[(i)]

If \( \theta^2 \leq2( 1-t)\), then
\begin{align*}
A(\alpha, \theta, t) &   \approx \frac{2^{-\alpha-1}-(1-t)^{\alpha+1}}{\alpha+1}\mathbf{1}_{\al \neq -1} + |\ln(2(1-t))|\mathbf{1}_{ \al =-1}, \\
B(\alpha, \theta, t)  &  \approx \left(\frac{(1-t)^{\alpha+1}-2^{-\alpha-1}\theta^{2\alpha+2}}{\alpha+1}+  \theta^{2\alpha +2} \right)\mathbf{1}_{\alpha\neq -1}+\left(1+\ln\frac{2(1-t)}{\theta^2}\right)\mathbf{1}_{\alpha=-1}.
\end{align*}

\item[(ii)] If \( 2(1-t) \leq \theta^2 \leq \frac14\), then
\begin{align*}
A(\alpha, \theta, t) & \approx \left(\frac{2^{-\alpha-1}-(4\theta^2/3)^{\alpha+1}}{\alpha+1}+\theta^{2\alpha+2}\right)\mathbf{1}_{\alpha\neq -1}+\ln\left(\frac{3}{8\theta^2}\right)\mathbf{1}_{\alpha=-1}, \quad B(\alpha, \theta, t) \approx   \theta^{-2 }  ( 1-t  )^{2 +\alpha} e^{-\frac{c\theta^2}{1-t}}.
\end{align*}

\item[(iii)] If \(\theta\geq \frac12 \), then
\begin{align*}
A(\alpha, \theta, t) & \approx \theta^{-2}   e^{-c\theta^2},\quad B(\alpha, \theta, t) \approx \theta^{-2}  (1-t)^{2 +\alpha}   e^{- \frac{c\theta^2}{1-t}  }.
\end{align*}

\end{itemize}
\end{lemma}
The proof is given in Appendix \ref{proofoflemma0928}.
\begin{remark}\label{rmk-0928}
Lemma \ref{lemma0928} remains valid, uniformly in \(b\in[\frac38,\frac12]\), when the upper limit \(\frac12\) in the definition of \(A(\alpha,\theta,t)\) is replaced by \(b\); in fact, item (iii) holds for every fixed \(b>\frac18\).
We use this with \(b=t-\frac12\in(\frac38,\frac12)\) in the proof of Lemma \ref{thoe1021-1}.
\end{remark}

The next lemma gives pointwise estimates for \(w^{L}_{ni}\).

\begin{lemma}[{{\bf  Estimates of \(w^{L}_{ni}\)}}]\label{lemma1003-2}
Let \( i \neq n\), let \(\frac78 < t<1\), and let \(|x'|\) be sufficiently large with \(x_i\geq\frac12|x'|\).
Let \(w^{L}_{ni}\) be given by \(\eqref{wis}_1\).
Then the following holds.
\begin{itemize}
\item[(i)] If \(x_n<(2(1-t))^{\frac12}\), then
\begin{align*}
w^{L}_{ni}(x,t)\approx\frac{x_ix_n}{|x|^{n+2}}\left(x_n|x'|^{n+2}e^{-c|x'|^2}+x_n(1-t)^a +x_n^2\left(\frac{1-(1-t)^{a-\frac12}}{a-\frac12}\mathbf{1}_{a\neq \frac12} +|\ln(1-t)|\mathbf{1}_{a=\frac12}\right)\right).
\end{align*}

\item[(ii)] If \((2(1-t))^{\frac12}\leq x_n\leq(t-\tfrac12)^{\frac12}\), then
\begin{align*}
\begin{split}
w^{L}_{ni}(x,t)\approx\frac{x_ix_n}{|x|^{n+2}}\Bigg(&x_n^2\frac{1-x_n^{2a-1}}{a-\frac12}\mathbf{1}_{a\neq \frac12} +x_n^2|\ln x_n|\mathbf{1}_{a=\frac12}\\
&+\frac{x_n^{2a+1}-(1-t)^{a+\frac12}}{a+\frac12}\mathbf{1}_{a\neq -\frac12} +\ln\frac{x_n^2}{1-t}\mathbf{1}_{a=-\frac12} +x_n|x'|^{n+2}e^{-c|x'|^2}\Bigg).
\end{split}
\end{align*}

\item[(iii)] If \(x_n>(t-\tfrac12)^{\frac12}\), then
\begin{align*}
w^{L}_{ni}(x,t)\approx\frac{x_ix_n}{|x|^{n+2}} \left(\frac{1-(1-t)^{a+\frac12}}{a+\frac12}\mathbf{1}_{a\neq -\frac12}+|\ln(1-t)|\mathbf{1}_{a=-\frac12}\right).
\end{align*}
\end{itemize}
\end{lemma}
\begin{proof}
By \(\eqref{wis}_1\) and \eqref{L^Lestimate}, applied at the shifted points \(x-y'\) (see Remark \ref{rmk-shift}),
\begin{align}\label{0527-3}
\begin{split}
w^{L}_{n i}(x,t) &= 4\int_0^t \int_{\R^{n-1}}{L}_{ni}(x-y',t-s)\psi(y')\phi(s)dy'ds \\
&\approx \norm{\psi}_{L^1(\R^{n-1})}\int_0^t \left(\frac{x_n}{|x|^{n+2}}+\frac{e^{-\frac{c|x|^2}{t-s}}}{(t-s)^{\frac{n+1}{2}}}\right)\frac{x_i}{(t-s)^{\frac12}}\min\left\{1,\frac{x_n^2}{t-s}\right\}\phi(s)\, ds,
\end{split}
\end{align}
where we have used that \(|x'-y'|\approx |x'|\) and \(|x-y'|\approx|x|\) for \(|x'|\geq 2\) and \(y'\in \supp\,\psi\).
In Cases 1 and 2 below \(x_n<1\), so that \(|x|\approx |x'|\).

{\bf\(\bullet\) (Case 1: \(x_n^2 < 2(1-t)\))}\quad In this case we split the last integral in \eqref{0527-3} as follows.
\begin{align*}
\left(\int_0^{\frac12}  +  \int_{\frac12}^{3t-2}  + \int_{3t-2}^t\right) \cdots ds =:  I_1 +  I_2 + I_3.
\end{align*}
For \(I_1\), we use \(t-s\approx 1\) for \(\frac78<t<1\) and \(0<s<\frac12\) and thus \(\min\left\{1,\frac{x_n^2}{t-s}\right\}\approx x_n^2\) to get
\begin{align}\label{I0}
\begin{split}
I_1 & \approx \frac{ x_i x_n^3}{  |x'|^{n+2} }     \int_0^\frac12     (t-s)^{-\frac32} \phi(s)ds  +    x_i x_n^2  \int_0^{\frac12} (t-s)^{-\frac{n+4}2}  e^{-\frac{c|x'|^2}{t-s}}  \phi(s)ds \approx   \frac{ x_i x_n^3}{  |x'|^{n+2} }    +  x_i  x_n^2  e^{-c |x'|^2 }.
\end{split}
\end{align}
For \(I_2\), we use \(t-s\approx 1-s\) for \(\frac12<s<3t-2\) and \(\min\left\{1,\frac{x_n^2}{t-s}\right\}\approx \frac{x_n^2}{t-s}\) to get
\begin{align*}
\begin{split}
I_2 & \approx    \frac{ x_i x_n^3}{  |x'|^{n+2} }    \int^{t-\frac12}_{2(1-t)}      s^{a-\frac32 } ds  +   x_i   x_n^2  \int^{t-\frac12}_{2(1-t)}   s^{-\frac{n+4}2 +a }  e^{-\frac{c|x'|^2}{s}}     ds\\
& \approx    \frac{ x_i x_n^3}{  |x'|^{n+2} } \Big( \frac{1-(1-t)^{a-\frac12}}{a-\frac12}\mathbf{1}_{a\neq \frac12} +|\ln (1-t)| \mathbf{1}_{a = \frac12}  \Big) +   x_i   x_n^2   e^{-c|x'|^2}.
\end{split}
\end{align*}

For \(I_3\), we use \(1-s\approx 1-t\) for \(3t-2<s<t\), together with \(x_n^2<2(1-t)\), to get
\begin{align*}
\begin{split}
I_3 & \approx      \frac{ x_i x_n^3}{  |x'|^{n+2} }    ( 1-t)^a       \int_{3t-2}^{t-x_n^2}      (t-s)^{-\frac32}  ds  +  x_i   x_n^2    ( 1-t)^a \int_{3t-2 }^{t-x_n^2}   (t-s)^{-\frac{n+4}2} e^{-\frac{c|x'|^2}{t-s}}    ds\\
&  \quad +      \frac{x_i x_n}{ |x'|^{n+2}} (1-t)^a \int_{t -x_n^2}^t    ( t-s)^{-\frac12}   ds \approx     \frac{x_i x^2_n}{ |x'|^{n+2}} (1-t)^a.
\end{split}
\end{align*}

Hence, for \( x_n < (2(1-t))^{\frac12}\), we have
\begin{align*}
w^{L}_{ni} (x,t) &\approx  x_i x_n^2 e^{-c|x'|^2}+\frac{x_i x^2_n}{ |x'|^{n+2}} (1-t)^a+\frac{x_i x_n^3}{ |x'|^{n+2}} \left(  \frac{1-(1-t)^{a-\frac12}}{a-\frac12}\mathbf{1}_{a\neq \frac12}  +  |\ln(1-t)|\mathbf{1}_{a=\frac{1}{2}}\right),
\end{align*}
which is (i), since \(|x|\approx|x'|\).

{\bf\(\bullet\) (Case 2: \(2(1-t) \leq x_n^2 \leq t -\frac12\))}\quad In this case we split the last integral in \eqref{0527-3} as follows.
\begin{align*}
\left(\int_0^{\frac12}  +  \int_{\frac12}^{\frac52 t-\frac32}  + \int_{\frac52 t-\frac32}^t\right) \cdots ds =:  I_1 +  I_2 + I_3.
\end{align*}
The estimate of \(I_1\) is the same as \(\eqref{I0}\).

For \(I_2\), we find that \(t-s\approx 1-s\) for \(\frac12<s<\frac52 t-\frac32\), and that \(2(1-t)\leq x_n^2\leq t-\frac12\) implies \(\frac12\leq t-x_n^2 <\frac52 t-\frac32\).
Thus we further split \(I_2\) as follows.
\begin{align*}
\begin{split}
I_2 &\approx   \frac{x_i x_n^3}{  |x'|^{n+2}} \int_{x_n^2}^{t -\frac12}         s^{a -\frac32} ds    +  x_i  x_n^2  \int_{x_n^2}^{t -\frac12}      s^{-\frac{n+4}2 +a }    e^{-\frac{c|x'|^2}{s}}    ds  + \frac{x_i x_n}{    |x'|^{n+2}} \int_{\frac32( 1-t)}^{x_n^2}    s^{-\frac12 +a}   ds \\
&\approx   \frac{x_i x_n^3}{  |x'|^{n+2}} \left( \frac{1-x_n^{2a-1}}{a-\frac12}\mathbf{1}_{a\neq \frac12} + |\ln x_n|\mathbf{1}_{a =\frac12}  \right)     +  x_i  x_n^2   e^{-c|x'|^2}  \\
& \qquad    + \frac{x_i x_n}{    |x'|^{n+2}}  \left(\frac{x_n^{2a+1}-(1-t)^{a+\frac12}}{a+\frac12}\mathbf{1}_{a \neq -\frac12 } + \ln\frac{x_n^2}{1-t} \mathbf{1}_{a =-\frac12}  \right),
\end{split}
\end{align*}
where the last logarithm is at least \(\ln2>0\) because \(x_n^2\geq2(1-t)\).
For \(I_3\), we find that \(1-s\approx 1-t\) for \(\frac52 t-\frac32<s<t\) and \(\min\left\{1,\frac{x_n^2}{t-s}\right\}= 1\).
Thus we find that
\begin{align*}
I_3 &\approx \frac{ x_i  x_n}{ |x'|^{n+2}}  (1-t)^a  \int^t_{\frac52 t -\frac32}   (t-s)^{-\frac12}  ds   \approx  \frac{x_i x_n}{ |x'|^{n+2}}  (1-t)^{a+\frac12}.
\end{align*}

Combining all the estimates, we get
\begin{align*}
w^{L}_{ni}(x,t) & \approx   \frac{x_i x_n}{|x'|^{n+2}}\left(x_n^2\frac{1-x_n^{2a-1}}{a-\frac12}\mathbf{1}_{a\neq \frac12}   +  x_n^2 | \ln x_n|\mathbf{1}_{a =\frac12}  +  \frac{x_n^{2a+1}-(1-t)^{a+\frac12}}{a+\frac12}\mathbf{1}_{a\neq -\frac12}\right.\\
&\quad \left.  +  \ln\frac{x_n^2}{1-t}\mathbf{1}_{a=-\frac12}  + x_n^2+(1-t)^{a+\frac12}   + x_n |x'|^{n+2}  e^{-c|x'|^2}   \right)\\
&\approx \frac{x_ix_n}{|x'|^{n+2}}\Bigg( x_n^2\frac{1-x_n^{2a-1}}{a-\frac12}\mathbf{1}_{a\neq \frac12} +x_n^2|\ln x_n|\mathbf{1}_{a=\frac12}\\
&\qquad\qquad\qquad +\frac{x_n^{2a+1}-(1-t)^{a+\frac12}}{a+\frac12}\mathbf{1}_{a\neq -\frac12} +\ln\frac{x_n^2}{1-t}\mathbf{1}_{a=-\frac12} +x_n|x'|^{n+2}e^{-c|x'|^2}\Bigg),
\end{align*}
where only \(x_n^2\) and \((1-t)^{a+\frac12}\) are absorbed into the retained algebraic terms.
Since \(|x|\approx|x'|\), this is (ii).

{\bf\(\bullet\) (Case 3: \(x_n^2\geq t-\frac12\))}\quad In this case we split the last integral in \eqref{0527-3} as follows.
\begin{align*}
\left( \int_0^{\frac12} + \int_{\frac12}^{2t-1} + \int_{2t-1}^t\right) \cdots ds =: I_1 + I_2 + I_3.
\end{align*}
For \(I_1\) we further split the integral at \(s=t-x_n^2\), noting that \(t-x_n^2<\frac12\); if \(x_n^2\geq t\), then \(\min\left\{1,\frac{x_n^2}{t-s}\right\}=1\) for all \(0<s<\frac12\) and only the second integral below occurs.
Using \(t-s\approx1\) for \(0<s<\frac12\), we obtain
\begin{align*}
\begin{split}
I_1 &  \approx \left(  \frac{x_i x_n }{  |x|^{n+2}}      +     x_i  e^{-c|x|^2} \right) \int_0^{t -x_n^2} x_n^2\phi(s) ds +   \frac{x_i x_n}{  |x|^{n+2}} \int_{t -x_n^2}^\frac12     \phi(s)  ds  \approx   \frac{x_i x_n }{|x|^{n+2} },
\end{split}
\end{align*}
where we used \(x_n\approx1\) on \(t-\frac12\leq x_n^2\leq 1\) and \(t-x_n^2\leq0\) for \(x_n^2\geq t\).

For \(I_2\) and \(I_3\), using \(t-s\approx 1-s\) and \(1-s\approx 1-t\) for \(\frac12<s<2t-1\) and \(2t-1<s<t\) respectively, and \(\min\left\{1,\frac{x_n^2}{t-s}\right\}=1\), we have
\begin{align*}
I_2 &\approx       \frac{x_i  x_n}{    |x|^{n+2}}       \int_{\frac12}^{2t -1 }  (t-s)^{-\frac12 +a } ds \approx \frac{x_i x_n }{|x|^{n+2}}\left(\frac{1-(1-t)^{a+\frac12}}{a+\frac12}\mathbf{1}_{a\neq-\frac12}+|\ln(1-t)|\mathbf{1}_{a=-\frac12}\right),\\
I_3 &\approx \frac{x_i x_n }{|x|^{n+2}}  (1-t)^a  \int_{2t-1}^t   ( t-s)^{-\frac12}      ds \approx  \frac{x_i x_n }{|x|^{n+2}}  (1-t)^{a+\frac12}.
\end{align*}

Hence, for \( x_n^2 \geq t -\frac12\), with \(t\) close to \(1\), we have
\begin{align*}
w^{L}_{ni}(x,t) &\approx  \frac{x_i x_n }{|x|^{n+2} }\left( \frac{1-(1-t)^{a+\frac12}}{a+\frac12}\mathbf{1}_{a\neq -\frac12}  +  |\ln(1-t)|\mathbf{1}_{a=-\frac12}  +  (1-t)^{a+\frac12} + 1\right)\\
&\approx  \frac{x_i x_n }{|x|^{n+2} }\left( \frac{1-(1-t)^{a+\frac12}}{a+\frac12}\mathbf{1}_{a\neq -\frac12}  +  |\ln(1-t)|\mathbf{1}_{a=-\frac12}\right),
\end{align*}
which is (iii).
This completes the proof.
\end{proof}

In the next lemma, we give pointwise estimates for \(w_i^B(x,t)+w_i^N(x,t)\).

\begin{lemma}\label{thoe1021-1}
Let \(i\neq n\) and \(a\neq0\).
Then there exists \(\epsilon_1=\epsilon_1(a,M)\in(0,\frac18]\) such that the following holds whenever \(1-\epsilon_1<t<1\), \(|x'|\) is sufficiently large and \(x_i\geq\frac12|x'|\).
\begin{itemize}
\item[(i)] If \(x_n\geq \frac12\), we have
\begin{align*}
w_i^B(x,t)+w_i^N(x,t) \approx -\frac{x_ix_n}{|x'|^n}e^{-cx_n^2}+\frac{x_i}{|x|^n}(1-t)^a.
\end{align*}
\item[(ii)] If \((2(1-t))^{\frac12}\leq x_n<\frac12\), we have
\begin{align*}
w_i^B(x,t)+w_i^N(x,t)\approx -\frac{1}{|x'|^{n-1}}\left(x_n\frac{1-x_n^{2a-1}}{a-\frac12}\mathbf{1}_{a>0, a\neq \frac12}+x_n|\ln x_n|\mathbf{1}_{a=\frac12}-(1-t)^a\mathbf{1}_{-1<a<0}\right).
\end{align*}
\item[(iii)] If \(x_n<(2(1-t))^{\frac12}\), we have
\begin{align*}
w_i^B(x,t)+w_i^N(x,t)\approx -\frac{x_n}{|x'|^{n-1}}\left(\frac{1-(1-t)^{a-\frac12}}{a-\frac12}\mathbf{1}_{a>0, a\neq \frac12}+|\ln(1-t)|\mathbf{1}_{a=\frac12}-(1-t)^{a-\frac12}\mathbf{1}_{-1<a<0}\right).
\end{align*}
\end{itemize}
Since \(\frac78<t<1\), the three ranges above partition \((0,\infty)\).
\end{lemma}
\begin{proof}
We first outline the proof.
For (i), we estimate \(w_i^B\) and \(w_i^N\) separately.
For (ii) and (iii), we decompose \(w_i^B + w_i^N\) as follows. We set
\begin{align*}
R_i'\psi(x'):=\text{p.v.}\int_{\Rn}\partial_{x_i}N(x'-y',0)\,\psi(y')\,dy',
\end{align*}
which is one half of the \(i\)th Riesz transform of \(\psi\) in \(\Rn\).
Then \(w_i^N(x',0,t)=2R_i'\psi(x')\phi(t)\) by \(\eqref{wis}_3\), and \(w_i^B(x',0,t)=-2R_i'\psi(x')\phi(t)\), where \(w_i^B(x',0,t)\) denotes the limit as \(x_n\to0^+\) (see \eqref{0106-1} below).
Using the identity \(w_i^B(x',0,t)=-w_i^N(x',0,t)\), we write
\begin{align*}
w_i^B(x,t)+w_i^N(x,t)=(w_i^B(x,t)-w_i^B(x', 0, t)) + (w_i^N(x,t)-w_i^N(x',0,t)),
\end{align*}
which produces the factor \(x_n\) in the estimate in (iii).

{\bf\(\bullet\) (Case 1: \(x_n\geq \frac12\))}\quad Recall from Lemma \ref{lemma0709-1}(iii) that \(|J_3 (x', t) | \leq c t|x'|^{-n-1}\).
Since \( \partial_{x_i}N(x', 0) \approx \frac{x_i}{|x'|^{n }}\) for large \(|x'|\), we have \(\partial_{x_i}N(x', 0) + J_3(x',t-s) \approx x_i|x'|^{-n}\) for \(\frac78<t<1\), \(0<s<t\) and \(|x'|\approx x_i\).
Hence, by \(\eqref{wis}_2\) and Lemma \ref{lem-0317-01},
\begin{align*}
\begin{split}
w_i^B(x,t) &     \approx    - \frac{x_i x_n }{|x'|^{n }}\int_0^t  (t-s)^{-\frac32} e^{-\frac{x_n^2}{4(t-s)}}  \phi(s)ds=  - \frac{x_i x_n }{|x'|^{n }}\left(\int_{0}^{\frac12}+\int_{\frac12}^{2t-1}+\int_{2t-1}^t\right)\cdots ds\\
& =: - \frac{x_i x_n }{|x'|^{n }}\left(I_{1}+I_{2}+I_{3}\right).
\end{split}
\end{align*}
For \(I_1\), since \(t-s\approx 1\) for \(0<s<\frac12\) and \(\frac78<t<1\), we have
\begin{align*}
I_1 &  \approx        e^{-cx_n^2}   \int_0^\frac12  \phi(s)ds    \approx e^{-cx_n^2}.
\end{align*}
For \(I_2\), since \(t-s\approx 1-s\) for \(\frac12<s<2t-1\), we have
\begin{align*}
I_{2} & \approx \int_{1-t}^{t-\frac12}   s^{-\frac32 +a } e^{-\frac{x_n^2}{4s}} ds     \approx A(a -3/2, x_n, t),
\end{align*}
by Remark \ref{rmk-0928}.
For \(I_3\), since \(1-s\approx 1-t\) for \(2t-1<s<t\), we have
\begin{align*}
I_3 & \approx     (1-t)^a \int_0^{1-t}s^{-\frac32} e^{-\frac{x_n^2}{4s}}ds \approx (1-t)^a B(-3/2, x_n, t).
\end{align*}
Collecting all the estimates with Lemma \ref{lemma0928}(iii), which applies because \(x_n\geq\frac12\), we have
\begin{align}\label{0527-1}
\begin{split}
w^B_i (x,t) &\approx -  \frac{x_i x_n }{|x'|^{n }}  \Big(  e^{-cx_n^2}  + A(a -3/2, x_n, t)   + (1-t)^a B(-3/2, x_n, t) \Big)\\
&\approx -  \frac{x_i x_n }{|x'|^{n }}  \left(e^{-cx_n^2}+x_n^{-2}e^{-cx_n^2} + (1-t)^{a+\frac12} x_n^{-2}e^{-\frac{cx_n^2}{1-t}}\right)\approx  -  \frac{x_i x_n }{|x'|^{n }} e^{-cx_n^2},
\end{split}
\end{align}
where we used \(x_n^{-2}\leq4\).
On the other hand, it is easy to see from \(\eqref{wis}_3\) that for \(\frac34<t<1\),
\begin{align}\label{wN}
w^N_i(x,t) & \approx   (1 -t)^a x_i  |x|^{-n}, \quad 1 \leq i \leq n.
\end{align}
{\bf\(\bullet\) (Case 2: \(x_n<\frac12\))}\quad Note that \(\partial_{x_n} \partial_{x_i} N(x' -y', x_n) \approx - \frac{x_ix_n}{ |x|^{n+2}}\) for \( |y'| < 1/2\).
Hence, we have
\begin{align*}
\begin{split}
\partial_{x_n} w^N_i(x,t) &  =  2\phi(t)\int_{|y'|  <  \frac12} \partial_{x_n}  \partial_{x_i} N(x' -y', x_n)  \psi(y') dy'   \approx   - \phi(t)  \frac{x_ix_n}{ |x|^{n+2}}\|   \psi\|_{L^1(\Rn)}.
\end{split}
\end{align*}

Since \( w^{N}_i (x', 0,t) = 2 R'_i \psi (x') \phi(t)\), for \(|x'|\geq 1\), we have
\begin{align*}
\begin{split}
w^{N}_i(x,t) - 2 R'_i\psi (x') \phi(t) &  = \int_0^{x_n} \partial_{y_n} w^{N}_i  (x', y_n,t) dy_n \approx -  \frac{x_i x_n^2}{|x|^{n+2}} \phi(t).
\end{split}
\end{align*}

Since \(2\int_{-\infty}^t \int_{\Rn} \partial_{x_n} \Ga(x' -y', x_n, t-s) dy' ds= -1 \) for all \( (x,t)\) and \(\phi(s) =0\) for \(s < 0\), we have
\begin{align}\label{0106-1}
\begin{split}
w_i^B + 2 R'_i\psi(x') \phi(t) & =  4\int_{-\infty}^t   \int_{\Rn} \partial_{x_n} \Ga(x' -y', x_n, t-s)\big(  R'_i\psi(y') \phi(s)  -   R'_i\psi(x') \phi(t) \big) dy' ds\\
& = I_1 + I_2 + I_3,
\end{split}
\end{align}
where
\begin{align*}
\begin{split}
I_1&:=    4\int_0^t  \phi(s)  \int_{\Rn} \partial_{x_n} \Ga(x' -y', x_n, t-s) \big( R'_i \psi(y')   - R'_i \psi(x')    \big)  dy' ds,\\
I_2& :=    4  R'_i \psi(x') \int_0^t  \big( \phi(s) -   \phi(t)  \big) \int_{\Rn} \partial_{x_n} \Ga(x' -y', x_n, t-s)   dy' ds,\\
I_3&  :=    -4   R'_i \psi(x') \phi(t)\int_{-\infty}^0\int_{\mathbb{R}^{n-1}}\partial_{x_n} \Gamma(x'-y',x_n,t-s)dy'ds.
\end{split}
\end{align*}

{\it \underline{1.
Estimate of \(I_1\)}.}\quad Since \(| R'_i \psi(x')| \lesssim \| \psi\|_{C^1(\Rn)} \) for \( x' \in \Rn\), we have
\begin{align*}
\begin{split}
\left|\int_{|x' - y'| \geq \frac12|x'|} \Ga' (x' -y',t-s)\big( R'_i \psi(y')   - R'_i \psi(x')    \big)  dy'\right| & \leq c \int_{|y'| \geq \frac{|x'|}{2\sqrt{t-s}}}  \Ga' (y',1) dy'  \leq c e^{-\frac{c|x'|^2}{t-s}}.
\end{split}
\end{align*}
Note that \( |\na^2_{x'} R'_i \psi(x')| \leq c \| \psi\|_{L^1(\mathbb{R}^{n-1})} |x'|^{-n-1} \) for \( |x'|\geq 1\).
Since \(\int_{|y'| \leq \frac12 |x'| } y_k \Ga'(y', t-s) dy' =0 \) for all \( 1 \leq k \leq n-1\) and \(t -s> 0\), a second-order Taylor expansion of \(R_i'\psi\) about \(x'\) gives
\begin{align*}
\begin{split}
\left|\int_{|x' - y'| \leq \frac12 |x'|} \Ga' (x' -y',t-s)\big( R'_i \psi(y')   - R'_i \psi(x')    \big)     dy' \right| & \leq  c\, |x'|^{-n-1}\int_{|y'| \leq \frac12 |x'|} |y'|^2\,\Ga' (y',t-s)  dy'\\
& \leq c |x'|^{-n-1} (t-s).
\end{split}
\end{align*}
By the same calculation as in \eqref{0527-1}, we have
\begin{align}\label{0106-2}
\begin{split}
|I_1| & \leq    c  |x'|^{-n-1} x_n \Big(e^{-x_n^2} + A(a-1/2, x_n,t) +  (1-t)^{a}  B(-1/2, x_n, t)  \Big)+c\,x_ne^{-c|x'|^2},
\end{split}
\end{align}
where the last term comes from the region \(|x'-y'|\geq\frac12|x'|\): by the first bound above it is at most \(c\,x_n\int_0^t\phi(s)(t-s)^{-\frac32}e^{-\frac{c|x'|^2}{t-s}}ds\leq c\,x_ne^{-c|x'|^2}\int_0^1\phi(s)\,ds\), since \(t-s<1\) and \(|x'|\geq1\).

{\it \underline{2.
Estimate of \(I_2\)}.}\quad Since \( \int_{\Rn} \Ga'(x' -y',t) dy' =1\), we have
\begin{align*}
\begin{split}
I_2& =    4 R'_i \psi(x') \int_0^t  \big( \phi(s) -   \phi(t)  \big) \int_{\Rn} \partial_{x_n} \Ga(x' -y', x_n, t-s)   dy' ds\\
& \approx -  |x'|^{-n+1}\int_0^t  \big( \phi(s) -   \phi(t)  \big) \frac{x_n}{(t -s)^{\frac32}} e^{-\frac{x_n^2}{4(t-s)}} ds\\
&\approx -  |x'|^{-n+1}\left(\int_0^\frac12+\int_{\frac12}^{2t-1}+\int_{2t-1}^t\right)\cdots ds=:-  |x'|^{-n+1}(I_{21}+I_{22}+I_{23}).
\end{split}
\end{align*}

For \(0<s<\frac12\) and \(\frac78<t<1\), the weight
\begin{align*}
W(s):=\frac{x_n}{s^{3/2}}e^{-x_n^2/(4s)}
\end{align*}
satisfies \(c x_ne^{-x_n^2}\leq W(t-s)\leq Cx_ne^{-x_n^2/4}\).
Hence
\begin{align}\label{0211-1}
c\bigl(M-C\phi(t)\bigr)x_ne^{-x_n^2} \leq I_{21}\leq C\bigl(M-c\phi(t)\bigr)x_ne^{-x_n^2/4}.
\end{align}
After decreasing \(\epsilon_1(a,M)\), we may assume throughout this lemma that \(\phi(t)\leq\varepsilon M\) when \(a>0\), and \(\phi(t)\geq M/\varepsilon\) when \(a<0\), for a sufficiently small fixed \(\varepsilon>0\).

For \(I_{23}\), since \(\phi (s) -\phi(t) = -\phi'(\eta) (t-s) \approx -\phi'(t) (t-s)\) for \( 2t -1 <s <t\) and \(\phi'(t)=-a(1-t)^{a-1}\), we have
\begin{align}\label{0211-2}
\begin{split}
\frac{I_{23}}{a} &  \approx (1-t)^{a-1}  x_n  \int_0^{1-t}    s^{-\frac12} e^{-\frac{x_n^2}{4s}} ds    = (1-t)^{a-1}x_n  B(-1/2, x_n, t).
\end{split}
\end{align}

For \(I_{22}\), we first consider the case \( a > 0\).
In this case, we see that \(\frac12<s<2t-1\) implies \( \phi(s) - \phi(t) \approx \phi(s) = (1-s)^a \approx (t -s)^a\).
Hence, we have
\begin{align}\label{0211-3}
\begin{split}
I_{22} &  \approx   \int_{1-t}^{t-\frac12} x_n s^{-\frac32 +a}    e^{-\frac{x_n^2}{4s}} ds \approx x_n A(a-3/2, x_n, t).
\end{split}
\end{align}
On the other hand, if \( a < 0\), then \( \phi(s) - \phi(t) \approx -\phi(t)\) for \(\frac12 < s < 2t -1\).
Hence, we have
\begin{align}\label{0211-4}
\begin{split}
I_{22} &  \approx   - \phi(t) \int_{1-t}^{t-\frac12}  x_n s^{-\frac32 }    e^{-\frac{x_n^2}{4s}} ds   \approx - \phi(t) x_n  A(-3/2, x_n, t).
\end{split}
\end{align}

Since \( R'_i\psi(x') \approx |x'|^{-n+1}\), \( \phi(t) = (1-t)^a\) and \(e^{-x_n^2}\approx1\) for \(x_n<\frac12\), from \eqref{0211-1}--\eqref{0211-4}, we have
\begin{align}\label{0106-3}
I_2 \approx \frac{x_n}{|x'|^{n-1}}\left\{ \begin{array}{ll} \vspace{2mm} - \left(  M     +  A(a-3/2, x_n, t)   +a(1-t)^{a-1}B(-1/2,x_n,t)  \right), &\quad a > 0,\\
(1-t)^{a-1}\left( (1-t)  +   (1-t)  A(-3/2, x_n, t)   -a B(-1/2, x_n, t) \right), &\quad a < 0.
\end{array}
\right.
\end{align}

{\it \underline{3.
Estimate of \(I_3\)}.}\quad Since \( x_n <\frac12 \), we have
\begin{align}\label{0106-4}
I_{3} &\approx    |x'|^{-n+1}\phi(t) x_n \int_t^\infty \frac{1}{s^\frac{3}{2}}e^{-\frac{x_n^2}{4s}}ds \approx  |x'|^{-n+1}\phi(t) x_n.
\end{align}

{\it \underline{4.
Absorption of \(I_1\)}.}\quad The term \(I_1\) is of lower order than \(I_2\).
Indeed, for \(1-t\leq s\leq\frac12\) we have \(s^{a-\frac12}=s\cdot s^{a-\frac32}\leq\frac12s^{a-\frac32}\) if \(a>0\), and \(s^{a-\frac12}\leq(1-t)^as^{-\frac12}\leq\frac12(1-t)^as^{-\frac32}\) if \(a<0\).
Hence
\begin{align*}
A(a-\tfrac12,x_n,t)\leq\tfrac12A(a-\tfrac32,x_n,t)\quad(a>0),\qquad A(a-\tfrac12,x_n,t)\leq\tfrac12(1-t)^aA(-\tfrac32,x_n,t)\quad(a<0).
\end{align*}
Moreover \((1-t)^aB(-\frac12,x_n,t)=(1-t)\cdot(1-t)^{a-1}B(-\frac12,x_n,t)\), \(e^{-x_n^2}\leq1\) and \(1\leq(1-t)^a\) if \(a<0\).
Comparing the right-hand side of \eqref{0106-2} with \eqref{0106-3} term by term, and using that \eqref{0106-3} gives \(|I_2|\gtrsim x_n|x'|^{-n+1}M\) if \(a>0\) and \(|I_2|\gtrsim x_n|x'|^{-n+1}(1-t)^a\geq x_n|x'|^{-n+1}\) if \(a<0\), we obtain
\begin{align*}
|I_1|\leq C\bigl(|x'|^{-2}+|x'|^{n-1}e^{-c|x'|^2}\bigr)|I_2|,
\end{align*}
with \(C\) depending on \(a\) and \(M\).
Thus \(I_1\) can be absorbed into \(I_2\) for \(|x'|\) sufficiently large.

Combining all the estimates \eqref{0106-1}, \eqref{0106-2}, \eqref{0106-3} and \eqref{0106-4}, the absorption of \(I_1\) just shown, and using \(x_i\approx|x'|\approx|x|\), gives that for \( a > 0\) and \(x_n<\frac12\),
\begin{align}\label{0911-1}
\begin{split}
w_i^B(x,t)&+w_i^N(x,t) = \left(w_i^B(x,t) +2 R'_i \psi(x')  \phi(t)\right)+\left(w_i^N(x,t) -2 R'_i \psi(x')  \phi(t)\right) \\
&    \approx -  \frac{x_n}{|x'|^{n-1}}\Big( M   + A(a-3/2, x_n, t)  + a(1-t)^{a-1}   B(-1/2, x_n, t)  + \frac{x_n}{|x'|^2}(1-t)^a-(1-t)^a\Big)\\
&\approx -\frac{x_n}{|x'|^{n-1}}\left(M+A(a-3/2,x_n,t)+a(1-t)^{a-1}B(-1/2,x_n,t)\right)
\end{split}
\end{align}
and for \(a<0\),
\begin{align}\label{0911-2}
\begin{split}
w_i^B(x,t)+w_i^N(x,t) &= \left(w_i^B(x,t) + 2R'_i \psi(x')  \phi(t)\right)+\left(w_i^N(x,t) - 2R'_i \psi(x')  \phi(t)\right) \\
&    \approx  \frac{x_n}{|x'|^{n-1}}\Big(1   +  A(-3/2, x_n, t) -\frac{a}{1-t} B(-1/2, x_n, t) - \frac{x_n}{|x'|^2}  \Big) (1-t)^a\\
&\approx \frac{x_n}{|x'|^{n-1}}\left(1+A(-3/2,x_n,t)-\frac{a}{1-t}B(-1/2,x_n,t)\right)(1-t)^a,
\end{split}
\end{align}
where in both cases the term \(\frac{x_n}{|x'|^2}\) is negligible for \(|x'|\) large, and in \eqref{0911-1} the term \((1-t)^a\) is dominated by \(M\) after decreasing \(\epsilon_1(a,M)\), since \(a>0\).

{\bf \underline{Subcase 2-1}: \(x_n<(2(1-t))^{\frac12}\).} \quad By Lemma \ref{lemma0928}(i), we have for \(a>0\),
\begin{align*}
M+A(a-3/2,x_n,t)&+a(1-t)^{a-1}B(-1/2,x_n,t)\\
&\approx M+ \frac{2^{-a+\frac12}-(1-t)^{a-\frac12}}{a-\frac12}\mathbf{1}_{a\neq \frac12}+|\ln(2(1-t))|\mathbf{1}_{a=\frac12}\\
&\qquad + a(1-t)^{a-1}\left(2\left((1-t)^{\frac12}-2^{-\frac12}x_n\right)+x_n\right)\\
&\approx \frac{1-(1-t)^{a-\frac12}}{a-\frac12}\mathbf{1}_{a\neq \frac12}+|\ln(1-t)|\mathbf{1}_{a=\frac12},
\end{align*}
where we used \(x_n\leq(2(1-t))^{\frac12}\), so that \(a(1-t)^{a-1}B(-1/2,x_n,t)\approx a(1-t)^{a-\frac12}\), which is dominated by the retained terms.
On the other hand for \(a<0\),
\begin{align*}
1+A(-3/2,&x_n,t)-\frac{a}{1-t}B(-1/2,x_n,t)\\
&\approx 1+2\left((1-t)^{-\frac12}-2^{\frac12}\right)-\frac{a}{1-t}\left(2\left((1-t)^{\frac12}-2^{-\frac12}x_n\right)+x_n\right)\approx (1-t)^{-\frac12}.
\end{align*}
Together with \eqref{0911-1} and \eqref{0911-2}, this proves (iii).

{\bf \underline{Subcase 2-2}: \((2(1-t))^{\frac12}\leq x_n<\frac12\).} \quad By Lemma \ref{lemma0928}(ii), we have for \(a>0\),
\begin{align*}
M&+A(a-3/2,x_n,t)+a(1-t)^{a-1}B(-1/2,x_n,t)\\
&\approx M+ \frac{2^{-a+\frac12}-\left(4x_n^2/3\right)^{a-\frac12}}{a-\frac12}\mathbf{1}_{a\neq \frac12}+\ln\frac{3}{8x_n^2} \mathbf{1}_{a=\frac12} + \frac{a(1-t)^{a+\frac12}}{x_n^2}e^{-\frac{cx_n^2}{1-t}}\\
&\approx \frac{1-x_n^{2a-1}}{a-\frac12}\mathbf{1}_{a\neq \frac12}+|\ln x_n|\mathbf{1}_{a=\frac12},
\end{align*}
where the last term of the middle line is dominated because \(x_n^2\geq2(1-t)\).
On the other hand for \(a<0\),
\begin{align*}
1+A(-3/2,x_n,t)&-\frac{a}{1-t}B(-1/2,x_n,t)\approx 1+2\left(\left(4x_n^2/3\right)^{-\frac12}-2^{\frac12}\right)-\frac{a(1-t)^{\frac12}}{x_n^2}e^{-\frac{cx_n^2}{1-t}}\approx \frac{1}{x_n}.
\end{align*}
Together with \eqref{0911-1} and \eqref{0911-2}, this proves (ii).
This completes the proof.
\end{proof}

\subsubsection{{\bf Asymptotics of \texorpdfstring{\(x_{ni}^*\)}{x\_ni*} and sign of \texorpdfstring{\(w_i\)}{w\_i} }}

We first introduce the following sets: for \(\theta_0,\theta_1>0\),
\begin{align}\label{0107-3}
C_a^1:=\left\{(x',t) \mid (1-t)^a\lesssim(\ln|x'|)^{-\frac12}|x'|^{-\theta_0}\right\},\quad C_a^2:=\left\{(x',t) \mid |\ln(1-t)|\lesssim|x'|^2(1-t)^{\theta_1}\right\}.
\end{align}
We now state the main result of this subsection, which gives the signs of \(w_i\) for \(t<1\).

\begin{proposition}\label{3rdmainprop}
Let \(w\) be the solution of the Stokes system \eqref{StokesRn+} defined by \eqref{rep-bvp-stokes-w} with the boundary data \(g=g_n\mathbf{e}_n\) given by \eqref{0502-6}--\eqref{boundarydataspecific}, let \(a\neq0\), and let \(i=1,2,\dots,n-1\).
There are \(c_*>0\) and \(0<\epsilon_2<\frac18\), depending on \(n\), \(a\) and \(M\), such that if
\begin{align*}
|x'|>c_*,\qquad x_i\geq\frac12|x'|,\qquad 1-\epsilon_2<t<1,
\end{align*}
then the following holds.
\begin{itemize}
\item[(i)] If \(-1<a<0\), then \(w_i(x,t)>0\) for every \(x_n>0\).

\item[(ii)] If \(a>0\), then there are \(0<\theta_0<2\) and \(0<\theta_1<a\) such that, with \(C_a^1\) and \(C_a^2\) given in \eqref{0107-3}, one has \((x',t)\in C_a^1\cup C_a^2\) and there is a pair \((x_{ni2}^*,x_{ni1}^*)\in S^{-+}(w_i(x',\cdot\,,t);0,\infty)\), where \(S^{-+}\) is defined in Definition \ref{1202}, satisfying, for \(k=1,2\),
\begin{align*}
x_{nik}^*\approx(\ln|x'|)^{\frac12}\quad\text{if }(x',t)\in C_a^1,\qquad x_{nik}^*\approx|\ln(1-t)|^{\frac12}\quad\text{if }(x',t)\in C_a^2 .
\end{align*}
\end{itemize}
\end{proposition}

\begin{proof}
In the course of the proof \(c_*\) is enlarged and \(\epsilon_2\) is decreased finitely many times, each time by an amount depending only on \(n\), \(a\) and \(M\).
Since \(\frac12|x'|\leq x_i\leq|x'|\), we have \(x_i\approx|x'|\), the kernel estimates apply at the shifted points by Remark \ref{rmk-shift},
and \(|x|\approx|x'|\) whenever \(x_n<|x'|\).
By \eqref{1220-8},
\begin{align*}
w_i(x,t)=w^{L}_{ni}(x,t)+w^B_i(x,t)+w^N_i(x,t),
\end{align*}
and we estimate the three terms by Lemmas \ref{lemma1003-2} and \ref{thoe1021-1}; after decreasing \(\epsilon_2\) if necessary, we assume \(\epsilon_2\leq\epsilon_1(a,M)\), so that Lemma \ref{thoe1021-1} applies.
Since \(1-t<\epsilon_2<\frac18\), the four cases below partition \((0,\infty)\).

{\bf\(\bullet\) (Case 1: \( x_n<(2(1-t))^{\frac12}\))}\quad By Lemma \ref{lemma1003-2}(i) and Lemma \ref{thoe1021-1}(iii),
\begin{align}\label{0914-1}
\begin{split}
w^{L}_{ni}(x,t)&\approx\frac{x_n}{|x'|^{n-1}}\left(x_n\left(|x'|^{n}e^{-c|x'|^2} +\frac{(1-t)^a}{|x'|^2}\right)\right.\\
&\hspace{2.4cm}\left.+\frac{x_n^2}{|x'|^2} \left(\frac{1-(1-t)^{a-\frac12}}{a-\frac12}\mathbf{1}_{a\neq\frac12} +|\ln(1-t)|\mathbf{1}_{a=\frac12}\right)\right),\\
w^B_i(x,t)+w^N_i(x,t)&\approx-\frac{x_n}{|x'|^{n-1}} \left(\frac{1-(1-t)^{a-\frac12}}{a-\frac12}\mathbf{1}_{a>0,\,a\neq\frac12}\right.\\
&\hspace{2.4cm}\left.+|\ln(1-t)|\mathbf{1}_{a=\frac12}-(1-t)^{a-\frac12}\mathbf{1}_{-1<a<0}\right).
\end{split}
\end{align}
Since \(0<1-t<\frac18\), we have \(\frac{1-(1-t)^{a-\frac12}}{a-\frac12}>0\) for every \(a\neq\frac12\); indeed the numerator and the denominator are both positive if \(a>\frac12\) and both negative if \(a<\frac12\).
Hence every term in the first line of \eqref{0914-1} is positive and \(w^{L}_{ni}(x,t)>0\).

If \(-1<a<0\), the bracket in the second line of \eqref{0914-1} equals \(-(1-t)^{a-\frac12}<0\), so that \(w^B_i(x,t)+w^N_i(x,t)>0\) as well, and therefore \(w_i(x,t)>0\).

If \(a>0\), the bracket in the second line of \eqref{0914-1} is positive and coincides with the bracket multiplying \(\frac{x_n^2}{|x'|^2}\) in the first line, and it is bounded below by a positive constant depending only on \(a\), since \(1-t<\frac18\).
Taking the ratio of the two estimates in \eqref{0914-1} and using \(x_n<(2(1-t))^{\frac12}<\frac12\) and \((1-t)^a<1\), we obtain
\begin{align*}
\frac{w^{L}_{ni}(x,t)}{|w^B_i(x,t)+w^N_i(x,t)|}\leq c\left(|x'|^ne^{-c|x'|^2}+\frac1{|x'|^2}\right) \leq\frac12
\end{align*}
after enlarging \(c_*\).
Hence \(w_i(x,t)<0\).

{\bf\(\bullet\) (Case 2: \((2(1-t))^{\frac12}\leq x_n<\frac12\))}\quad By Lemma \ref{lemma1003-2}(ii) and Lemma \ref{thoe1021-1}(ii),
\begin{align}\label{0914-2}
\begin{split}
w^{L}_{ni}(x,t)&\approx\frac{x_n}{|x'|^{n+1}}\left(x_n^2\frac{1-x_n^{2a-1}}{a-\frac12} \mathbf{1}_{a\neq\frac12}+x_n^2|\ln x_n|\mathbf{1}_{a=\frac12}\right.\\
&\hspace{2.4cm}\left.+\frac{x_n^{2a+1}-(1-t)^{a+\frac12}}{a+\frac12}\mathbf{1}_{a\neq-\frac12} +\ln\frac{x_n^2}{1-t}\mathbf{1}_{a=-\frac12}\right),\\
w^B_i(x,t)+w^N_i(x,t)&\approx-\frac{1}{|x'|^{n-1}}\left(x_n\frac{1-x_n^{2a-1}}{a-\frac12} \mathbf{1}_{a>0,\,a\neq\frac12}\right.\\
&\hspace{2.4cm}\left.+x_n|\ln x_n|\mathbf{1}_{a=\frac12}-(1-t)^a\mathbf{1}_{-1<a<0}\right).
\end{split}
\end{align}
Since \(x_n<\frac12\), the quantity \(\frac{1-x_n^{2a-1}}{a-\frac12}\) is positive for every \(a\neq\frac12\).
Since \(x_n^2\geq2(1-t)\), we have \(\ln\frac{x_n^2}{1-t}\geq\ln2>0\), and \(\frac{x_n^{2a+1}-(1-t)^{a+\frac12}}{a+\frac12}>0\) for every \(a\neq-\frac12\): if \(a>-\frac12\) then \(x_n^{2a+1}>(2(1-t))^{a+\frac12}>(1-t)^{a+\frac12}\), while if \(a<-\frac12\) then \(x_n^{2a+1}<(2(1-t))^{a+\frac12}<(1-t)^{a+\frac12}\), the numerator and the denominator changing sign together.
Hence \(w^{L}_{ni}(x,t)>0\).

If \(-1<a<0\), the bracket in the second line of \eqref{0914-2} equals \(-(1-t)^a<0\), so that \(w^B_i(x,t)+w^N_i(x,t)>0\) and therefore \(w_i(x,t)>0\).

If \(a>0\), the exponentially small term retained in Lemma~\ref{lemma1003-2}(ii) contributes at most \(C x_n|x'|^n e^{-c|x'|^2}\) to the corresponding ratio and is absorbed after enlarging \(c_*\).
The two estimates in \eqref{0914-2} contain the common factor \(\frac{1-x_n^{2a-1}}{a-\frac12}\mathbf{1}_{a\neq\frac12}+|\ln x_n|\mathbf{1}_{a=\frac12}\), which is bounded below by a positive constant depending only on \(a\) because \(x_n<\frac12\), while the remaining term of the first bracket obeys \(0<\frac{x_n^{2a+1}-(1-t)^{a+\frac12}}{a+\frac12}\leq\frac{x_n^{2a+1}}{a+\frac12}\leq2\).
Taking the ratio of the two estimates in \eqref{0914-2}, we obtain
\begin{align*}
\frac{w^{L}_{ni}(x,t)}{|w^B_i(x,t)+w^N_i(x,t)|}\approx\frac{x_n^2}{|x'|^2} +\frac{1}{|x'|^2}\cdot\frac{\frac{x_n^{2a+1}-(1-t)^{a+\frac12}}{a+\frac12}} {\frac{1-x_n^{2a-1}}{a-\frac12}\mathbf{1}_{a\neq\frac12}+|\ln x_n|\mathbf{1}_{a=\frac12}} \leq\frac{c}{|x'|^2}\leq\frac12
\end{align*}
after enlarging \(c_*\).
Hence \(w_i(x,t)<0\).

{\bf\(\bullet\) (Case 3: \(\frac12\leq x_n<|x'|\))}\quad By Lemma \ref{lemma1003-2}(iii) and Lemma \ref{thoe1021-1}(i), for \(x_n>(t-\frac12)^{\frac12}\), and by the argument below for \(\frac12\leq x_n\leq(t-\frac12)^{\frac12}\),
\begin{align}\label{0914}
\begin{split}
w^{L}_{ni}(x,t)&\approx\frac{x_n}{|x'|^{n+1}}\left(\frac{1-(1-t)^{a+\frac12}}{a+\frac12} \mathbf{1}_{a\neq-\frac12}+|\ln(1-t)|\mathbf{1}_{a=-\frac12}\right),\\
w^B_i(x,t)&\approx-\frac{x_n}{|x'|^{n-1}}e^{-cx_n^2},\qquad w^N_i(x,t)\approx\frac{(1-t)^a}{|x'|^{n-1}} .
\end{split}
\end{align}
The bracket in the first line of \eqref{0914} is bounded below by a positive constant depending only on \(a\): it equals \(|\ln(1-t)|\geq\ln8\) if \(a=-\frac12\), it is at least \(\frac{1-8^{-a-\frac12}}{a+\frac12}\) if \(a>-\frac12\), and it is at least \(\frac{8^{|a+\frac12|}-1}{|a+\frac12|}\) if \(a<-\frac12\), because \(1-t<\frac18\).

For \(x_n>(t-\frac12)^{\frac12}\), the first estimate in \eqref{0914} follows directly from Lemma \ref{lemma1003-2}(iii).
On the remaining range \(\frac12\leq x_n\leq(t-\frac12)^{\frac12}\) we use Lemma \ref{lemma1003-2}(ii).
There \(\frac14\leq x_n^2<\frac12\) and \(|x|\approx|x'|\), so that
\begin{align*}
x_n^2\,\frac{1-x_n^{2a-1}}{a-\frac12}\,\mathbf{1}_{a\neq\frac12}+x_n^2|\ln x_n|\,\mathbf{1}_{a=\frac12}\approx1,
\end{align*}
and the exponentially small term is negligible for \(|x'|\) large.
Moreover, since \(1-t<\frac18\leq\frac12x_n^2\),
\begin{align*}
\frac{x_n^{2a+1}-(1-t)^{a+\frac12}}{a+\frac12}\,\mathbf{1}_{a\neq-\frac12}+\ln\frac{x_n^2}{1-t}\,\mathbf{1}_{a=-\frac12}
\approx
\begin{cases}
1, & a>-\frac12,\\[1mm]
|\ln(1-t)|, & a=-\frac12,\\[1mm]
(1-t)^{a+\frac12}, & a<-\frac12 .
\end{cases}
\end{align*}
The right-hand side is bounded below by a positive constant and is comparable to the bracket in the first line of \eqref{0914}.
Hence the bracket of Lemma \ref{lemma1003-2}(ii) is comparable to that bracket, and \eqref{0914} holds on the whole range \(\frac12\leq x_n<|x'|\).

If \(-1<a<0\), then \(w^{L}_{ni}(x,t)>0\) and \(w^N_i(x,t)>0\), while \eqref{0914} and \(\sup_{r>0}re^{-cr^2}=(2ec)^{-\frac12}\) give
\begin{align*}
\frac{|w^B_i(x,t)|}{w^N_i(x,t)}\leq\frac{c\,x_ne^{-cx_n^2}}{(1-t)^a} \leq\frac{c}{(2ec)^{\frac12}(1-t)^a}\leq\frac12
\end{align*}
after decreasing \(\epsilon_2\), since \((1-t)^a\rightarrow\infty\) as \(t\rightarrow1^-\).
Hence \(w_i(x,t)>0\).

Assume from now on that \(a>0\).
Since \(\frac{1-(1-t)^{a+\frac12}}{a+\frac12}\approx\frac{1}{a+\frac12}\) for \(0<1-t<\frac18\), we obtain from \eqref{0914} that there are \(0<c_j,d_j\), \(1\leq j\leq4\), such that
\begin{align}\label{0107-2}
\begin{split}
&  w_i (x, t) \leq  |x'|^{-n +1} \left(   \frac{d_1}{a+\frac12}  \frac{x_n}{|x'|^2}    -  d_2  x_n       e^{- d_3 x_n^2   } +  d_4  (1 -t)^a\right),\\
& w_i (x , t) \geq    |x'|^{-n +1} \left(   \frac{c_1}{a+\frac12}  \frac{x_n}{|x'|^2}    -  c_2       x_n e^{- c_3 x_n^2  } +  c_4  (1 -t)^a\right)
\end{split}
\end{align}
for \(\frac12\leq x_n<|x'|\).
Enlarging \(d_3\) and decreasing \(c_3\), which only weakens \eqref{0107-2}, we may and do assume \(c_3\leq d_3\).
We put
\begin{align}\label{0107-4}
\theta_0:=1\in(0,2),\qquad \theta_1:=\frac{a}{2}\in(0,a).
\end{align}

We first check that \(C_a^1\cup C_a^2\) exhausts the admissible \((x',t)\).
Suppose \((x',t)\notin C_a^1\), that is, \((1-t)^a>c(\ln|x'|)^{-\frac12}|x'|^{-1}\) for the constant \(c\) implicit in \eqref{0107-3}.
Raising this to the power \(\frac1a>0\) gives \(1-t>c^{\frac1a}(\ln|x'|)^{-\frac1{2a}}|x'|^{-\frac1a}\), whence
\begin{align*}
|\ln(1-t)|\leq\frac1a\ln|x'|+\frac1{2a}\ln\ln|x'|+\frac1a\ln\frac1c\leq c\ln|x'|.
\end{align*}
Moreover, taking square roots in the inequality for \((1-t)^a\) gives
\begin{align*}
(1-t)^{\theta_1}>c^{\frac12}(\ln|x'|)^{-\frac14}|x'|^{-\frac12}, \qquad |x'|^2(1-t)^{\theta_1}>c^{\frac12}|x'|^{\frac32}(\ln|x'|)^{-\frac14}.
\end{align*}
Since \(|x'|^{\frac32}(\ln|x'|)^{-\frac14}\) grows faster than \(\ln|x'|\), we get \((x',t)\in C_a^2\) after enlarging \(c_*\).

The functions \(r\mapsto re^{-c_3r^2}\) and \(r\mapsto re^{-d_3r^2}\) are decreasing on \([(2c_3)^{-\frac12},\infty)\) and on \([(2d_3)^{-\frac12},\infty)\) respectively.
Enlarging \(c_*\) and decreasing \(\epsilon_2\), we may assume that the point \(x_{ni2}^*\) chosen in each of the two cases below satisfies
\begin{align}\label{0107-5}
x_{ni2}^*\geq\max\left\{\tfrac12,\ (2c_3)^{-\frac12}\right\},\qquad x_{ni2}^*e^{-d_3(x_{ni2}^*)^2}\leq\tfrac12e^{-\frac{d_3}{4}} .
\end{align}
The second inequality of \eqref{0107-5} says that \(r\mapsto re^{-d_3r^2}\) is not smaller at \(r=\frac12\) than at \(r=x_{ni2}^*\), so that its minimum over \([\frac12,x_{ni2}^*]\) is attained at \(r=x_{ni2}^*\), whichever side of \((2d_3)^{-\frac12}\) the latter lies on; and since \(x_{ni1}^*>x_{ni2}^*\geq(2c_3)^{-\frac12}\), the maximum of \(r\mapsto re^{-c_3r^2}\) over \([x_{ni1}^*,\infty)\) is attained at \(r=x_{ni1}^*\).

{\it \underline{1.
The case \((x',t)\in C_a^1\)}.}\quad Put \(x_{ni2}^*:=\bigl(d_3^{-1}\ln|x'|\bigr)^{\frac12}\) and \(x_{ni1}^*:=\bigl(3c_3^{-1}\ln|x'|\bigr)^{\frac12}\).
Since \(c_3\leq d_3\) we have \(x_{ni2}^*<x_{ni1}^*\), and both are \(\approx(\ln|x'|)^{\frac12}\); moreover \(e^{-d_3(x_{ni2}^*)^2}=|x'|^{-\theta_0}\) and \(e^{-c_3(x_{ni1}^*)^2}=|x'|^{-3}\), and \(x_{ni1}^*<|x'|\) after enlarging \(c_*\).

Enlarging \(c_*\) so that \(\frac{d_1}{a+\frac12}|x'|^{-2}<\frac{d_2}{2}|x'|^{-\theta_0}\), which is possible because \(\theta_0<2\), the first term of \(\eqref{0107-2}_1\) is at most half of the second in absolute value for every \(\frac12\leq x_n\leq x_{ni2}^*\).
Hence, by \eqref{0107-5},
\begin{align*}
w_i(x,t)&\leq|x'|^{-n+1}\Bigl(-\frac{d_2}{2}x_ne^{-d_3x_n^2}+d_4(1-t)^a\Bigr)\\
&\leq|x'|^{-n+1}\Bigl(-\frac{d_2}{2}\bigl(d_3^{-1}\ln|x'|\bigr)^{\frac12}|x'|^{-\theta_0} +d_4(1-t)^a\Bigr)<0
\end{align*}
for all \(\frac12\leq x_n\leq x_{ni2}^*\), the last inequality because \((x',t)\in C_a^1\) and \(c_*\) has been enlarged.
Together with Cases 1 and 2, which give \(w_i(x,t)<0\) for \(0<x_n<\frac12\), this shows that \(w_i(x',x_n,t)<0\) for all \(0<x_n<x_{ni2}^*\).

For \(x_{ni1}^*\leq x_n<|x'|\) we have \(e^{-c_3x_n^2}\leq|x'|^{-3}\), so that, after enlarging \(c_*\) so that \(c_2|x'|^{-3}<\frac{c_1}{a+\frac12}|x'|^{-2}\), the inequality \(\eqref{0107-2}_2\) gives \(w_i(x,t)\geq|x'|^{-n+1}c_4(1-t)^a>0\) there.
Together with Case 4 below this shows that \(w_i(x',x_n,t)>0\) for all \(x_n>x_{ni1}^*\).
Consequently \((x_{ni2}^*,x_{ni1}^*)\in S^{-+}(w_i(x',\cdot\,,t);0,\infty)\) and \(x_{nik}^*\approx(\ln|x'|)^{\frac12}\).

{\it \underline{2.
The case \((x',t)\in C_a^2\)}.}\quad Put \(x_{ni2}^*:=\bigl(\frac{a}{2d_3}|\ln(1-t)|\bigr)^{\frac12}\) and \(x_{ni1}^*:=\bigl(\frac{2a}{c_3}|\ln(1-t)|\bigr)^{\frac12}\).
Since \(c_3\leq d_3\) we have \(x_{ni2}^*<x_{ni1}^*\), and both are \(\approx|\ln(1-t)|^{\frac12}\); moreover \(e^{-d_3(x_{ni2}^*)^2}=(1-t)^{\theta_1}\) and \(e^{-c_3(x_{ni1}^*)^2}=(1-t)^{2a}\).
Since \((x',t)\in C_a^2\), we have \((x_{ni1}^*)^2\leq c|x'|^2(1-t)^{\theta_1}\), so that \(x_{ni1}^*<|x'|\) after decreasing \(\epsilon_2\).

Since \((x',t)\in C_a^2\), we have \(|x'|^{-2}\lesssim(1-t)^{\theta_1}|\ln(1-t)|^{-1}\). Hence, after decreasing \(\epsilon_2\) so that \(|\ln(1-t)|\) is large, \(\frac{d_1}{a+\frac12}|x'|^{-2}<\frac{d_2}{2}(1-t)^{\theta_1}\), so that the first term of \(\eqref{0107-2}_1\) is at most half of the second in absolute value for every \(\frac12\leq x_n\leq x_{ni2}^*\).
Decreasing \(\epsilon_2\) so that \(d_4(1-t)^a<\frac{d_2}{2}\bigl(\frac{a}{2d_3}|\ln(1-t)|\bigr)^{\frac12}(1-t)^{\theta_1}\), which is possible because \(\theta_1<a\), we obtain from \eqref{0107-5}, for all \(\frac12\leq x_n\leq x_{ni2}^*\),
\begin{align*}
w_i(x,t)\leq|x'|^{-n+1}\Bigl(-\frac{d_2}{2}\Bigl(\frac{a}{2d_3}|\ln(1-t)|\Bigr)^{\frac12} (1-t)^{\theta_1}+d_4(1-t)^a\Bigr)<0 .
\end{align*}
Together with Cases 1 and 2 this gives \(w_i(x',x_n,t)<0\) for all \(0<x_n<x_{ni2}^*\).

Decreasing \(\epsilon_2\) so that \(c_2\bigl(\frac{2a}{c_3}|\ln(1-t)|\bigr)^{\frac12}(1-t)^{2a}<\frac{c_4}{2}(1-t)^a\), which is possible because \(a>0\), we obtain from \(\eqref{0107-2}_2\) and \eqref{0107-5}, for all \(x_{ni1}^*\leq x_n<|x'|\),
\begin{align*}
w_i(x,t)\geq|x'|^{-n+1}\left(\frac{c_1}{a+\frac12}\frac{x_n}{|x'|^2}+\frac{c_4}{2}(1-t)^a\right)>0 .
\end{align*}
Together with Case 4 below this shows that \(w_i(x',x_n,t)>0\) for all \(x_n>x_{ni1}^*\).
Consequently \((x_{ni2}^*,x_{ni1}^*)\in S^{-+}(w_i(x',\cdot\,,t);0,\infty)\) and \(x_{nik}^*\approx|\ln(1-t)|^{\frac12}\).

{\bf\(\bullet\) (Case 4: \(x_n\geq|x'|\))}\quad Here \(|x|\approx x_n\), so that Lemma \ref{lemma1003-2}(iii) and Lemma \ref{thoe1021-1}(i) give
\begin{align*}
w^{L}_{ni}(x,t)&\approx\frac{|x'|}{x_n^{n+1}}\left(\frac{1-(1-t)^{a+\frac12}}{a+\frac12} \mathbf{1}_{a\neq-\frac12}+|\ln(1-t)|\mathbf{1}_{a=-\frac12}\right),\\
w^B_i(x,t)&\approx-\frac{x_n}{|x'|^{n-1}}e^{-cx_n^2},\qquad w^N_i(x,t)\approx\frac{|x'|}{x_n^{n}}(1-t)^a .
\end{align*}
The bracket is bounded below by a positive constant depending only on \(a\), as in Case 3, and \(r\mapsto r^{n+2}e^{-cr^2}\) is decreasing for \(r\geq c_*\) once \(c_*\) is large.
Hence, for \(x_n\geq|x'|\),
\begin{align*}
\frac{|w^B_i(x,t)|}{w^{L}_{ni}(x,t)}\leq c\,\frac{x_n^{n+2}e^{-cx_n^2}}{|x'|^{n}} \leq c\,|x'|^2e^{-c|x'|^2}\leq\frac12
\end{align*}
after enlarging \(c_*\).
Since \(w^N_i(x,t)>0\), it follows that \(w_i(x,t)>0\), for every \(-1<a\leq1\) with \(a\neq0\).

Case 4 completes the proof of (i), and, together with Cases 1--3, of (ii).
\end{proof}

\begin{remark}\label{rp-rmk-C}
For fixed \(x'\), the set \(C_a^1\) consists of the times satisfying \(1-t\lesssim\rho(|x'|):=(\ln|x'|)^{-\frac{1}{2a}}|x'|^{-\frac{\theta_0}{a}}\), immediately below \(1\), and, since \(u\mapsto u^{\theta_1}|\ln u|^{-1}\) is increasing for small \(u>0\), the set \(C_a^2\) consists of the times satisfying \(1-t\gtrsim u_0(|x'|)\approx|x'|^{-\frac{2}{\theta_1}}(\ln|x'|)^{\frac{1}{\theta_1}}\), further from \(1\).
Comparing exponents, the two ranges overlap for all large \(|x'|\) precisely when \(\theta_0\theta_1<2a\), in particular under the constraints \(0<\theta_0<2\), \(0<\theta_1<a\) of Proposition \ref{3rdmainprop}; the proof above uses \(\theta_0=1\), \(\theta_1=\frac a2\) and establishes \((x',t)\in C_a^1\cup C_a^2\) directly.
\end{remark}

We now prove Theorem \ref{thmseparationpoint}.
\begin{proof}[{\it Proof of Theorem \ref{thmseparationpoint}}]
Let \(R_*>3c_*\), where \(c_*\) is the larger of the constants of Propositions \ref{1stmainprop} and \ref{3rdmainprop}, let \(|x'|>R_*\), and let \(i\in\{1,\dots,n-1\}\) satisfy \(x_i\geq\frac12|x'|\).
Then Propositions~\ref{1stmainprop} and~\ref{3rdmainprop} apply to \(w_i\); in particular \(|x'|>c_*(1+\sqrt t)\) for \(1<t<\frac98\), and all the time intervals \((1,1+\delta_2)\) below are taken inside \((1,\frac98)\).

{\it Proof of (i).}\quad Let \(-1<a\leq-\frac12\).
We verify the four conditions of Definition \ref{defsp} at the point \((x',0,1)\), in the first of the two alternatives there.

Condition (i) of Definition \ref{defsp}.
Cases 1 and 2 of the proof of Proposition \ref{3rdmainprop} together cover \(0<x_n<\frac12\) and give \(w_i(x',x_n,t)>0\) there for every \(1-\epsilon_2<t<1\), since \(a<0\).
Hence (i) holds with \(\delta_1:=\min\{\frac18,\epsilon_2\}\) and with the constant function \(\alpha_i(t):=\frac12\).

Conditions (ii), (iii) and (iv) of Definition \ref{defsp}.
By Proposition \ref{1stmainprop} there is, for each \(t>1\) under consideration, a pair \((x_{ni2}^*,x_{ni1}^*)\in S^{-+}(w_i(x',\cdot\,,t);0,\infty)\); that is, \(w_i(x',x_n,t)<0\) for \(0<x_n<x_{ni2}^*\) and \(w_i(x',x_n,t)>0\) for \(x_n>x_{ni1}^*\).
We distinguish three ranges of \(a\), following the definition \eqref{0804-2} of \(A_a^1\).

For \(-1<a<-\frac34\), we have \(A_a^1=U_1\bigl(t-1,(t-1)^{a+\frac12}\bigr)\), and Proposition \ref{1stmainprop}(ii)(a) gives \(x_{nik}^*\approx\bigl(\frac{(a+1)|x'|^2}{(t-1)^{a+\frac12}}\bigr)^{\frac{1}{1-2a}}\) for \((x',t)\in A_a^1\).
The condition \((x',t)\in A_a^1\) reads \(t-1\lesssim(a+1)|x'|^2\) and \((a+1)|x'|^2\lesssim(t-1)^{a+\frac12}\), and, \(a+\frac12\) being negative, it is satisfied for all \(1<t<1+\delta_2\) with
\begin{align*}
\delta_2\approx\min\left\{(a+1)|x'|^2,\ \bigl((a+1)|x'|^2\bigr)^{\frac{2}{2a+1}}\right\}.
\end{align*}

For \(-\frac34\leq a<-\frac12\), we have \(A_a^1=U_2\bigl(e^{-\frac{1-2a}{1+2a}}(t-1),(t-1)^{a+\frac12}\bigr)\), and Proposition \ref{1stmainprop}(ii)(b) gives \(x_{nik}^*\approx\bigl(\frac{|a+\frac12||x'|^2}{(t-1)^{a+\frac12}}\bigr)^{\frac{1}{1-2a}}\) for \((x',t)\in A_a^1\); the condition \((x',t)\in A_a^1\) is satisfied for all \(1<t<1+\delta_2\) with
\begin{align*}
\delta_2\approx\min\left\{e^{\frac{1-2a}{1+2a}}\Bigl|a+\frac12\Bigr||x'|^2,\ \Bigl|a+\frac12\Bigr|^{\frac{2}{1+2a}}|x'|^{\frac{4}{1+2a}}\right\}.
\end{align*}

For \(a=-\frac12\), Proposition \ref{1stmainprop}(ii)(c) gives \(x_{nik}^*\approx|x'|\bigl(\ln\frac{|x'|^2}{t-1}\bigr)^{-\frac12}\) for \((x',t)\in A_a^2\), where \(A_a^2=U_3(t-1,1)\) is given in \eqref{0804-3}; the condition \((x',t)\in A_a^2\) is satisfied for all \(1<t<1+\delta_2\) with \(\delta_2\approx|x'|^2e^{-c|x'|^2}\).

In each of the three ranges write \(x_{nik}^*\approx E(t)\) for the corresponding expression, so that there are \(0<c_1<c_2\), depending only on \(n\), \(a\) and \(M\), with \(c_1E(t)\leq x_{ni2}^*\) and \(x_{ni1}^*\leq c_2E(t)\) for \(1<t<1+\delta_2\), and put
\begin{align*}
\beta_{i1}(t):=c_1E(t),\qquad \beta_{i2}(t):=c_2E(t).
\end{align*}
Then \(\beta_{i1}\leq\beta_{i2}\), and both are nondecreasing on \((1,1+\delta_2)\): in the first two ranges \(E(t)\) is a positive multiple of \((t-1)^{\frac{a+\frac12}{2a-1}}\), whose exponent is positive because \(a<-\frac12\), and in the third \(E(t)=|x'|(\ln\frac{|x'|^2}{t-1})^{-\frac12}\) increases in \(t\).
Since \(E(t)\approx X_a(x',t)^{\frac{1}{1-2a}}\) by \eqref{Xa} (for \(a=-\frac12\) because \(\ln\frac{|x'|^2}{t-1}\approx|\ln(t-1)|\) on \(A_a^2\)), we have \(\beta_{ik}(t)\approx X_a(x',t)^{\frac{1}{1-2a}}\), as stated.
Moreover \(E(t)\rightarrow0\) as \(t\rightarrow1^+\) in each range, which is condition (iv).
Decreasing \(\delta_2\) once more, we may therefore assume \(\delta_2\leq\frac18\) and \(\beta_{i2}(t)<\frac18\) on \((1,1+\delta_2)\), so that \(\beta_{ik}\) maps \((1,1+\delta_2)\) into \((0,\delta)\) with \(\delta:=\frac18\).

Finally, \(\beta_{i1}(t)\leq x_{ni2}^*\) gives \(w_i(x',x_n,t)<0\) for all \(0<x_n<\beta_{i1}(t)\), which is condition (ii), and \(\beta_{i2}(t)\geq x_{ni1}^*\) gives \(w_i(x',x_n,t)>0\) for all \(\beta_{i2}(t)<x_n<\delta\), which is condition (iii).
Hence \((x',0,1)\) is a separation point of \(w\).

{\it Proof of (ii).}\quad Let \(-\frac12<a\leq1\) and \(1<t<\frac98\).
By the proof of Proposition \ref{1stmainprop}(iii), \((x',t)\in A_a^3\), so that the pair provided by Proposition \ref{1stmainprop} satisfies \(x_{ni2}^*>\frac12\) and hence \(w_i(x',x_n,t)<0\) for \(0<x_n<\frac12\).
Suppose that the conditions of Definition \ref{defsp} hold for \(w_i\) at \((x',0,1)\) with some \(\delta\), \(\delta_2\le\frac18\) and \(\beta_{i1}\leq\beta_{i2}\).
In the alternative indicated by ``resp.'', condition (ii) requires \(w_i>0\) on \((0,\beta_{i1}(t))\) for \(1<t<1+\delta_2\), which is impossible.
In the first alternative, condition (iii) requires \(w_i(x',x_n,t)>0\) on the nonempty interval \((\beta_{i2}(t),\delta)\), which forces \(\beta_{i2}(t)\geq\frac12\) for \(1<t<1+\delta_2\), in contradiction with condition (iv).
\end{proof}

\subsection{Normal component of the velocity field}

We recall from \eqref{1220-8w_n} that
\begin{align}\label{0805}
w_n(x,t)&=-\sum_{i=1}^{n-1}w_{ii}^{L}(x,t)+w_n^N(x,t).
\end{align}

\begin{proposition}[{{\bf Signs of \(w_n\) for \(t<1\)}}]\label{0716}
Let \(w\) be the solution of the Stokes system \eqref{StokesRn+} defined by \eqref{rep-bvp-stokes-w} with the boundary data \(g=g_n\mathbf{e}_n\) given by \eqref{0502-6}--\eqref{boundarydataspecific}, and assume in addition to \eqref{boundarydata} that, if \(-1<a<0\), \(\phi(s)\leq\phi(\tfrac12)=2^{-a}\) for \(0<s<\tfrac12\).
There exist \(R_1=R_1(n,a,M)>1\) and \(0<\gamma_1\leq \gamma_2\) depending only on \(n\) such that for \(|x'|\geq R_1\) the following holds.
\begin{itemize}
\item[(i)] Let \(-1<a<0\) and \(\frac12<t<1\).
Then \(w_n(x,t)>0\) for every \(x_n>0\).

\item[(ii)] Let \(0<a\leq1\).
There exists \(\de_0=\de_0(n,a,M)\in(0,\frac18]\) such that for \(1-\de_0<t<1\),
\begin{align}\label{0716-101}
w_n(x,t)<0 \quad\text{ for } 0<x_n\leq 1, \qquad w_n(x,t)>0 \quad\text{ for } x_n> \gamma_2|x'| .
\end{align}
Consequently \((1,\gamma_2|x'|)\in S^{-+}\bigl(w_n(x',\cdot\,,t);0,\infty\bigr)\), where \(S^{-+}\) is given in Definition \ref{1202}, and by Remark \ref{1125} there is \(x_{nn}^*\in[1,\gamma_2|x'|]\).

\item[(iii)] Let \(0<a\leq1\).
There exists \(\eta_0=\eta_0(n,M)>0\) such that if \(1-\de_0<t<1\) and
\begin{align}\label{0716-102}
(1-t)^a|x'|\leq \eta_0 ,
\end{align}
then \(w_n(x,t)<0\) for \(0<x_n<\gamma_1|x'|\) and \(w_n(x,t)>0\) for \(x_n>\gamma_2|x'|\).
Consequently \((\gamma_1|x'|,\gamma_2|x'|)\in S^{-+}\bigl(w_n(x',\cdot\,,t);0,\infty\bigr)\) and \(x_{nn}^*\approx|x'|\).
\end{itemize}
\end{proposition}

\begin{remark}\label{0716-rmk-i}
Item (i) contrasts with Proposition \ref{2ndmainprop}, which shows that \(w_n\) changes sign at a height comparable to \(|x'|\) for \(t>1\) and every \(a\in(-1,1]\), since then \(w^N_n\equiv0\).
For \(t<1\) and \(-1<a<0\), the term \(w_n^N\) dominates at heights \(x_n\gtrsim|x'|\), and the terms \(A_1\) and \(A_2\) in the proof below are positive closer to the wall; thus the sign change of \(w_n\) occurs only after \(t=1\).
\end{remark}

\begin{proof}
Throughout, \(c\) and \(C\) denote positive constants depending only on \(n\), which may change from line to line.
Since \(\supp\,\psi\subset B'_{\frac12}\) and \(|x'|\geq R_1\geq2\), for every \(y'\in\supp\,\psi\),
\begin{align}\label{0716-103}
\tfrac12|x|\leq |x-y'|\leq \tfrac32|x| ,\qquad \tfrac12|x'|\leq |x'-y'|\leq \tfrac32|x'| ,
\end{align}
where \(x-y'=(x'-y',x_n)\) as in \eqref{rep-bvp-stokes-w}.
Moreover, if \(R_1\) is taken larger than \(2\ep_0^{-1}\) and than \(2R_0\), with \(R_0\) and \(\ep_0\) as in Lemma \ref{estimateofkernelt<1}, then
\begin{align}\label{0716-104}
|x'-y'|\geq \max\{1,\sqrt{t-s}\}\qquad\text{and}\qquad \sqrt{t-s}<\ep_0|x'-y'|
\end{align}
for all \(0<s<t<1\) and all \(y'\in\supp\,\psi\), so that Lemmas \ref{technicallemma}, \ref{lemma0709-1} and \ref{estimateofkernelt<1} are all applicable at \((x-y',t-s)\).
We shall use repeatedly the two elementary bounds, valid for \(x_n>0\) and \(0<t<1\),
\begin{align}\label{0716-105}
\int_0^t \frac{1}{\sqrt u}\min\Bigl\{1,\frac{x_n^2}{u}\Bigr\}du\leq 4\min\{x_n,1\}, \qquad \int_0^t \sqrt u\,\min\Bigl\{1,\frac{x_n^2}{u}\Bigr\}du\leq 2x_n\min\{x_n,1\},
\end{align}
both of which follow on splitting the integral at \(u=x_n^2\).

By \eqref{1220-8w_n} and \eqref{wis},
\begin{align}\label{1020-1}
w_n(x,t)=2\phi(t)\int_{\Rn}\pa_{x_n}N(x-y')\psi(y')\,dy' -4\sum_{k=1}^{n-1}\int_0^t\!\!\int_{\Rn}L_{kk}(x-y',t-s)\psi(y')\phi(s)\,dy'ds .
\end{align}
Since
\begin{align*}
L_{kk}(x,t)=\int_0^{x_n}\pa_{x_n}\Ga_1(x_n-z_n,t) \left[\int_{\Rn}\pa_{x_k}\Ga'(x'-z',t)\,\pa_{z_k}N(z',z_n)\,dz'\right]dz_n,
\end{align*}
applying \eqref{0306-1} of Lemma \ref{lemma0709-1}(ii) with \(i=j=k\) to the inner integral, we get
\begin{align}\label{1020-2}
\sum_{k=1}^{n-1}L_{kk}(x,t)=-\int_0^{x_n}\pa_{x_n}\Ga_1(x_n-z_n,t) \Bigl(\pa_{z_n}^2N(x',z_n)-J(x',z_n,t)\Bigr)dz_n,
\end{align}
where \(J:=\sum_{k<n}J_2\) satisfies the bound of \(J_2\) in Lemma \ref{lemma0709-1}(ii), namely
\begin{align}\label{0716-107}
|J(x',z_n,t)|\leq c\,\frac{t}{(|x'|^2+z_n^2)^{\frac{n+2}{2}}} .
\end{align}

Now writing \(\phi(s)=\bigl(\phi(s)-\phi(t)\bigr)+\phi(t)\) in \eqref{1020-1} we set
\begin{align}\label{0716-108}
A_1(x,t):=-4\int_0^t\bigl(\phi(s)-\phi(t)\bigr)\int_{\Rn}\psi(y') \sum_{k=1}^{n-1}L_{kk}(x-y',t-s)\,dy'ds,
\end{align}
so that \(w_n-A_1\) equals the first term of \eqref{1020-1} minus \(4\phi(t)\int_0^t\int_{\Rn}\psi\sum_{k<n}L_{kk}\,dy'ds\).
Inserting \eqref{1020-2} into the latter, and using \(\pa_{x_n}N(x',0)=0\), which gives \(\int_0^{x_n}\pa_{z_n}^2N(x'-y',z_n)dz_n=\pa_{x_n}N(x-y')\), we obtain
\begin{align*}
w_n(x,t)-A_1(x,t)&=\int_0^{x_n}\!\!\int_{\Rn}\pa_{z_n}^2N(x'-y',z_n)\psi(y')\\
&\qquad\times \Bigl(2\phi(t)+4\phi(t)\int_0^t\pa_{x_n}\Ga_1(x_n-z_n,t-s)\,ds\Bigr)dy'dz_n+A_3(x,t),
\end{align*}
where
\begin{align}\label{0716-109}
A_3(x,t):=-4\phi(t)\int_0^t\!\!\int_{\Rn}\!\int_0^{x_n} \pa_{x_n}\Ga_1(x_n-z_n,t-s)\,J(x'-y',z_n,t-s)\,\psi(y')\,dz_n\,dy'ds .
\end{align}
For \(r>0\) we have the identity
\begin{align}\label{0716-110}
\int_{-\infty}^{t}\pa_{x_n}\Ga_1(r,t-s)\,ds =-\frac{r}{2\sqrt{4\pi}}\int_0^\infty u^{-\frac32}e^{-\frac{r^2}{4u}}\,du=-\frac12 ,
\end{align}
which follows from the substitution \(u=\frac{r^2}{4v}\) together with \(\int_0^\infty v^{-\frac12}e^{-v}dv=\sqrt\pi\).
Hence the bracket above equals \(-4\phi(t)\int_{-\infty}^0\pa_{x_n}\Ga_1(x_n-z_n,t-s)ds\), and, setting
\begin{align}\label{0716-111}
A_2(x,t):=-4\phi(t)\int_0^{x_n}\!\!\int_{\Rn}\pa_{z_n}^2N(x'-y',z_n)\psi(y')\,dy' \int_{-\infty}^0\pa_{x_n}\Ga_1(x_n-z_n,t-s)\,ds\,dz_n,
\end{align}
we arrive at
\begin{align}\label{0716-112}
w_n(x,t)=A_1(x,t)+A_2(x,t)+A_3(x,t).
\end{align}
The interchanges of the order of integration above are justified by \(\int_0^t(t-s)^{-\frac12}\phi(s)\,ds<\infty\) (which holds for \(t<1\) because \(\phi\) is bounded near \(t\)), by \(|\pa_{z_n}^2N(x'-y',z_n)|\leq c|x'|^{-n}\) on the region of integration, and by \(\int_0^{x_n}|\pa_{x_n}\Ga_1(x_n-z_n,t-s)|\,dz_n\leq c(t-s)^{-\frac12}\).

With \(\kappa_1,\dots,\kappa_4\) as in \eqref{240707} we set
\begin{align}\label{0716-113}
\gamma_1:=\tfrac12\min\Bigl\{\Bigl(\frac{\kappa_1}{2\kappa_2}\Bigr)^{\frac12},\ \frac1{\sqrt{n-1}}\Bigr\}, \qquad \gamma_2:=2\Bigl(\frac{\kappa_3}{\kappa_4}\Bigr)^{\frac12},
\end{align}
and, decreasing \(\gamma_1\) if necessary, we assume \(\gamma_1\leq \gamma_2\).
Let \(0<x_n\leq \gamma_1|x'|\) and \(y'\in\supp\,\psi\).
Then \(x_n\leq 2\gamma_1|x'-y'|\) by \eqref{0716-103}, so that \(\kappa_2x_n^2\leq\frac12\kappa_1|x'-y'|^2\); hence, by \eqref{240707} applied at \((x-y',t-s)\), which is legitimate by \eqref{0716-104}, and by \eqref{0716-103},
\begin{align}\label{0716-114}
c\,\frac{|x'|^2}{|x|^{n+2}}\frac{1}{\sqrt{t-s}}\min\Bigl\{1,\frac{x_n^2}{t-s}\Bigr\} \ \leq\ \sum_{k=1}^{n-1}L_{kk}(x-y',t-s)\ \leq\ C\,\frac{|x'|^2}{|x|^{n+2}}\frac{1}{\sqrt{t-s}}\min\Bigl\{1,\frac{x_n^2}{t-s}\Bigr\}
\end{align}
for all \(0<s<t\).
In particular \(\sum_{k<n}L_{kk}(x-y',t-s)>0\) there.

We now claim that
\begin{align}\label{0716-115}
w_n(x,t)>0\qquad\text{whenever } x_n>\gamma_2|x'| ,
\end{align}
for every \(a\in(-1,1]\) and every \(\frac12<t<1\).
This uses only \eqref{240707}, \eqref{1220-8w_n} and the signs of \(\psi\), \(\phi\) and \(\pa_{x_n}N\).

Let \(x_n>\gamma_2|x'|\) and \(y'\in\supp\,\psi\).
By \eqref{0716-103} and \eqref{0716-113}, \(x_n>\gamma_2|x'|\geq\frac23\gamma_2|x'-y'|=\frac43(\kappa_3/\kappa_4)^{\frac12}|x'-y'|\), whence \(\kappa_3|x'-y'|^2-\kappa_4x_n^2<0\).
By \eqref{0716-104} the second inequality of \eqref{240707} applies at \((x-y',t-s)\) and gives
\begin{align}\label{0716-116}
\sum_{k=1}^{n-1}L_{kk}(x-y',t-s)\leq \frac{\kappa_3|x'-y'|^2-\kappa_4x_n^2}{(t-s)^{\frac12}|x-y'|^{n+2}} \min\Bigl\{1,\frac{x_n^2}{t-s}\Bigr\}<0
\end{align}
for all \(0<s<t\).
Since \(\psi\geq0\) and \(\phi\geq0\) by \eqref{boundarydata}, integrating \eqref{0716-116} against \(\psi(y')\phi(s)\) shows that the second term on the right of \eqref{1020-1} is strictly positive.
The first term is strictly positive as well, because \(\pa_{x_n}N(z)=\frac{z_n}{n\omega_n|z|^n}>0\) for \(z_n>0\) and \(\phi(t)=(1-t)^a>0\); indeed, by \eqref{0716-103},
\begin{align}\label{0716-117}
2\phi(t)\int_{\Rn}\pa_{x_n}N(x-y')\psi(y')\,dy'\ \geq\ \frac{c\,\norm{\psi}_{L^1(\Rn)}\,(1-t)^a\,x_n}{|x|^{n}}\ >\ 0 .
\end{align}
Adding the two terms gives \eqref{0716-115}.

\medskip {\it \underline{1. Estimate of \(A_1\) near the wall}.}\quad
In this step \(0<x_n\leq \gamma_1|x'|\) and \(\frac12<t<1\).
As in the estimate of \(w_n\) for \(t>1\), we split
\begin{align}\label{0601-2}
\int_0^t \frac{1}{\sqrt{t-s}}\min\Bigl\{1,\frac{x_n^2}{t-s}\Bigr\}\bigl(\phi(s)-\phi(t)\bigr)ds =\int_0^{\frac12}\cdots ds+\int_{\frac12}^t\cdots ds=:A_{11}+A_{12}.
\end{align}

\emph{(a) The case \(-1<a<0\).}
Here \(\phi(s)<\phi(t)\) for every \(s\in(0,t)\).
Indeed, for \(\frac12<s<t\) this is \((1-s)^a<(1-t)^a\); and for \(0<s\leq\frac12\), the assumption on \(\phi\) gives \(\phi(s)\leq\phi(\frac12)=2^{-a}<(1-t)^a=\phi(t)\), since \(1-t<\frac12\) and \(a<0\).
Thus the integrand of \eqref{0601-2} is negative throughout, and \eqref{0716-108} together with \eqref{0716-114} gives
\begin{align}\label{0716-118}
c\,\norm{\psi}_{L^1(\Rn)}\frac{|x'|^2}{|x|^{n+2}}\bigl|A_{11}+A_{12}\bigr| \ \leq\ A_1(x,t)\ \leq\ C\,\norm{\psi}_{L^1(\Rn)}\frac{|x'|^2}{|x|^{n+2}}\bigl|A_{11}+A_{12}\bigr| .
\end{align}
In particular \(A_1>0\).

\emph{(b) The case \(0<a\leq1\).}
Throughout (b) we assume \(1-\de_0<t<1\), where \(\de_0=\de_0(n,a,M)\in(0,\frac18]\) is fixed in \eqref{0716-124} below; in particular \(t>\frac78\).
For \(0<s<\frac12\), the scalar factor in \eqref{0601-2} is then comparable to \(\min\{1,x_n^2\}\).
When \eqref{0716-114} is inserted into \eqref{0716-108}, however, its lower and upper constants must be kept separate because \(\phi(s)-\phi(t)\) changes sign.
Denote these constants by \(0<c_K\leq C_K\).
On \((0,\frac12)\) write \(\phi(s)-\phi(t)=(\phi(s)-\phi(t))_+-(\phi(s)-\phi(t))_-\).
The negative part is weighted by at most \(C_K\), whereas the positive parts are weighted by at least \(c_K\).
Since \(\phi\geq0\),
\begin{align*}
\int_0^{\frac12}(\phi(s)-\phi(t))_-\,ds\leq\frac12\phi(t),\qquad \int_0^{\frac12}(\phi(s)-\phi(t))_+\,ds\geq M-\frac12\phi(t),
\end{align*}
and the contribution from \((\frac12,t)\) is nonnegative, we obtain
\begin{align}\label{A11estimate}
-A_1(x,t)\geq c\norm{\psi}_{L^1(\Rn)}\frac{|x'|^2}{|x|^{n+2}} \min\{1,x_n^2\}\left(M-\frac{C_K}{c_K}\,C\phi(t)\right).
\end{align}
By \eqref{0716-124}, where the constant \(c=c(n)\) is chosen with \(c\geq2CC_K/c_K\), the last bracket is at least \(M/2\).
Therefore
\begin{align}\label{0716-121}
A_1(x,t)\leq-\,c\norm{\psi}_{L^1(\Rn)}\frac{|x'|^2}{|x|^{n+2}}M\min\{1,x_n^2\} \leq-\frac{c\norm{\psi}_{L^1(\Rn)}M\min\{1,x_n^2\}}{|x'|^n}<0,
\end{align}
where \(|x|\approx|x'|\) since \(x_n\leq \gamma_1|x'|\).

\medskip {\it \underline{2. Estimates of \(A_2\) and \(A_3\) near the wall}.}\quad
Let \(0<x_n\leq \gamma_1|x'|\) and \(\frac12<t<1\).
A substitution as in \eqref{0716-110} gives, for \(r>0\),
\begin{align}\label{0716-122}
-\int_{-\infty}^0\pa_{x_n}\Ga_1(r,t-s)\,ds =\frac{r}{2\sqrt{4\pi}}\int_t^\infty u^{-\frac32}e^{-\frac{r^2}{4u}}du\approx\min\{1,r\},
\end{align}
the two-sided constants being absolute because \(\frac12<t<1\).
Since \(z_n<x_n\leq \gamma_1|x'|\leq\frac12(n-1)^{-\frac12}|x'|\) and \(|x'-y'|\geq\frac12|x'|\) by \eqref{0716-103}, we have \(|x'-y'|^2-(n-1)z_n^2>0\) for \(0<z_n<x_n\) and \(y'\in\supp\,\psi\). Hence, by
\begin{align}\label{0716-123}
\pa_{z_n}^2N(z',z_n)=\frac{|z'|^2-(n-1)z_n^2}{n\omega_n\bigl(|z'|^2+z_n^2\bigr)^{\frac{n+2}{2}}},
\end{align}
we get \(\pa_{z_n}^2N(x'-y',z_n)>0\) there, and \(A_2>0\) by \eqref{0716-111} and \eqref{0716-122}.
Furthermore, integrating by parts, the primitive of the integrand on the left being \(z_n(|z'|^2+z_n^2)^{-\frac n2}\), which vanishes at \(z_n=0\),
\begin{align*}
\int_0^{x_n}\frac{|z'|^2-(n-1)z_n^2}{(|z'|^2+z_n^2)^{\frac{n+2}{2}}}\min\{1,x_n-z_n\}\,dz_n =\int_{\max\{x_n-1,0\}}^{x_n}\frac{z_n\,dz_n}{(|z'|^2+z_n^2)^{\frac{n}{2}}} \approx\frac{x_n\min\{1,x_n\}}{(|z'|^2+x_n^2)^{\frac n2}},
\end{align*}
so that, by \eqref{0716-103},
\begin{align}\label{A2final}
A_2(x,t)\ \approx\ \norm{\psi}_{L^1(\Rn)}\,(1-t)^a\,\frac{x_n\min\{1,x_n\}}{|x|^{n}}\ >\ 0 .
\end{align}

For \(A_3\), we insert \eqref{0716-107} into \eqref{0716-109} and carry out the \(z_n\)-integration by means of \eqref{I-mk} and Lemma \ref{technicallemma}, which gives
\begin{align*}
&\Bigl|\int_0^{x_n}\pa_{x_n}\Ga_1(x_n-z_n,t-s)J(x'-y',z_n,t-s)\,dz_n\Bigr|\\
&\qquad\leq c\,(t-s)\, I_{n+2,0}(x-y',t-s)\\
&\qquad\leq \frac{c\sqrt{t-s}}{|x|^{n+2}} \min\Bigl\{1,\frac{x_n^2}{t-s}\Bigr\}
\end{align*}
for \(y'\in\supp\,\psi\).
Hence, by \eqref{0716-109} and \eqref{0716-105},
\begin{align}\label{0313-01}
\begin{split}
|A_3(x,t)| &\leq \frac{C\norm{\psi}_{L^1(\Rn)}\phi(t)}{|x|^{n+2}} \int_0^t\sqrt{t-s}\, \min\Bigl\{1,\frac{x_n^2}{t-s}\Bigr\}ds\\
&\leq \frac{C\norm{\psi}_{L^1(\Rn)}\phi(t)}{|x|^{n+2}}\,x_n\min\{1,x_n\} \ \leq\ \frac{C}{|x'|^{2}}\,A_2(x,t),
\end{split}
\end{align}
where we used the second bound in \eqref{0716-105} and, in the last step, \eqref{A2final}.
In particular \(|A_3|\leq\frac18A_2\) once \(R_1\) is taken large enough.

\medskip {\it \underline{3. Proof of (i)}.}\quad
Let \(-1<a<0\) and \(\frac12<t<1\).
By \eqref{0716-115} we may assume \(0<x_n\leq \gamma_2|x'|\).

\emph{(a) The range \(0<x_n\leq \gamma_1|x'|\).}
By \eqref{0716-118} we have \(A_1>0\), and \(A_2>0\) by \eqref{A2final}; hence, by \eqref{0716-112} and \eqref{0313-01},
\begin{align*}
w_n(x,t)=A_1+A_2+A_3\geq A_1+\tfrac78A_2>0 .
\end{align*}

\emph{(b) The range \(\gamma_1|x'|\leq x_n\leq \gamma_2|x'|\).}
Here we compare the two terms of \eqref{1020-1} directly.
Enlarging \(R_1\) so that \(\gamma_1R_1\geq1\), we have \(x_n\geq \gamma_1|x'|\geq1\), and therefore \(\min\{1,\frac{x_n^2}{t-s}\}=1\) for all \(0<s<t<1\).
By \eqref{0716-103} we have \(|x'-y'|^2+x_n^2\approx|x-y'|^2\approx|x|^2\), so \eqref{240707} gives, for \(0<s<t\) and \(y'\in\supp\,\psi\),
\begin{align*}
\Bigl|\sum_{k<n}L_{kk}(x-y',t-s)\Bigr| \leq\frac{\max\bigl\{\bigl|\kappa_1|x'-y'|^2-\kappa_2x_n^2\bigr|,\ \bigl|\kappa_3|x'-y'|^2-\kappa_4x_n^2\bigr|\bigr\}}{(t-s)^{\frac12}|x-y'|^{n+2}} \leq\frac{c}{\sqrt{t-s}\,|x|^{n}} ,
\end{align*}
so that the second term of \eqref{1020-1} is bounded in absolute value by \(c\norm{\psi}_{L^1(\Rn)}|x|^{-n}\int_0^t(t-s)^{-\frac12}\phi(s)ds\).
Splitting this integral at \(s=\frac12\), using \(\phi\leq\phi(\frac12)=2^{-a}\) on \((0,\frac12)\) together with \(\int_0^{\frac12}(t-s)^{-\frac12}ds\leq2\sqrt t\leq2\), and then \(1-s=1-t+u\geq u\) with \(u=t-s\) and \(a<0\) on \((\frac12,t)\), we obtain
\begin{align*}
\int_0^t\frac{\phi(s)}{\sqrt{t-s}}\,ds\leq 2^{1-a}+\int_0^{t-\frac12}u^{-\frac12}(1-t+u)^a\,du \leq 2^{1-a}+c_a\Bigl((1-t)^{a+\frac12}+1+|\ln(1-t)|\,\mathbf 1_{a=-\frac12}\Bigr).
\end{align*}
Each term on the right is at most \(C(a)(1-t)^a\) on \((\frac12,1)\): indeed \((1-t)^{a+\frac12}=(1-t)^a(1-t)^{\frac12}\leq(1-t)^a\), while \(1<2^{-a}\leq(1-t)^a\) and \(|\ln(1-t)|\leq c_a(1-t)^a\), all because \(a<0\) and \(1-t<\frac12\).
Combining this with \eqref{0716-117},
\begin{align*}
w_n(x,t)\geq\frac{c\,\norm{\psi}_{L^1(\Rn)}(1-t)^a}{|x|^{n}}\bigl(x_n-C(a)\bigr)>0,
\end{align*}
because \(x_n\geq \gamma_1|x'|\geq \gamma_1R_1\) and \(R_1\) is chosen so that \(\gamma_1R_1>C(a)\).

Together with \eqref{0716-115}, (a) and (b) give \(w_n(x,t)>0\) for every \(x_n>0\).
This proves (i).

\medskip {\it \underline{4. Proof of (ii)}.}\quad
Let \(0<a\leq1\) and let \(\de_0=\de_0(n,a,M)\in(0,\frac18]\) be so small that \(1-\de_0<t<1\) implies
\begin{align}\label{0716-124}
\phi(t)=(1-t)^a\ \leq\ \frac{M}{c},
\end{align}
where \(c=c(n)\geq1\) is large enough for both \eqref{A11estimate} and the comparison below; such a \(\de_0\) exists because \(a>0\), and \(\de_0\leq\frac18\) guarantees \(t>\frac78\).
The range \(x_n>\gamma_2|x'|\) is \eqref{0716-115}.
Enlarging \(R_1\) so that \(\gamma_1R_1\geq1\) (as in Step 3(b)), we have \((0,1]\subset(0,\gamma_1|x'|]\), so that the estimates of Steps 1 and 2 are available for \(0<x_n\leq1\).
Let \(0<x_n\leq1\); then \(\min\{1,x_n^2\}=x_n^2\) and \(\min\{1,x_n\}=x_n\), so \eqref{0716-121} and \eqref{A2final} give
\begin{align*}
|A_1(x,t)|\geq \frac{c\,\norm{\psi}_{L^1(\Rn)}M x_n^2}{|x'|^{n}}, \qquad A_2(x,t)\leq \frac{C\,\norm{\psi}_{L^1(\Rn)}(1-t)^ax_n^2}{|x'|^{n}} ,
\end{align*}
using \(|x|\geq|x'|\) in the second.
By \eqref{0716-124} this yields \(A_2\leq\frac14|A_1|\), and together with \eqref{0716-121} and \eqref{0313-01},
\begin{align*}
w_n(x,t)=A_1+A_2+A_3\leq A_1+\tfrac14|A_1|+\tfrac18A_2\leq-\tfrac12|A_1(x,t)|<0 .
\end{align*}
Thus \(w_n<0\) on \((0,1]\) and \(w_n>0\) on \((\gamma_2|x'|,\infty)\), which is \eqref{0716-101}.
By Definition \ref{1202} this means \((1,\gamma_2|x'|)\in S^{-+}(w_n(x',\cdot\,,t);0,\infty)\), and Remark \ref{1125} yields \(x_{nn}^*\in[1,\gamma_2|x'|]\) with \(w_n(x',x_{nn}^*,t)=0\).

\medskip {\it \underline{5. Proof of (iii)}.}\quad
It remains to extend the negativity of Step 4 from \(x_n\leq1\) to \(x_n<\gamma_1|x'|\).
For \(1\leq x_n<\gamma_1|x'|\) we have \(\min\{1,x_n^2\}=\min\{1,x_n\}=1\), so \eqref{0716-121} and \eqref{A2final} give
\begin{align*}
|A_1(x,t)|\geq\frac{c\,\norm{\psi}_{L^1(\Rn)}M}{|x'|^{n}}, \qquad A_2(x,t)\leq\frac{C\,\norm{\psi}_{L^1(\Rn)}(1-t)^ax_n}{|x'|^{n}} \leq\frac{C\,\gamma_1\norm{\psi}_{L^1(\Rn)}(1-t)^a}{|x'|^{n-1}} ,
\end{align*}
and therefore \(A_2\leq\frac{C\gamma_1}{cM}(1-t)^a|x'|\,|A_1|\leq\frac14|A_1|\) as soon as \eqref{0716-102} holds with \(\eta_0:=\frac{cM}{4C\gamma_1}\).
Arguing as in Step 4 we obtain \(w_n<0\) on \((0,\gamma_1|x'|)\), and with \eqref{0716-115} this gives \((\gamma_1|x'|,\gamma_2|x'|)\in S^{-+}(w_n(x',\cdot\,,t);0,\infty)\).
By Remark \ref{1125} there is \(x_{nn}^*\in[\gamma_1|x'|,\gamma_2|x'|]\) with \(w_n(x',x_{nn}^*,t)=0\), that is, \(x_{nn}^*\approx|x'|\).
This completes the proof.
\end{proof}

We now prove Theorem \ref{thmasymptotic} and Corollary \ref{cor-sign-changes}.

\begin{proof}[{\it Proof of Theorem \ref{thmasymptotic}}]
Let \(c_*\) be the largest of the constants \(c_*\) of Propositions \ref{1stmainprop}, \ref{2ndmainprop} and \ref{3rdmainprop}, let \(R_*:=\max\{c_*,R_1\}\) with \(R_1\) as in Proposition \ref{0716}, and let \(\tau_0:=\epsilon_2\) if \(-1<a<0\) and \(\tau_0:=\min\{\epsilon_2,\de_0\}\) if \(0<a\leq1\).

{\it Item (i)(a).}\quad
The pair \((x_{ni2}^*,x_{ni1}^*)\) is provided by Proposition \ref{1stmainprop}, and it remains to show that \(x_{nik}^*\approx Z_a(x',t)\).
If \(t\geq\frac98\), this is Proposition \ref{1stmainprop}(i).
If \(-\frac12<a\leq1\) and \(1<t<\frac98\), then \((x',t)\in A_a^3\) by the proof of Proposition \ref{1stmainprop}(iii), so that Proposition \ref{1stmainprop}(ii)(d) and (ii)(e) give \(x_{nik}^*\approx(\ln|x'|)^{\frac12}\), and the final assertion of that proposition gives \(x_{ni2}^*>\frac12\).

Let now \(-1<a\leq-\frac12\) and \(1<t<\frac98\), and write \(X:=X_a(x',t)\) and \(\beta:=\frac{1}{1-2a}\).
Since \(\min\{X^\beta,(\ln(e+X))^{\frac12}\}\) is comparable to \(X^\beta\) for \(X\lesssim1\) and to \((\ln(e+X))^{\frac12}\) for \(X\gtrsim1\), it suffices to consider the following three cases.
If \((x',t)\in A_a^1\cup A_a^2\), then \(X\lesssim1\) by \eqref{0804-2}--\eqref{0804-3}, and Proposition \ref{1stmainprop}(ii)(a)--(c) give \(x_{nik}^*\approx X^\beta\); for \(a=-\frac12\) we use that \(\ln\frac{|x'|^2}{t-1}\approx|\ln(t-1)|\) on \(A_a^2\).
If \((x',t)\in A_a^3\), then \(X\gtrsim1\) by \eqref{0804-4}, and items (ii)(a)--(c) give \(x_{nik}^*\approx(\ln Y)^{\frac12}\) for an expression \(Y\approx X\); since \(x_{nik}^*>\frac12\), \(\ln Y\) is bounded below by a positive constant, so that \(\ln Y\approx\ln(e+X)\).
Otherwise, \(X\approx1\) and \(x_{nik}^*\approx1\) by Proposition \ref{1stmainprop}(iii) and its proof, where \(\rho\approx X\).
Hence \(x_{nik}^*\approx Z_a(x',t)\) in all cases.

{\it Item (i)(b).}\quad
For \(t>1\), this is Proposition \ref{2ndmainprop}.
If \(0<a\leq1\), then \(\phi(1)=0\), so that \(w_n^N(\cdot\,,1)=0\) by \(\eqref{wis}_3\), and the integral in \eqref{est-wn} is finite at \(t=1\); hence the proofs of Lemma \ref{w_nformula} and Proposition \ref{2ndmainprop} apply verbatim at \(t=1\).

{\it Item (ii).}\quad
In (a), the assertion for \(w_i\) follows from Proposition \ref{3rdmainprop}(i), and the assertion for \(w_n\) follows from Proposition \ref{0716}(i).
In (b), the pair \((x_{ni2}^*,x_{ni1}^*)\) is provided by Proposition \ref{3rdmainprop}(ii), and \(x_{ni2}^*\geq\frac12\) by \eqref{0107-5}.
If \((x',t)\in C_a^1\), then \(a|\ln(1-t)|\geq\ln|x'|\) for \(|x'|\) large, so that \(\min\{\ln|x'|,|\ln(1-t)|\}=\ln|x'|\); if \((x',t)\notin C_a^1\), then \((x',t)\in C_a^2\) and \(|\ln(1-t)|\leq c\ln|x'|\) by the proof of Proposition \ref{3rdmainprop}, so that \(\min\{\ln|x'|,|\ln(1-t)|\}\approx|\ln(1-t)|\).
In both cases Proposition \ref{3rdmainprop}(ii) gives \(x_{nik}^*\approx Z_a(x',t)\).
The statements on \(w_n\) are Proposition \ref{0716}(ii) and (iii).
\end{proof}

\begin{proof}[{\it Proof of Corollary \ref{cor-sign-changes}}]
Put \(\tau:=\max\{t,1\}\) and \(s:=|x'|^2/\tau\), so that \(s>c_*^2\) in (i) and \(s>R_*^2\) in (ii).
By \eqref{Xa}--\eqref{Za}, \(Z_a(x',t)\lesssim(\tau\ln s)^{\frac12}\) in all cases, and therefore, by Theorem \ref{thmasymptotic},
\begin{align*}
\frac{x_{ni1}^*}{|x'|}\lesssim\Bigl(\frac{\ln s}{s}\Bigr)^{\frac12}<\min\{\mu_1,\gamma_1\}
\end{align*}
after enlarging \(c_*\) and \(R_*\).
In (i), Theorem \ref{thmasymptotic}(i) gives \(w_n<0\) on \((0,\mu_1|x'|)\supset(0,x_{ni2}^*)\), \(w_i>0\) on \((x_{ni1}^*,\infty)\supset(\mu_2|x'|,\infty)\), and the signs of \(w_i\) below \(x_{ni2}^*\) and of \(w_n\) above \(\mu_2|x'|\); this gives the three assertions.
Item (ii) follows in the same way from Theorem \ref{thmasymptotic}(ii)(b).
\end{proof}

We next record a criterion for reversal points.

\begin{lemma}\label{rp-criterion}
Let \(1\leq i\leq n\), let \(x'\in\Rn\) and \(t_*>0\) be fixed, and suppose there are \(0<p<q\) and \(\de_1>0\) such that, for every \(t\in(t_*-\de_1,t_*+\de_1)\), the map \(x_n\mapsto w_i(x',x_n,t)\) is continuous on \([p,q]\) and
\begin{align}\label{rp-crit-hyp}
w_i(x',p,t)\,w_i(x',q,t)<0 .
\end{align}
Then there is \(x_n^*\in(p,q)\) such that \((x',x_n^*,t_*)\) is a reversal point of \(w_i\) in the sense of Definition \ref{reversalpoint}.
\end{lemma}

\begin{proof}
By \eqref{rp-crit-hyp}, the continuity assumption and the intermediate value theorem, for each \(t\in(t_*-\de_1,t_*+\de_1)\) the set \(Z(t):=\{x_n\in(p,q)\mid w_i(x',x_n,t)=0\}\) is nonempty; choose \(\zeta(t)\in Z(t)\) and set \(x_n^*:=\zeta(t_*)\in(p,q)\).
Take the constant functions \(f_1\equiv p\), \(f_2\equiv q\), and set \(f_3:=\zeta\).
By construction \(f_3(t_*)=x_n^*\), \(w_i(x',f_3(t),t)=0\), \(f_1(t)<f_3(t)<f_2(t)\) and \(w_i(x',f_1(t),t)w_i(x',f_2(t),t)<0\) for every such \(t\).
Thus the conditions of Definition~\ref{reversalpoint} are satisfied.
\end{proof}

The continuity assumption in Lemma \ref{rp-criterion} is satisfied in all applications below.
Indeed, fix \(x'\), \(t\) and \(0<p<q\), with \(x'\) and \(t\) as in the corresponding application.
The term \(w_i^N(x',\cdot\,,t)\) is continuous by \(\eqref{wis}_3\).
For \(u>0\) the kernels \(L_{ij}(\cdot\,,u)\) and \(B_{in}(\cdot\,,u)\) are continuous away from the origin (see Section \ref{sect3}), and by \eqref{240707}, \eqref{L^Lestimate}, \eqref{1101-01} and Remark \ref{rmk-shift},
\begin{align*}
\Bigl|\sum_{k=1}^{n-1}L_{kk}(x-y',u)\Bigr|+|L_{ni}(x-y',u)|+|B_{in}(x-y',u)|\leq C\,u^{-\frac12}
\end{align*}
for all \(p\leq x_n\leq q\), \(y'\in\supp\psi\) and \(0<u\leq t\), where \(C\) depends on \(x'\), \(p\), \(q\) and \(t\), but not on \(x_n\) or \(u\).
Since
\begin{align*}
\int_0^{\min\{t,1\}}\frac{\phi(s)}{\sqrt{t-s}}\,ds<\infty
\end{align*}
for every \(t\neq1\), and also for \(t=1\) if \(a>0\), the dominated convergence theorem applied to \eqref{wis} shows that \(x_n\mapsto w_i(x',x_n,t)\), \(1\leq i\leq n\), is continuous on \([p,q]\).
The case \(t=1\) is considered below only when \(a>0\).

\begin{proof}[{\it Proof of Theorem \ref{rp-flow}}]
Let \(c_*\), \(R_*\), \(\tau_0\), \(\mu_1\leq\mu_2\), \(\gamma_1\leq\gamma_2\) and \(\eta_0\) be as in Theorem \ref{thmasymptotic}; we assume \(c_*\geq R_*\), and \(c_*\) is enlarged finitely many times below.
Whenever we consider a tangential component, we assume \(x_i\geq\frac12|x'|\).
In each case we find heights \(0<p<q\), independent of \(t\) on an interval \((t_*-\de_1,t_*+\de_1)\), such that \eqref{rp-crit-hyp} holds there, and conclude by Lemma \ref{rp-criterion}.

\medskip {\bf \(\bullet\) (Case 1: \(i<n\), \(t>1\).)}\quad
Let \(t_*>1\) and \(\de_1:=\frac12\min\{t_*-1,1\}\).
Since \(t_*+\de_1\leq\frac32t_*\), after enlarging \(c_*\) by the factor \(\sqrt{3/2}\) the hypothesis \(|x'|>c_*(1+\sqrt{t_*})\) gives \(|x'|>c_*(1+\sqrt t)\) for every \(t\in(t_*-\de_1,t_*+\de_1)\).
For these \(t\) we have \(\frac12(t_*-1)<t-1<\frac32(t_*-1)\) and \(\frac12t_*<t<\frac32t_*\), and if \(t_*<\frac98\), then also \(|\ln(t-1)|\approx|\ln(t_*-1)|\); hence \eqref{Xa}--\eqref{Za} give \(Z_a(x',t)\approx Z_a(x',t_*)\), where, if \(t\) and \(t_*\) lie on different sides of \(\frac98\), both quantities are \(\approx(\ln|x'|)^{\frac12}\).
By Theorem \ref{thmasymptotic}(i)(a), for each such \(t\) there are \(x_{ni2}^*(t)\leq x_{ni1}^*(t)\), both comparable to \(Z_a(x',t_*)\), with \(w_i(x',\cdot\,,t)<0\) on \((0,x_{ni2}^*(t))\) and \(w_i(x',\cdot\,,t)>0\) on \((x_{ni1}^*(t),\infty)\).
Put
\begin{align*}
p:=\tfrac12\inf_{|t-t_*|<\de_1}x_{ni2}^*(t),\qquad q:=2\sup_{|t-t_*|<\de_1}x_{ni1}^*(t).
\end{align*}
Then \(p\approx q\approx Z_a(x',t_*)\) and \(w_i(x',p,t)<0<w_i(x',q,t)\) for \(|t-t_*|<\de_1\), and Lemma \ref{rp-criterion} gives a reversal point \((x',x_n^*,t_*)\) of \(w_i\) with \(x_n^*\in(p,q)\), so that \(x_n^*\approx Z_a(x',t_*)\).
This proves (ii).

\medskip {\bf \(\bullet\) (Case 2: \(i=n\), \(t>1\).)}\quad
Let \(t_*>1\) and \(\de_1:=\frac12\min\{t_*-1,1\}\), and let \(|x'|>c_*(1+\sqrt{t_*})\), so that \(|x'|>c_*(1+\sqrt t)\) for \(|t-t_*|<\de_1\), as in Case 1.
Put \(p:=\frac12\mu_1|x'|\) and \(q:=2\mu_2|x'|\).
By Theorem \ref{thmasymptotic}(i)(b), \(w_n(x',p,t)<0<w_n(x',q,t)\) for \(|t-t_*|<\de_1\), and Lemma \ref{rp-criterion} gives a reversal point \((x',x_n^*,t_*)\) of \(w_n\) with \(x_n^*\in(p,q)\), that is, \(x_n^*\approx|x'|\).
This is the first assertion of (i); it holds for every \(a\in(-1,1]\).

\medskip {\bf \(\bullet\) (Case 3: \(i<n\), \(t<1\).)}\quad
Let \(0<a\leq1\), \(t_*\in(1-\tau_0,1)\) and \(\de_1:=\frac12\min\{1-t_*,\,t_*-1+\tau_0\}\), and let \(|x'|>c_*\).
For \(|t-t_*|<\de_1\) we have \(1-\tau_0<t<1\) and \(\frac12(1-t_*)<1-t<\frac32(1-t_*)\); since \(1-t_*<\frac18\), this gives \(|\ln(1-t)|\approx|\ln(1-t_*)|\), and hence \(Z_a(x',t)\approx Z_a(x',t_*)\).
By Theorem \ref{thmasymptotic}(ii)(b), for each such \(t\) there are \(x_{ni2}^*(t)\leq x_{ni1}^*(t)\), both comparable to \(Z_a(x',t_*)\), bracketing the sign change of \(w_i(x',\cdot\,,t)\) as in Case 1.
With \(p\) and \(q\) defined as in Case 1, Lemma \ref{rp-criterion} yields a reversal point \((x',x_n^*,t_*)\) of \(w_i\) with \(x_n^*\approx Z_a(x',t_*)\).
This is the first assertion of (iii).

If \(-1<a<0\), then Theorem \ref{thmasymptotic}(ii)(a) gives \(w_i(x,t)>0\) for every \(x_n>0\) and every \(1-\tau_0<t<1\); so \(w_i(x',\cdot\,,t)\) has no zero and the conditions of Definition \ref{reversalpoint} cannot hold.
This is the second assertion of (iii).

\medskip {\bf \(\bullet\) (Case 4: \(i=n\), \(t\leq1\).)}\quad
Let \(0<a\leq1\), \(t_*\in(1-\tau_0,1]\) and \(\de_1:=\frac12\min\{1,\,t_*-1+\tau_0\}\), so that \((t_*-\de_1,t_*+\de_1)\subset(1-\tau_0,\frac32)\).
Enlarging \(c_*\), we may assume that, for \(|x'|>c_*\), Theorem \ref{thmasymptotic}(i)(b) applies for all \(1\leq t<\frac32\) and that \(1<\frac12\mu_1|x'|\).
Put \(p:=1\) and \(q:=2\max\{\gamma_2,\mu_2\}|x'|\).
For \(1-\tau_0<t<1\), Theorem \ref{thmasymptotic}(ii)(b) gives \(w_n(x',p,t)<0<w_n(x',q,t)\); for \(1\leq t<\frac32\), Theorem \ref{thmasymptotic}(i)(b) gives the same, since \(p<\mu_1|x'|\) and \(q>\mu_2|x'|\).
Lemma \ref{rp-criterion} yields a reversal point \((x',x_n^*,t_*)\) of \(w_n\) with \(1<x_n^*<q\).

If moreover \((1-t_*)^a|x'|<\eta_0\) (the left-hand side being \(0\) when \(t_*=1\)), then, since \(\phi\) is continuous at \(t=1\) for \(a>0\) and vanishes for \(t\geq1\), we have \(\phi(t)|x'|<\eta_0\) for \(t\) in a neighborhood of \(t_*\).
On that neighborhood, Theorem \ref{thmasymptotic}(ii)(b) gives \(w_n(x',\cdot\,,t)<0\) on \((0,\gamma_1|x'|)\) for \(t<1\), and Theorem \ref{thmasymptotic}(i)(b) gives \(w_n(x',\cdot\,,t)<0\) on \((0,\mu_1|x'|)\) for \(t\geq1\).
Replacing \(p\) by \(\frac12\min\{\gamma_1,\mu_1\}|x'|\) and \(\de_1\) by a smaller number, Lemma \ref{rp-criterion} gives a reversal point with \(x_n^*\approx|x'|\).
This proves the second assertion of (i).

Finally, if \(-1<a<0\) and \(\phi(s)\leq2^{-a}\) for \(0<s<\frac12\), Theorem \ref{thmasymptotic}(ii)(a) gives \(w_n(x,t)>0\) for every \(x_n>0\) and every \(\frac12<t<1\), so \(w_n(x',\cdot\,,t)\) has no zero for \(\frac12<t<1\) and \(w_n\) has no reversal point there.
This is the third assertion of (i), and the proof is complete.
\end{proof}

\appendix
\section{Proofs of auxiliary lemmas}
\subsection{Proof of Lemma \texorpdfstring{\ref{technicallemma}}{}}\label{proofoftechnicallemma}\hfill\\

We divide the integral into two parts as follows:
\begin{align*}
I_{m,k}=\left(\int_{0}^{\frac{x_n}{2}}+\int_{\frac{x_n}{2}}^{x_n}\right)\frac{x_n-z_n}{t^{\frac{3}{2}}}e^{-\frac{(x_n-z_n)^2}{4t}}\frac{z_n^k}{(|x'|^2+z_n^2)^{\frac{m}{2}}}dz_n=:A+B.
\end{align*}
For \(B\), we have since \(\frac{x_n}{2}\leq z_n\leq x_n\),
\begin{align*}
B&\approx \frac{1}{t^{\frac{3}{2}}}\frac{x_n^k}{|x|^m}\int_{\frac{x_n}{2}}^{x_n}(x_n-z_n)e^{-\frac{(x_n-z_n)^2}{4t}}dz_n\approx \frac{1}{t^{\frac{3}{2}}}\frac{x_n^k}{|x|^m}t(1-e^{-\frac{cx_n^2}{t}}) \approx \frac{1}{t^{\frac{1}{2}}}\frac{x_n^k}{|x|^m}\min\left\{1,x_n^2t^{-1}\right\},
\end{align*}
where we have used the equivalences \(1-e^{-\theta}\approx \min\left\{1,\theta\right\}\) for \(\theta\geq 0\) and \(\min\left\{1,c\theta\right\}\approx \min\left\{1,\theta\right\}\) for \(c,\theta\geq0\).
If \(2|x'|\geq x_n\), using the inequality \(\theta e^{-\theta}\lesssim \min\left\{1,\theta\right\}\) for \(\theta\geq 0\), we obtain
\begin{align*}
A&\leq \frac{x_n}{t^{\frac{3}{2}}}e^{-\frac{x_n^2}{16t}}\frac{x_n^{k+1}}{|x'|^m}=\frac{x_n^k}{t^{\frac{1}{2}}}\frac{1}{|x'|^m}\frac{x_n^2}{t}e^{-\frac{x_n^2}{16t}}\lesssim \frac{x_n^k}{t^{\frac{1}{2}}|x'|^m}\min\left\{1,x_n^{2}t^{-1}\right\}.
\end{align*}
If \(2|x'|\leq x_n\), using \(\sqrt{t}\leq |x'|\) and the inequality \(\theta^{\frac{m+2}{2}}e^{-\theta}\lesssim e^{-\frac{\theta}{2}}\) for \(\theta\geq0\), we have
\begin{align*}
A&\leq \frac{x_n}{t^{\frac{3}{2}}}e^{-\frac{x_n^2}{16t}}\frac{x_n^{k+1}}{|x'|^m}\leq \frac{x_n}{t^{\frac{3}{2}}}e^{-\frac{x_n^2}{16t}}\frac{x_n^{k+1}}{t^{\frac{m}{2}}} \lesssim \frac{x_n^k}{t^{\frac{1}{2}}x_n^m}\frac{x_n^2}{t}e^{-\frac{x_n^2}{32t}}\lesssim \frac{1}{x_n^{m-k}t^{\frac{1}{2}}}\min\left\{1,x_n^2t^{-1}\right\}.
\end{align*}
This completes the proof.

\subsection{Proof of Lemma \ref{lemma0709-1}}
\label{proofoflemma}\hfill\\

\(\bullet\)\,\,{\underline{\bf Proof of \eqref{0515-1}}} \quad Throughout the proof, the constant \(c\) depends only on \(n\) and may change from line to line.
We record the following useful estimates on the derivatives of \(N\).
For \(x\neq 0\),
\begin{align}\label{0105-1}
\begin{split}
&|  \partial_{x_n} N (x)| \leq c \frac{ x_n  }{|x|^{n }}, \quad  |\nabla_{x'} \partial_{x_n} N (x)| \leq c \frac{ x_n |x'| }{|x|^{n+2}}, \quad |\nabla^2_{x'} \partial_{x_n} N (x)| \leq c \frac{ x_n}{|x|^{n+2}},\quad  |\nabla^3_{x'} \partial_{x_n} N (x)| \leq c \frac{ x_n|x'|}{|x|^{n+4}}.
\end{split}
\end{align}
For fixed \(x' \in \Rn \setminus \{ 0\}\) and \( x_n > 0\), we divide \(\Rn\) into the three sets \(D_1, D_2\) and \(D_3\) defined by
\begin{align*}
D_1=\{z'\in\Rn: |x'-z'| \leq \frac1{10} |x'| \},
\end{align*}
\begin{align*}
D_2=\{z'\in\Rn: |z'| \leq \frac1{10} |x'| \}, \qquad D_3=\Rn\setminus (D_1\cup D_2).
\end{align*}
We split the integral on the left-hand side of \eqref{0515-1} into three terms:
\begin{align}\label{0730-21}
\begin{split}
\mbox{LHS of \eqref{0515-1}}    = \int_{D_1}\cdots  dz'+ \int_{D_2} \cdots dz' + \int_{D_3}\cdots dz'  =: I_1(x,t) + I_2(x,t) +I_3(x,t).
\end{split}
\end{align}
For \(I_1\), we use \(\int_{\Rn} \Ga'(z',t) dz' =1\) for all \(t>0\), and integration by parts to decompose \(I_1\) as follows.

\begin{align*}
\begin{split}
I_{1} (x,t)  & =    \int_{\{| x' - z'| \leq    \frac1{10}  |x'| \}}   \Ga'(x' - z',t)  \partial_{z_i}   \partial_{x_n} N(z',x_n) dz'\\
&\quad +  \int_{\{| x' - z'| =     \frac1{10}  |x'| \}}   \Ga'(x' - z',t) \frac{x_i -z_i}{|x' -z'|}  \partial_{x_n} N(z',x_n) d\sigma(z')\\
& =    \int_{\{| x' - z'| \leq    \frac1{10}  |x'| \}}   \Ga'(x' - z',t) \Big(    \partial_{z_i}   \partial_{x_n} N(z',x_n) -     \partial_{x_i}   \partial_{x_n}N(x) \Big)dz'\\
&\quad + \partial_{x_n} \partial_{x_i} N(x)   - \partial_{x_n} \partial_{x_i} N(x) \int_{\{|z'| \geq  \frac1{10}\frac{ |x'|}{\sqrt{t}}\}}  \Ga'(z',1)  dz'\\
&\quad +   \int_{\{| x' - z'|  =     \frac1{10}  |x'| \}}   \Ga'(x' - z',t) \frac{x_i -z_i}{|x' -z'|}  \partial_{x_n} N(z',x_n) d\sigma(z')\\
& =: I_{11}(x,t) +\partial_{x_n} \partial_{x_i} N(x) + I_{12}(x,t) + I_{13}(x,t).
\end{split}
\end{align*}

From \(\eqref{0105-1}_2\), we have
\begin{align}\label{estI12}
\begin{split}
|I_{12} (x,t)|&  \leq c\frac{x_n |x'| }{|x|^{n+2} }   \int_{\{|z'| \geq  \frac1{10}\frac{ |x'|}{\sqrt{t}}\}}    e^{-\frac14|z'|^2}   dz' \le c \frac{x_n |x'| }{|x|^{n+2} } e^{-\frac{c|x'|^2}{t}}\leq c\frac{tx_n}{|x|^{n+2}|x'|},
\end{split}
\end{align}
where we used \(|x'|\geq 1\) and \( e^{-A} \leq c_m A^{-m}\) for \( A>0\) and \( m \in {\mathbb N} \cup \{ 0 \}\).

For \(I_{13}\), using \( \int_{\{| x' - z'| = \frac1{10} |x'| \}} \Ga'(x' - z',t) \frac{x_i -z_i}{|x' -z'|} d \si(z') =0\) and \(\Ga'(x'-z',t) \leq ct^{-\frac{n-1}2} e^{-\frac{|x'|^2}{400t}}\) for \(|x' -z'| =\frac1{10} |x'|\), we have
\begin{align*}
I_{13} (x,t)&=     \int_{\{| x' - z'|  =     \frac1{10}  |x'| \}}   \Ga'(x' - z',t) \frac{x_i -z_i}{|x' -z'|} \Big( \partial_{x_n} N(z',x_n)  - \pa_{x_n} N(x',x_n) \Big)d\sigma(z').
\end{align*}
Using the mean-value theorem, for each \(z' \in\pa D_1\), there exists \(\xi' \in D_1\) on the line segment joining \(x' \) and \(z' \) such that from \(\eqref{0105-1}_2\),
\begin{align*}
|\partial_{x_n}    N(z',x_n) - \partial_{x_n}     N(x)|  =  |\na' \partial_{x_n}    N  ( \xi', x_n) \cdot (x' -z')| \leq  c \frac{ x_n |x'|^2 }{|x|^{n+2}},
\end{align*}
where we have used \(|\xi'|\approx |x'|\) for such \(\xi'\).
Then we find that
\begin{align}\label{estI131}
|I_{13}(x,t)|&  \leq c \frac{ x_n |x'|^n }{|x|^{n+2}} t^{-\frac{n-1}2} e^{-\frac{|x'|^2}{400t}}   \leq c\frac{tx_n}{|x|^{n+2}|x'|}.
\end{align}

For \(I_{11}\), we first note that by the mean-value theorem, there exists \(\xi' \in D_1\) on the line segment joining \(x' \in D_1\) and \(z' \in D_1\) such that
\begin{align*}
\partial_{x_n}    \partial_{z_i} N(z',x_n) - \partial_{x_n}   \partial_{z_i} N(x) =   \na' \partial_{x_n}   \partial_{z_i} N  ( \xi', x_n) \cdot (z' -x').
\end{align*}

Then using \( \int_{|x'-z'| \leq \frac1{10}|x'|} \Gamma'(x' -z',t) (x'-z') dz' =0\) for all \( x' \in {\mathbb R}^{n-1}\) and \(t > 0\), we have
\begin{align*}
I_{11} (x,t)& =  \int_{|x'-z'| \leq \frac1{10}|x'|}  \Gamma'(x' -z', t)  \na' \partial_{x_n}   \partial_{z_i} N  ( \xi', x_n) \cdot (z' -x')   dz'\\
& =  \int_{|x'-z'| \leq \frac1{10}|x'|} \Gamma'(x' -z', t)   \Big( \na' \partial_{x_n}   \partial_{z_i} N  ( \xi', x_n) - \na' \partial_{x_n}   \partial_{z_i} N  ( x', x_n) \Big)   \cdot (z' -x') dz'\\
& =  \int_{|x'-z'| \leq \frac1{10}|x'|} \Gamma'(x' -z', t)    (\xi' - x') (\na')^2 \partial_{x_n}   \partial_{z_i} N  ( \eta', x_n)    (z' -x') dz',
\end{align*}
where in the final line we have used the mean-value theorem again with \(\eta'\) lying on the line segment joining \(x'\) and \(\xi'\) so that in particular, \(\eta' \in D_1\).
Here and henceforth, for vector-valued maps, the mean-value theorem is applied componentwise, so the intermediate point may depend on the component. However, since all the bounds below are uniform over the segment, this does not affect them.

Noting that \( |\xi'|, |\eta'|\approx |x'|\) and \( |\xi' - x'| \leq |x' -z'|\) for \(\xi' \in D_1\), it follows from \(\eqref{0105-1}_4\) that
\begin{align*}
|  (\xi' - x') (\na')^2 \partial_{x_n}   \partial_{z_i} N  ( \eta', x_n)    (x' -z')   | \leq c\frac{x_n|\eta'| |x' -z'|^2}{(|\eta' |^2 + x_n^2)^{\frac{n+4}2}  }    \leq c\frac{x_n|x'||x' -z'|^2}{|x|^{n+4}  },
\end{align*}
which implies that
\begin{align}\label{estI_11}
\begin{split}
|I_{11} (x,t)| &  \leq c\frac{x_n|x'| }{|x|^{n+4} }    t \int_{|z'| \leq  \frac{ |x'|}{10\sqrt{t}}}    |z'|^2  e^{-c|z'|^2}  dz'\le c\frac{x_n|x'|}{|x|^{n+4} }t \leq c\frac{tx_n}{|x|^{n+2}|x'|}.
\end{split}
\end{align}
Combining the estimates \eqref{estI12}, \eqref{estI131} and \eqref{estI_11} gives
\begin{align}\label{estI_1}
|I_{11}(x,t)+I_{12}(x,t)+I_{13}(x,t)|\leq c\frac{tx_n}{|x|^{n+2}|x'|}.
\end{align}
This finishes the estimate of \(I_1\).

For the estimate of \(I_2+I_3\), we consider two cases: \(|x'|\leq x_n\) and \(|x'|\geq x_n\).
If \(|x'|\leq x_n\), then, integrating by parts,
\begin{align*}
\begin{split}
I_{2}(x,t) + I_3 (x,t)  &= \int_{\{|x'-z'|\geq \frac{1}{10}|x'|\}}\partial_{x_i}\Gamma'(x'-z',t)\partial_{x_n}N(z',x_n)dz'\\
& =    \int_{\{| x' - z'| \geq    \frac1{10}  |x'| \}}   \Ga'(x' - z',t)       \pa_{z_i} \partial_{x_n} N(z',x_n) dz'\\
&\quad -   \int_{\{| x' - z'| =     \frac1{10}  |x'| \}}   \Ga'(x' - z',t) \frac{x_i -z_i}{|x' -z'|}  \partial_{x_n} N(z',x_n) d\sigma(z')=:A(x,t)+B(x,t).
\end{split}
\end{align*}
Since \(B=-I_{13}\), it follows from \eqref{estI131} that
\begin{align}\label{0317-B}
|B(x,t)| &     \leq c\frac{tx_n}{|x|^{n+2}|x'|}.
\end{align}
On the other hand, for \(A\), using the fact that \(|z'|\leq c|x'-z'|\) and \(\Ga'(x'-z',t) \leq c t^{-\frac{n-1}{2}} e^{-\frac{c|z'|^2}{t}}\) for \( |x'-z'|\geq \frac1{10} |x'|\), we have
\begin{align}\label{0317-A}
\begin{split}
|A(x,t)|&   \leq   \int_{\{| x' - z'| \geq    \frac1{10}  |x'| \}}    \Ga'(x'-z',t) \frac{x_n|z'|}{(|z'|^2 + x_n^2)^{\frac{n+2}{2}}}dz' \\
&\leq  c \frac{x_n}{|x|^{n+2}} \int_{\{| x' - z'| \geq    \frac1{10}  |x'| \}}  |x'-z'|  \Ga'(x'-z',t) dz'\\
& \leq c\frac{x_nt^{\frac12}}{|x|^{n+2}}\int_{|z'|\geq \frac{|x'|}{10\sqrt{t}}}|z'|e^{-c|z'|^2}dz'\\
&   \leq  c \frac{x_n}{|x|^{n+2}}t^{\frac12} e^{-\frac{c|x'|^2}{t}} \leq c\frac{tx_n}{|x|^{n+2}|x'|}.
\end{split}
\end{align}
Thus using \eqref{0317-B} and \eqref{0317-A}, we find that for \(|x'|\leq x_n\),
\begin{align}\label{estI_2I_31}
|I_2(x,t)+I_3(x,t)|\leq c\frac{tx_n}{|x|^{n+2}|x'|}.
\end{align}

Next we consider the case \(|x'|\geq x_n\).
Integrating by parts, we obtain
\begin{align*}
\begin{split}
I_3(x,t)   & =    \int_{D_3}   \Ga'(x' - z',t)  \partial_{z_i}   \partial_{x_n} N(z',x_n)  dz'\\
&\quad -   \int_{\{| x' - z'| =     \frac1{10}  |x'| \}}   \Ga'(x' - z',t) \frac{x_i -z_i}{|x' -z'|}  \partial_{x_n} N(z',x_n) d\sigma(z')\\
&\quad +   \int_{\{| z'| =     \frac1{10}  |x'| \}}   \Ga'(x' - z',t) \frac{ z_i}{|z'|}  \partial_{x_n} N(z',x_n) d\sigma(z')\\
& =: I_{31}(x,t)  + I_{32} (x,t)+ I_{33}(x,t).
\end{split}
\end{align*}
We first estimate \(I_{31}\).
Using the fact that \(|z'|\leq c|x'-z'|\) and \( |\Ga'(x'-z',t)| \leq c t^{-\frac{n-1}2} e^{-\frac{c|z'|^2}{t}} \) for \( z' \in D_3\) and \(\eqref{0105-1}_2\), we obtain
\begin{align}\label{est-I311}
\begin{split}
|I_{31}(x,t)| &\leq c\frac{x_n}{t^{\frac{n-1}{2}}}\int_{|z'|\geq \frac{|x'|}{10}}e^{-\frac{c|z'|^2}{t}}\frac{|z'|}{(|z'|^2+x_n^2)^{\frac{n+2}{2}}}dz'\\
&\leq c \frac{x_nt^{\frac12}}{|x|^{n+2}}\int_{|z'|\geq \frac{|x'|}{10\sqrt{t}}}e^{-c|z'|^2}|z'|dz'\\
&\leq c\frac{t^{\frac12}x_n}{|x|^{n+2}}e^{-\frac{c|x'|^2}{t}}\leq c\frac{tx_n}{|x|^{n+2}|x'|}.
\end{split}
\end{align}
Since \(I_{32} = -I_{13}\), it follows from \eqref{estI131} that
\begin{align}\label{estI_321}
|I_{32}(x,t)| &\leq   c \frac{tx_n}{|x|^{n+2}|x'|}.
\end{align}
For \(I_{33}\), using \(\eqref{0105-1}_1\), and the fact that \(|x|\leq c|x'|\) for \(|x'|\geq x_n\), we have
\begin{align}\label{estI_331}
\begin{split}
|I_{33}(x,t)| &\leq  c \frac{x_n }{|x|^{n}}  t^{-\frac{n-1}2} \int_{\{| z'| =     \frac1{10}  |x'| \}}  e^{-\frac{|x'-z'|^2}{4t}}    d\sigma(z') \\
&\leq  c \frac{x_n|x'|^{n-2} }{|x|^{n}}  t^{-\frac{n-1}2}   e^{-\frac{|x'|^2}{8t}} \leq   c \frac{tx_n}{|x|^{n}|x'|^3}\leq c\frac{tx_n}{|x|^{ n+2}|x'|}.
\end{split}
\end{align}
Combining \eqref{est-I311}, \eqref{estI_321} and \eqref{estI_331}, we obtain that for \(|x'|\geq x_n\),
\begin{align}\label{est-I31}
|I_3(x,t)|\leq c \frac{tx_n}{|x|^{n+2}|x'|}.
\end{align}

It remains to estimate \(I_2\) for \(|x'|\geq x_n\).
Since \(\int_{\Rn} \partial_{x_n} N( z',x_n) dz' = \frac12\) for all \( x_n>0\), it follows that
\begin{align*}
I_{2} (x,t)& =  \int_{|z'|\leq \frac1{10} |x'| }    \Big(   \partial_{x_i}  \Ga'(x'-z',t) -  \partial_{x_i}  \Ga'(x',t) \Big) \partial_{x_n}  N( z',x_n) dz'\\
&  \quad  +    \frac12  \partial_{x_i}  \Ga'(x',t)  -  \partial_{x_i}  \Ga'(x',t)\int_{|z'|\geq \frac1{10} |x'| }    \partial_{x_n}  N( z',x_n)    dz'\\
& =:  I_{21}  (x,t) + \frac12  \partial_{x_i}  \Ga'(x',t)   + I_{22} (x,t).
\end{align*}
For \(I_{21}\), by the mean-value theorem, there exists \(\xi'\) on the line segment joining \(0\) and \(z'\) such that
\begin{align*}
I_{21} (x,t) =  -\int_{|z'|\leq \frac1{10} |x'| }   \partial_{x_n}  N( z',x_n)  \na' \partial_{x_i}  \Ga'(x'-\xi',t) \cdot z'  dz'.
\end{align*}

Using \(\int_{|z'|\leq \frac1{10} |x'| } z' \partial_{x_n}N(z',x_n) dz' =0 \), we find that again by the mean-value theorem and \(\eqref{0105-1}_1\),
\begin{align}\label{estI_211}
\begin{split}
|I_{21} (x,t) | & =    \left|\int_{|z'|\leq \frac1{10} |x'| }   \partial_{x_n}  N( z',x_n)  \Big( \na' \partial_{x_i}  \Ga'(x'-\xi',t) - \na' \partial_{x_i}  \Ga'(x',t) \Big) \cdot z'  dz'  \right|\\
& =   \left| \int_{|z'|\leq \frac1{10} |x'| }   \partial_{x_n}  N( z',x_n)  z'\cdot \na'^2 \partial_{x_i}  \Ga'(x'-\eta',t)\xi'  dz'\right|\\
& \leq c     \int_{|z'|\leq \frac1{10} |x'| }  \frac{x_n |z'|^2}{ ( | z'|^2 + x_n^2)^{\frac{n}2}}   t^{-\frac{n+2}{2}} e^{-\frac{c|x'|^2}{t}} dz'\\
&  =  c  x_n^2  t^{-\frac{n+2}{2}} e^{-\frac{c|x'|^2}{t}}  \int_{|z'|\leq \frac{|x'|}{10 x_n} }  \frac{ |z'|^2}{ ( | z'|^2 + 1)^{\frac{n}2}}   dz'\\
& \leq c t^{-\frac{n+2}{2}}x_n|x'|e^{-\frac{c|x'|^2}{t}}\leq c\frac{ tx_n}{|x'|^{n+3}}\leq c \frac{tx_n}{|x|^{n+2}|x'|},
\end{split}
\end{align}
where we have used the estimates \(|\nabla'^2 \partial_i\Gamma'(x',t)|\leq c t^{-\frac{n+2}{2}}e^{-\frac{c|x'|^2}{t}}\) and
\begin{align*}
\int_{|z'|\leq R}\frac{|z'|^2}{(|z'|^2+1)^{\frac{n}{2}}}dz'\leq c R\min\{1, R\}
\end{align*}
for \(R>0\).

For \(I_{22}\), using \(\eqref{0105-1}_1\), we find that
\begin{align}\label{estI212}
\begin{split}
|I_{22} (x,t) |& \leq c|x'| t^{-\frac{n+1}2} e^{-\frac{|x'|^2}{4t}} \int_{|z'|\geq \frac1{10} \frac{|x'|}{x_n} }    (|z'|^2 +1) ^{-\frac{n}2}   dz' \\
&\leq c t^{-\frac{n+1}2} e^{-\frac{|x'|^2}{4t}} x_n\leq c\frac{t x_n}{|x'|^{n+3}} \leq c \frac{tx_n}{|x|^{n+2}|x'|},
\end{split}
\end{align}
where we have used the estimate
\begin{align*}
\int_{|z'|\geq R} \frac{1}{(|z'|^2+1)^{\frac{n}{2}}}dz'\leq c\min\{1, R^{-1}\},
\end{align*}
for \(R>0\).

Hence, for \( x_n \leq |x'|\), we have from \eqref{estI_211} and \eqref{estI212} that
\begin{align}\label{0101-31}
\left|I_2 (x,t)- \frac12  \partial_{x_i}  \Ga'(x',t)   \right| \leq c \frac{t x_n}{|x|^{n+2}|x'|}.
\end{align}
Setting \(J_1= I_{11} + I_{12}+ I_{13}+ I_2 + I_3-\frac12\partial_{x_i}\Ga'(x',t)\mathbf{1}_{x_n\leq|x'|}\) and by \eqref{0730-21} and the estimates \eqref{estI_1}, \eqref{estI_2I_31}, \eqref{est-I31} and \eqref{0101-31}, we obtain Lemma \ref{lemma0709-1}(i).

\(\bullet\)\,\,{{\underline{\bf Proof of \eqref{0306-1} }}}\quad We record the following useful estimates on the derivatives of \(N\).
For \(1\leq i, j\leq n-1\) and \(x\neq 0\),
\begin{align}\label{0105-1-1}
|\partial_{x_j}N(x)|&\leq c\frac{|x'|}{|x|^n}, &|\partial_{x_i}\partial_{x_j}N(x)|&\leq c\frac{|x'|^2+\de_{ij}x_n^2}{|x|^{n+2}},\nonumber\\
|\na'\partial_{x_i}\partial_{x_j}N(x)|&\leq c\frac{|x'|}{|x|^{n+2}}, &|\nabla'^2\partial_{x_i}\partial_{x_j}N(x)|&\leq c|x|^{-n-2}.
\end{align}
As in the proof of \eqref{0515-1}, we split the integral on the left-hand side of \eqref{0306-1} into three terms:
\begin{align}\label{0730-2}
\begin{split}
\mbox{LHS of \eqref{0306-1}}    = \int_{D_1}\cdots  dz'+ \int_{D_2} \cdots dz' + \int_{D_3}\cdots dz'  =: I_1 + I_2 +I_3.
\end{split}
\end{align}
As in the proof of \eqref{0515-1}, we decompose \(I_1\) as follows.
\begin{align*}
\begin{split}
I_{1} (x,t) &= \int_{\{|x'-z'|\leq \frac{1}{10}|x'|\}}\Gamma'(x'-z',t) \left(\partial_{z_i}\partial_{z_j}N(z',x_n)-\partial_{x_i}\partial_{x_j}N(x)\right)dz'  \\
&\quad + \partial_{x_i}\partial_{x_j}N(x)-\partial_{x_i}\partial_{x_j}N(x)\int_{\{|z'|\geq \frac{1}{10}\frac{|x'|}{\sqrt{t}}\}}\Gamma'(z',1)dz'\\
&\quad + \int_{\{|x'-z'|=\frac{1}{10}|x'|\}}\Gamma'(x'-z',t)\frac{x_i-z_i}{|x'-z'|}\partial_{x_j}N(z',x_n)d\sigma(z')\\
& = :I_{11}(x,t) +\partial_{x_i} \partial_{x_j} N(x) + I_{12}(x,t) + I_{13}(x,t).
\end{split}
\end{align*}
For \(I_{12}\), it follows from \(\eqref{0105-1-1}_2\) that
\begin{align}\label{0317-10}
\begin{split}
|I_{12} (x,t)|&  \leq c\frac{|x'|^2 +\de_{ij} x_n^2}{|x|^{n+2}}    \int_{\{|z'| \geq  \frac1{10}\frac{ |x'|}{\sqrt{t}}\}}    e^{-\frac14|z'|^2}   dz' \le c \frac{|x'|^2 +\de_{ij} x_n^2}{|x|^{n+2}}  e^{-\frac{c|x'|^2}{t}}.
\end{split}
\end{align}
For \(I_{11}\), by the mean-value theorem, there exists \(\xi' \) on the line segment joining \(x'\) and \(z'\) such that
\begin{align*}
\partial_{z_i}    \partial_{z_j} N(z',x_n) - \partial_{x_i}   \partial_{x_j} N(x)  =  \na' \partial_{\xi_i}   \partial_{\xi_j} N  ( \xi', x_n) \cdot (z' -x'),
\end{align*}
so that
\begin{align*}
I_{11}(x,t)&=\int_{\{|x' -z'| \leq \frac1{10} |x'|\}}\Gamma'(x'-z',t)\na' \partial_{\xi_i}   \partial_{\xi_j} N  ( \xi', x_n) \cdot (z' -x')dz'\\
&=\int_{\{|x' -z'| \leq \frac1{10} |x'|\}}\Gamma'(x'-z',t)\left(\na' \partial_{\xi_i}   \partial_{\xi_j} N  ( \xi', x_n)-\na' \partial_{x_i}   \partial_{x_j} N  (x) \right)\cdot (z' -x')dz',
\end{align*}
where we have used that \(\int_{\{|x' -z'| \leq \frac1{10} |x'|\}}\Gamma'(x'-z',t)(x_j-z_j)dz'=0\) for \(1\leq j\leq n-1\).
Then by the mean-value theorem again, there exists \(\eta'\) on the line segment joining \(x'\) and \(\xi'\) such that
\begin{align*}
\na' \partial_{\xi_i}   \partial_{\xi_j} N  ( \xi', x_n)-\na' \partial_{x_i}   \partial_{x_j} N  (x) =\nabla'^2\partial_{\eta_i}\partial_{\eta_j}N(\eta',x_n)\cdot(\xi'-x').
\end{align*}
Then we get that
\begin{align*}
I_{11}(x,t) =\int_{\{|x' -z'| \leq \frac1{10} |x'|\}}\Gamma'(x'-z',t)\nabla'^2\partial_{\eta_i}\partial_{\eta_j}N(\eta',x_n):(\xi'-x')\otimes(z'-x')dz'.
\end{align*}
From \(\eqref{0105-1-1}_4\), we have
\begin{align}\label{0317-11}
\begin{split}
|I_{11} (x,t)| &  \leq c\frac{t}{|x|^{n+2}  }   \int_{\{|z'| \leq  \frac1{10}\frac{ |x'|}{\sqrt{t}}\}}    |z'|^2  e^{-c|z'|^2}  dz'\le c\frac{t}{|x|^{n+2}  }.
\end{split}
\end{align}

For \(I_{13}\), by the mean-value theorem, there exists \(\xi' \in D_1\) on the line segment joining \(x'\) and \(z' \in\pa D_1\) such that from \(\eqref{0105-1-1}_2\),
\begin{align*}
|\partial_{z_j}    N(z',x_n) - \partial_{z_j}     N(x)|  =  |\na' \partial_{z_j}    N  ( \xi', x_n) \cdot (x' -z')| \leq  c|x'| \frac{|x'|^2 + x_n^2 }{|x|^{n+2}}=\frac{c|x'|}{|x|^{n}},
\end{align*}
where we have used \(|x'|\approx |\xi'|\) for such \(\xi'\), and the fact that \(\na'\partial_{z_j}N\) contains \(\partial_{z_j}^2N\), so the term \(x_n^2\) is present for every \(i\).
Then, since \( \int_{\{| x' - z'| = \frac1{10} |x'| \}} \Ga'(x' - z',t) \frac{x_i -z_i}{|x' -z'|} d \si(z') =0\) and \(\Ga'(x'-z',t) \leq ct^{-\frac{n-1}2} e^{-\frac{c|x'|^2}{t}}\) for \(z'\) such that \(|x'-z'|=\frac{1}{10}|x'|\), we have
\begin{align}\label{0317-12}
\begin{split}
|I_{13}(x,t)| &\leq     \int_{\{| x' - z'|  =     \frac1{10}  |x'| \}}   \Ga'(x' - z',t) \frac{|x_i -z_i|}{|x' -z'|} \left| \partial_{z_j} N(z',x_n)  - \pa_{z_j} N(x',x_n) \right|d\sigma(z') \\
&\leq c\, \frac{|x'|^{n-1}}{|x|^{n}}\, t^{-\frac{n-1}2} e^{-\frac{c|x'|^2}{t}}\leq \frac{c}{|x|^{n}}e^{-\frac{c|x'|^2}{t}}.
\end{split}
\end{align}
Combining the estimates \eqref{0317-10}, \eqref{0317-11} and \eqref{0317-12} gives
\begin{align}\label{0317-13}
|I_{11}(x,t)+I_{12}(x,t)+I_{13}(x,t)| \leq c\frac{t}{|x|^{n+2}}+\frac{c}{|x|^{n}}e^{-\frac{c|x'|^2}{t}}.
\end{align}

Suppose first that \(|x'|\leq x_n\).
Let \(f(\zeta'):=\partial_i\partial_jN(\zeta',x_n)\).
The left-hand side of \eqref{0306-1} is \(e^{t\Delta'}f(x')\).
Since \(\|\Delta'f\|_{L^\infty(\Rn)}\leq Cx_n^{-n-2}\), the heat-semigroup identity gives
\begin{align*}
|e^{t\Delta'}f(x')-f(x')|\leq t\|\Delta'f\|_\infty\leq Ct x_n^{-n-2}\leq Ct|x|^{-n-2}.
\end{align*}
This is the required estimate for \(J_2\) when \(|x'|\leq x_n\).

Next, we consider the case \(|x'|\geq x_n\).
Using integration by parts, we have
\begin{align*}
\begin{split}
I_3(x,t)   & =    \int_{D_3}   \Ga'(x' - z',t)  \partial_{z_i}   \partial_{z_j} N(z',x_n)  dz'\\
&\quad -   \int_{\{| x' - z'| =     \frac1{10}  |x'| \}}   \Ga'(x' - z',t) \frac{x_i -z_i}{|x' -z'|}  \partial_{z_j} N(z',x_n) d\sigma(z')\\
&\quad +   \int_{\{| z'| =     \frac1{10}  |x'| \}}   \Ga'(x' - z',t) \frac{ z_i}{|z'|}  \partial_{z_j} N(z',x_n) d\sigma(z')\\
& =: I_{31}(x,t)  + I_{32}(x,t) + I_{33}(x,t).
\end{split}
\end{align*}

On the other hand, for \( z' \in D_3\), we observe \( |\Ga'(x'-z',t)| \leq c t^{-\frac{n-1}2} e^{-\frac{c|z'|^2}{t}} \), and thus, from \(\eqref{0105-1-1}_2\), we obtain
\begin{align}\label{est-J3}
\begin{split}
|I_{31} (x,t)| &   \leq   c \frac{|x'|^2 + \de_{ij} x_n^2 }{|x|^{n+2}}     \int_{|z'| \geq \frac{|x'|}{10\sqrt{t}} }  e^{-c|z'|^2} dz'  \leq c \frac{|x'|^2 + \de_{ij} x_n^2 }{|x|^{n+2}}  e^{-\frac{c|x'|^2}{t}}.
\end{split}
\end{align}
Since \(I_{32} = -I_{13}\), we have
\begin{align*}
|I_{32}| &=|I_{13}|\leq \frac{c}{|x|^{n}}e^{-\frac{c|x'|^2}{t}}.
\end{align*}
For the estimate of \(I_{33}\), we use \(|\partial_{z_j}N(z',x_n)|\leq c|z'|^{1-n}=c'|x'|^{1-n}\) and \(|x'-z'|\geq\frac{9}{10}|x'|\) on \(\{|z'|=\frac1{10}|x'|\}\), together with \(|x|\approx|x'|\) for \(|x'|\geq x_n\), to obtain
\begin{align*}
|I_{33}| \leq  c|x'|^{1-n}\, t^{-\frac{n-1}2} \int_{\{| z'| =     \frac1{10}  |x'| \}}  e^{-\frac{|x'-z'|^2}{4t}}    d\sigma(z')  \leq  c\, |x'|^{-1} t^{-\frac{n-1}2}   e^{-\frac{|x'|^2}{8t}}\leq c\frac{t}{|x|^{n+2}}.
\end{align*}

It remains to estimate \(I_2\).
Since \(\int_{|z'|\leq |x'|/10}\partial_{z_j}N(z',x_n)\,dz'=0\) by oddness,
\begin{align*}
I_2=\int_{|z'|\leq |x'|/10} \bigl(\partial_i\Gamma'(x'-z',t)-\partial_i\Gamma'(x',t)\bigr) \partial_{z_j}N(z',x_n)\,dz'.
\end{align*}
A first-order Taylor expansion gives
\begin{align*}
I_2=-\sum_{\ell<n}\partial_\ell\partial_i\Gamma'(x',t) \int_{|z'|\leq |x'|/10}z_\ell\partial_{z_j}N(z',x_n)\,dz'+R.
\end{align*}
The moment is zero for \(\ell\neq j\) by oddness, while for \(\ell=j\)
\begin{align*}
\left|\int_{|z'|\leq |x'|/10}z_j\partial_{z_j}N(z',x_n)\,dz'\right|\leq C|x'|.
\end{align*}
Thus the first-order term is exponentially small and, using \(|x'|\geq\sqrt t\), is bounded by \(Ct|x'|^{-n-2}\).
The Taylor remainder satisfies the same bound because \(\int_{|z'|\leq |x'|/10}|z'|^2|\partial_{z_j}N(z',x_n)|\,dz'\leq C|x'|^2\).
Consequently
\begin{align*}
|I_2|\leq Ct|x'|^{-n-2}\leq C\frac{t}{|x|^{n+2}}
\end{align*}
in the present region \(|x'|\geq x_n\).

Combining the preceding estimates with \eqref{0317-13}, and using that \(|x|^2\leq2|x'|^2\) in the present region, so that \(|x|^{-n}e^{-\frac{c|x'|^2}{t}}\leq 2|x'|^2|x|^{-n-2}e^{-\frac{c|x'|^2}{t}}\leq Ct|x|^{-n-2}\) (and likewise for \eqref{est-J3}), we obtain \(|J_2(x,t)|\leq Ct|x|^{-n-2}\) also for \(|x'|\geq x_n\).
This proves Lemma~\ref{lemma0709-1}(ii).

\(\bullet\)\,\, {\underline{\bf Proof of \eqref{0515-2}}}\quad Put \(k(z'):=\partial_{z_i}N(z',0)=\frac{z_i}{n\omega_n|z'|^{n}}\) (for \(n=2\) this is also correct, with \(\omega_2=\pi\)).
Then \(k\) is smooth on \(\Rn\setminus\{0\}\), odd, and
\begin{align*}
|\nabla'^{m}k(z')|\leq c|z'|^{1-n-m},\qquad m=0,1,2 .
\end{align*}
With \(D_1,D_2,D_3\) as in the proof of \eqref{0515-1} (note that \(D_1\cap D_2=\emptyset\)), we split the left-hand side of \eqref{0515-2} into \(I_1+I_2+I_3\), the principal value being needed only in \(I_2\).
We shall use repeatedly that \(|x'|\geq\sqrt t\) and \(\sup_{r>0}r^me^{-cr}<\infty\) for \(m\geq0\), so that \(t^{-\frac{m}{2}}e^{-\frac{c|x'|^2}{t}}\leq C_m|x'|^{-m}\).

{\it Estimate of \(I_2\).}
Since \(D_2\) is a ball centered at the origin and \(k\) is odd, we have \(\text{p.v.}\int_{D_2}k(z')\,dz'=0\), while \(\int_{D_2}|z'|^m|k(z')|\,dz'\leq c|x'|^m\) for \(m=1,2\).
Hence, by Taylor's formula,
\begin{align*}
I_2=\text{p.v.}\int_{D_2}\bigl(\Gamma'(x'-z',t)-\Gamma'(x',t)\bigr)k(z')\,dz' =-\nabla'\Gamma'(x',t)\cdot\int_{D_2}z'k(z')\,dz'+R_2,
\end{align*}
with \(|R_2|\leq c\sup_{\xi'\in D_2}|\nabla'^2\Gamma'(x'-\xi',t)|\int_{D_2}|z'|^2|k(z')|\,dz'\).
Since \(|x'-\xi'|\geq\frac9{10}|x'|\) for \(\xi'\in D_2\), we have \(|\nabla'\Gamma'(x',t)|\leq ct^{-\frac n2}e^{-\frac{c|x'|^2}{t}}\) and \(|\nabla'^2\Gamma'(x'-\xi',t)|\leq ct^{-\frac{n+1}{2}}e^{-\frac{c|x'|^2}{t}}\).
Therefore
\begin{align*}
|I_2|\leq c\bigl(|x'|t^{-\frac n2}+|x'|^2t^{-\frac{n+1}{2}}\bigr)e^{-\frac{c|x'|^2}{t}}\leq c\frac{t}{|x'|^{n+1}} .
\end{align*}

{\it Estimate of \(I_3\).}
On \(D_3\) we have \(|k(z')|\leq c|x'|^{1-n}\) and \(|x'-z'|\geq\frac1{10}|x'|\), so that
\begin{align*}
|I_3|\leq c|x'|^{1-n}\int_{\{|w'|\geq\frac1{10}|x'|\}}\Gamma'(w',t)\,dw'\leq c|x'|^{1-n}e^{-\frac{c|x'|^2}{t}}\leq c\frac{t}{|x'|^{n+1}} .
\end{align*}

{\it Estimate of \(I_1\).}
We write
\begin{align*}
I_1=k(x')\int_{D_1}\Gamma'(x'-z',t)\,dz'+\int_{D_1}\Gamma'(x'-z',t)\bigl(k(z')-k(x')\bigr)dz'=:I_{11}+I_{12}.
\end{align*}
Since \(\int_{\Rn}\Gamma'(w',t)\,dw'=1\), we have \(I_{11}=k(x')-k(x')\int_{\{|w'|>\frac1{10}|x'|\}}\Gamma'(w',t)\,dw'\), and the last term is bounded by \(c|x'|^{1-n}e^{-c|x'|^2/t}\leq ct|x'|^{-n-1}\).
For \(I_{12}\), Taylor's formula gives \(k(z')-k(x')=\nabla'k(x')\cdot(z'-x')+R_1(z')\) with \(|R_1(z')|\leq c|x'|^{-n-1}|z'-x'|^2\) for \(z'\in D_1\), because \(|\nabla'^2k|\leq c|x'|^{-n-1}\) on \(D_1\).
The first-order term does not contribute, since \(\Gamma'(x'-\cdot\,,t)\) is radially symmetric about \(x'\) and \(D_1\) is a ball centered at \(x'\).
Hence
\begin{align*}
|I_{12}|\leq c|x'|^{-n-1}\int_{\Rn}|w'|^2\Gamma'(w',t)\,dw'\leq c\frac{t}{|x'|^{n+1}} .
\end{align*}
Collecting the estimates, the left-hand side of \eqref{0515-2} equals \(k(x')+J_3(x',t)=\partial_{x_i}N(x',0)+J_3(x',t)\) with \(|J_3(x',t)|\leq ct|x'|^{-n-1}\).
This proves Lemma \ref{lemma0709-1}(iii). \qed

\subsection{ Proof of Lemma \ref{estimateofkernelt<1}}
\label{proofofkernel}\hfill\\

\(\bullet\) \,\,{\underline{\bf{Proof of \eqref{240707}}}}\quad It follows from Lemma \ref{lemma0709-1}(ii) that
\begin{align*}
\begin{split}
\sum_{k=1}^{n-1}  L_{kk} (x,t) &  =    \int_0^{x_n} \partial_{x_n} \Ga_1 (x_n - z_n,t) \Big( \sum_{k=1}^{n-1} \partial_{x_k}\partial_{x_k}  N(x', z_n)    + \sum_{k=1}^{n-1}J_2(x', z_n,t)    \Big)    dz_n\\
&= \frac{1}{4\sqrt{\pi}\,n\omega_n}\Bigl(|x'|^2 I_{n+2,0}-(n-1)I_{n+2,2}\Bigr)+R,\qquad |R|\leq C\,t\,I_{n+2,0},
\end{split}
\end{align*}
where we used \(\sum_{k<n}\partial_{x_k}^2N=-\partial_{x_n}^2N\), the formula \eqref{0716-123} for \(\partial_{x_n}^2N\), \(\partial_{x_n}\Ga_1(r,t)=-\frac{1}{4\sqrt\pi}\frac{r}{t^{3/2}}e^{-\frac{r^2}{4t}}\) and the bound on \(J_2\) in Lemma \ref{lemma0709-1}(ii).
By Lemma \ref{technicallemma}, \(|x'|^2I_{n+2,0}\approx t^{-\frac12}|x'|^2|x|^{-n-2}\min\{1,\frac{x_n^2}{t}\}\) and \(I_{n+2,2}\approx t^{-\frac12}x_n^2|x|^{-n-2}\min\{1,\frac{x_n^2}{t}\}\), while \(tI_{n+2,0}\leq\epsilon_0^2|x'|^2I_{n+2,0}\) because \(\sqrt t<\epsilon_0|x'|\).
Hence \(R\) is absorbed into the first term after decreasing \(\epsilon_0\), and \eqref{240707} follows.

\(\bullet\) \,\,{\underline{\bf{Proof of \eqref{L^Lestimate}}}}\quad Since \( \Ga(x,t) = \Ga_1 (x_n, t) \Ga'(x',t)\), it follows from Lemma \ref{lemma0709-1} that
\begin{align}
\begin{split}
{L}_{ni}  &  =  \int_0^{x_n}     \partial_{x_n} \Ga_1(x_n -z_n,t) \Big(   \partial_{z_n}\partial_{x_i}   N(x', z_n )  + \frac12 \partial_{x_i} \Ga'(x',t)  \mathbf{1}_{z_n \leq |x'|}     + J_1 (x', z_n,t) \Big) dz_n\\
& = :L_{ni,1}  +L_{ni,2}   + L_{ni,3}.
\end{split}
\end{align}
We first estimate \(L_{ni,1}(x,t)\) directly as follows:
\begin{align}\label{1231-2-1}
\begin{split}
L_{ni,1}(x,t) &  \approx    \int_0^{x_n}  t^{-\frac32}  (x_n -z_n) e^{-\frac{(x_n -z_n)^2}{4t}} \frac{  x_i z_n }{(|x'|^2 + z_n^2)^{\frac{n+2}2}} dz_n =  x_i I_{n+2,1}\approx \frac{x_ix_n}{t^{\frac12}|x|^{n+2}}\min\left\{1, \frac{x_n^2}{t}\right\}.
\end{split}
\end{align}
For \(L_{ni,2}\), we have to consider the integral
\begin{align*}
I:=\int_0^{\min\{x_n,|x'|\}}\frac{x_n-z_n}{t^{\frac32}}e^{-\frac{(x_n-z_n)^2}{4t}}dz_n=\frac{2}{\sqrt{t}}\left( e^{-\frac{(x_n-\min\{x_n,|x'|\})^2}{4t}}-e^{-\frac{x_n^2}{4t}}\right).
\end{align*}
If \(x_n\leq |x'|\), then the integral becomes \(I=\frac{2}{\sqrt{t}}(1-e^{-\frac{x_n^2}{4t}})\), so that
\begin{align}\label{0317-01}
I\approx \frac{2}{\sqrt{t}}\min\left\{1, \frac{x_n^2}{4t}\right\},
\end{align}
where we have used the inequality \(1-e^{-\theta}\approx \min\{1,\theta\}\) for \(\theta>0\).

On the other hand, if \(x_n\geq |x'|\), then the integral becomes \(I= \frac{2}{\sqrt{t}}(e^{-\frac{(x_n-|x'|)^2}{4t}}-e^{-\frac{x_n^2}{4t}})\), so that
\begin{align}\label{0317-02}
I\approx \frac{2}{\sqrt{t}}e^{-\frac{(x_n-|x'|)^2}{4t}}\approx \frac{2}{\sqrt{t}}e^{-\frac{(x_n-|x'|)^2}{4t}}\min\left\{1, \frac{x_n^2}{4t}\right\},
\end{align}
where we have used the condition \(|x'|\geq \sqrt{t}\).
Thus combining \eqref{0317-01} and \eqref{0317-02} gives
\begin{align*}
I\approx \frac{1}{\sqrt{t}}e^{-\frac{(x_n-|x'|)_+^2}{4t}}\min\left\{1,\frac{x_n^2}{4t}\right\},
\end{align*}
which implies
\begin{align}\label{0317-05}
L_{ni,2}(x,t)\approx -\frac{1}{\sqrt{t}}e^{-\frac{(x_n-|x'|)_+^2}{4t}}\min\left\{1,\frac{x_n^2}{4t}\right\}\partial_{x_i}\Gamma'(x',t)\approx \frac{x_i}{t^{\frac{n+2}{2}}}e^{-\frac{c|x|^2}{t}}\min\left\{1,\frac{x_n^2}{t}\right\},
\end{align}
where, for \(r\in \R\), \(r_+:=\max\{0,r\}\) denotes the positive part of \(r\), and we have used the inequality \((x_n-|x'|)_+^2+|x'|^2\approx |x|^2\).
Thus, the estimates \eqref{1231-2-1} and \eqref{0317-05} give
\begin{align}\label{0317-03}
L_{ni,1}(x,t)+L_{ni,2}(x,t)\approx \left(\frac{x_nx_i}{t^{\frac{1}{2}}|x|^{n+2}} +x_it^{-\frac{n+2}{2}}e^{-\frac{c|x|^2}{t}}\right)\min\left\{1,\frac{x_n^2}{t}\right\}.
\end{align}
For \(L_{ni,3}\), Lemma \ref{lemma0709-1} gives
\begin{align}\label{0317-04}
|L_{ni,3}(x,t)|\lesssim \int_0^{x_n}\frac{x_n-z_n}{t^{\frac32}}e^{-\frac{(x_n-z_n)^2}{4t}}\frac{tz_n}{|x'|(|x'|^2+z_n^2)^{\frac{n+2}{2}}}dz_n\approx \frac{t}{|x'|} I_{n+2,1}.
\end{align}
Combining the estimates \eqref{0317-03} and \eqref{0317-04} gives
\begin{align*}
L_{ni}(x,t) &\approx x_i\left( \frac{x_n}{t^{\frac12}|x|^{n+2}}  +   t^{-\frac{n+2}{2}}e^{-\frac{c|x|^2}{t}}\right)\min\left\{1,\frac{x_n^2}{t}\right\},
\end{align*}
since \(\sqrt{t}\leq \epsilon_0|x'|\) for sufficiently small \(\epsilon_0>0\) and we assumed that \(|x'|\leq c_0 x_i\) for some \(c_0>1\).

\subsection{Proof of Lemma \ref{elementarylemma}}\label{proofofelementarylemma}\hfill\\

\(\bullet\) {\bf (Case \(t\geq\frac98\))}\, Noting that \( t -s \approx t-1 \) for \(t \geq \frac98\) and \( \frac12 < s < 1\), we obtain
\begin{align*}
\mathcal{G}_{a,k}(x_n,t)&\approx\int_{\frac{1}{2}}^1    (1-s)^a  (t-1)^{-k} \min\left\{1,\frac{x_n^2}{t-1}\right\} ds \approx \frac{1}{a+1}(t-1)^{-k} \min\left\{1,\frac{x_n^2}{t-1}\right\}.
\end{align*}

\(\bullet\) {\bf (Case \(1<t<\frac98\))}\,

i) If \(x_n^2\leq 2(t-1)\), then
\begin{align*}
\mathcal{G}_{a,k}(x_n,t)& = \int_{\frac12}^{2-t}\frac{(1-s)^a}{(t-s)^k}\min\left\{1,\frac{x_n^2}{t-s}\right\} ds  +  \int_{2-t}^{1}\frac{(1-s)^a}{(t-s)^k}\min\left\{1,\frac{x_n^2}{t-s}\right\} ds\\
&\leq \int_{\frac12}^{2-t}\frac{(1-s)^a}{(1-s)^k}\min\left\{1,\frac{x_n^2}{1-s}\right\} ds + \frac{1}{(t-1)^{k}}\min\left\{1, \frac{x_n^2}{t-1}\right\}\int_{2-t}^1 (1-s)^a ds\\
&\leq \int_{t-1}^{\frac12}u^{a-k}\min\left\{2, \frac{x_n^2}{u}\right\} du + \frac{(t-1)^{a+1}}{(a+1)(t-1)^k}\min\left\{2, \frac{x_n^2}{t-1}\right\}\\
&=x_n^2\int_{t-1}^{\frac12}u^{a-k-1}du  + \frac{x_n^2(t-1)^{a-k}}{a+1}\\
&\leq x_n^2\left(\frac{2^{k-a}-(t-1)^{a-k}}{a-k}\mathbf{1}_{a\neq k}  +|\ln(2(t-1))|\mathbf{1}_{a=k}\right)  +  \frac{x_n^2(t-1)^{a-k}}{a+1},
\end{align*}
where we have used that \(1-s<t-s\) and \(t-1<t-s\) hold for \(s<1<t\) and \(\min\{1, u\}\leq \min\{1,v\}\) for \(0\leq u\leq v\).
Similarly, we have
\begin{align*}
\mathcal{G}_{a,k}(x_n,t)&\geq \int_{\frac12}^{2-t}\frac{(1-s)^a}{2^k(1-s)^k}\min\left\{1, \frac{x_n^2}{2(1-s)}\right\} ds +  \frac{1}{2^k(t-1)^k}\min\left\{1, \frac{x_n^2}{2(t-1)}\right\}\int_{2-t}^1 (1-s)^a ds\\
&= \frac{1}{2^{k+1}}\int_{t-1}^{\frac12}u^{a-k}\min\left\{2, \frac{x_n^2}{u}\right\} du+ \frac{1}{2^{k+1}(t-1)^k}\min\left\{2, \frac{x_n^2}{t-1}\right\} \frac{(t-1)^{a+1}}{a+1}\\
&= \frac{x_n^2}{2^{k+1}}\int_{t-1}^{\frac12}u^{a-k-1}du + \frac{x_n^2}{2^{k+1}(t-1)^{k+1}}\frac{(t-1)^{a+1}}{a+1}\\
&=\frac{x_n^2}{2^{k+1}}\left(\frac{2^{k-a}-(t-1)^{a-k}}{a-k}\mathbf{1}_{a\neq k}  + |\ln (2(t-1))|\mathbf{1}_{a=k}\right) +\frac{x_n^2(t-1)^{a-k}}{2^{k+1}(a+1)},
\end{align*}
where we have used that \(t-s<2(1-s)\) and \(t-s<2(t-1)\) hold for \(\frac12<s<2-t\) and \(2-t<s<1\) respectively.

ii) If \(x_n\geq \frac12\), then
\begin{align*}
\mathcal{G}_{a,k}(x_n,t) & = \int_{\frac12}^{2-t}\frac{(1-s)^a}{(t-s)^k}\min\left\{1, \frac{x_n^2}{t-s}\right\} ds +\int_{2-t}^1 \frac{(1-s)^a}{(t-s)^k}\min\left\{ 1, \frac{x_n^2}{t-s}\right\} ds \\
&\leq \int_{\frac12}^{2-t}(1-s)^{a-k}\min\left\{ 1, \frac{x_n^2}{t-s}\right\} ds + \int_{2-t}^1 \frac{(1-s)^a}{(t-1)^k}\min\left\{1, \frac{x_n^2}{t-1}\right\} ds \\
& \leq \int_{\frac12}^{2-t}(1-s)^{a-k} ds + \int_{2-t}^1 \frac{(1-s)^a}{(t-1)^k} ds \\
& = \frac{2^{k-a-1}-(t-1)^{a-k+1}}{a-k+1}\mathbf{1}_{a\neq k-1}  +  |\ln(2(t-1))|\mathbf{1}_{a=k-1}+ \frac{(t-1)^{a+1-k}}{a+1}.
\end{align*}

Similarly,
\begin{align*}
\mathcal{G}_{a,k}(x_n,t)& \geq \int_{\frac12}^{2-t}\frac{(1-s)^a}{2^k(1-s)^k}\min\left\{1, \frac{x_n^2}{2(1-s)}\right\} ds + \int_{2-t}^1 \frac{(1-s)^a}{2^k(t-1)^k}\min\left\{1, \frac{x_n^2}{2(t-1)}\right\} ds \\
& \geq \int_{\frac12}^{2-t} \frac{(1-s)^{a-k}}{2^{k+1}}\min\left\{1, \frac{x_n^2}{1-s}\right\} ds + \int_{2-t}^1 \frac{(1-s)^a}{2^{k+1}(t-1)^k}\min\left\{1, \frac{x_n^2}{t-1}\right\} ds \\
& \geq \frac{1}{2^{k+2}}\int_{t-1}^{\frac12}u^{a-k}du + \frac{1}{2^{k+1}(t-1)^{k}}\int_0^{t-1}u^a du\\
& \geq \frac{1}{2^{k+2}}\left( \frac{2^{k-a-1}-(t-1)^{a-k+1}}{a-k+1}\mathbf{1}_{a\neq k-1}  +  |\ln(2(t-1))|\mathbf{1}_{a=k-1}  + \frac{(t-1)^{a-k+1}}{a+1}  \right),
\end{align*}
where we used that \(\min\{1,\frac{x_n^2}{1-s}\}\geq\frac12\) for \(x_n\geq\frac12\) and \(1-s\leq\frac12\), and that \(\min\{1,\frac{x_n^2}{t-1}\}=1\) since \(t-1<\frac18<x_n^2\).

iii) If \(2(t-1)\leq x_n^2\leq \frac14\), then \(\frac12 \leq t-x_n^2\leq 1\) and thus
\begin{align}\label{1215}
\begin{split}
\mathcal{G}_{a,k}(x_n,t) & = \int_{\frac12}^{t-x_n^2}\frac{(1-s)^a}{(t-s)^k}\frac{x_n^2}{t-s}ds + \int_{t-x_n^2}^1 \frac{(1-s)^a}{(t-s)^k}ds\\
&\leq \int_{\frac12}^{t-x_n^2}(1-s)^{a-k-1}x_n^2 ds + \int_{t-x_n^2}^{2-t}(1-s)^{a-k}ds + \int_{2-t}^1 \frac{(1-s)^a}{(t-1)^k}ds\\
& = x_n^2 \int_{x_n^2-t+1}^{\frac12} u^{a-k-1}du + \int_{t-1}^{x_n^2-t+1} u^{a-k}du  +  \int_0^{t-1}\frac{u^a}{(t-1)^k}du \\
&\leq x_n^2\int_{\frac{x_n^2}{2}}^{\frac12}u^{a-k-1}du + \int_{t-1}^{x_n^2}u^{a-k}du  +  \int_0^{t-1}\frac{u^a}{(t-1)^k}du\\
& = x_n^2\left(2^{k-a}\frac{1-x_n^{2(a-k)}}{a-k}\mathbf{1}_{a\neq k}  + |\ln x_n^2|\mathbf{1}_{a=k}\right) \\
&\quad  + \frac{x_n^{2(a-k+1)}-(t-1)^{a-k+1}}{a-k+1}\mathbf{1}_{a\neq k-1}  +  \ln \frac{x_n^2}{t-1}\mathbf{1}_{a= k-1}+ \frac{(t-1)^{a-k+1}}{a+1}.
\end{split}
\end{align}

Similarly,
\begin{align}\label{1215-1}
\begin{split}
\mathcal{G}_{a,k}(x_n,t)& \geq \int_{\frac12}^{t-x_n^2}\frac{x_n^2}{2^{k+1}}(1-s)^{a-k-1}ds + \int_{t-x_n^2}^{2-t}\frac{(1-s)^{a-k}}{2^k}ds +\int_{2-t}^1 \frac{(1-s)^a}{2^k(t-1)^k}ds \\
& = \int_{x_n^2-t+1}^{\frac12} \frac{x_n^2}{2^{k+1}}u^{a-k-1}du + \int_{t-1}^{x_n^2-t+1}\frac{u^{a-k}}{2^{k}}du + \int_0^{t-1}\frac{u^a}{2^k(t-1)^k}du\\
& \geq \int_{x_n^2}^{\frac12}\frac{x_n^2}{2^{k+1}}u^{a-k-1}du + \int_{t-1}^{\frac{x_n^2}{2}}\frac{u^{a-k}}{2^{k}}du + \frac{(t-1)^{a+1-k}}{2^k(a+1)}\\
& = \frac{x_n^2}{2^{k+1}}\left(\frac{2^{k-a}-x_n^{2(a-k)}}{a-k}\mathbf{1}_{a\neq k}  +  |\ln(2x_n^2)|\mathbf{1}_{a=k}  \right)\\
&\quad + \frac{1}{2^k}\left(2^{k-a-1}\frac{x_n^{2(a-k+1)}-(2(t-1))^{a-k+1}}{a-k+1}\mathbf{1}_{a\neq k-1}+ \ln \frac{x_n^2}{2(t-1)}\mathbf{1}_{a=k-1}\right) +\frac{(t-1)^{a+1-k}}{2^{k}(a+1)}.
\end{split}
\end{align}

We now derive item (iv) from \eqref{1215} and \eqref{1215-1}.
Throughout, \(2(t-1)\leq x_n^2\leq\frac14\), the constants may depend on \(a\) and \(k\), and we write \(\tau:=t-1\), \(\theta:=a-k\) and
\begin{align*}
T_1:=x_n^2\int_{x_n^2}^{\frac12}u^{\theta-1}\,du,\qquad T_2:=\int_{\tau}^{x_n^2}u^{\theta}\,du,\qquad T_3:=\frac{\tau^{\theta+1}}{a+1}.
\end{align*}
Since \(2x_n^2\leq\frac12\), we have 
\begin{align*}
    T_1\geq x_n^2\int_{x_n^2}^{2x_n^2}u^{\theta-1}du\geq c\,x_n^{2\theta+2},
\end{align*}
while
\begin{align*}
    x_n^2\int_{x_n^2/2}^{x_n^2}u^{\theta-1}du+\int_{x_n^2/2}^{x_n^2}u^{\theta}du\leq C x_n^{2\theta+2}.
\end{align*}
Hence the right-hand side of \eqref{1215} is at most \(C(T_1+T_2+T_3)\), and the right-hand side of \eqref{1215-1} is at least \(c(T_1+T_2-Cx_n^{2\theta+2}+T_3)\) as well as at least \(c\,T_1\); consequently
\begin{align}\label{G-three-terms}
\mathcal{G}_{a,k}(x_n,t)\approx T_1+T_2+T_3 .
\end{align}
We next evaluate the three terms, using \(\tau^{\theta+1}\leq2^{-(\theta+1)}x_n^{2\theta+2}\) if \(\theta>-1\) and \(x_n^{2\theta+2}\leq(2\tau)^{\theta+1}\leq\tau^{\theta+1}\) if \(\theta<-1\), both consequences of \(x_n^2\geq2\tau\):
\begin{itemize}
\item[(a)] \(T_1=x_n^2\frac{2^{-\theta}-x_n^{2\theta}}{\theta}\) if \(\theta\neq0\) and \(T_1=x_n^2|\ln(2x_n^2)|\) if \(\theta=0\).
Since \(x_n^2\leq\frac14\), we have \(T_1\approx x_n^{2\theta+2}\) if \(\theta<0\) and \(T_1\approx x_n^2\,\frac{2^{-\theta}}{\theta}\) if \(\theta>0\).
\item[(b)] \(T_2=\frac{x_n^{2\theta+2}-\tau^{\theta+1}}{\theta+1}\approx x_n^{2\theta+2}\) if \(\theta>-1\), \(T_2=\ln\frac{x_n^2}{\tau}\geq\ln2\) if \(\theta=-1\) and \(T_2=\frac{\tau^{\theta+1}-x_n^{2\theta+2}}{|\theta+1|}\approx\frac{\tau^{\theta+1}}{|\theta+1|}\) if \(\theta<-1\).
\item[(c)] \(T_3\lesssim x_n^{2\theta+2}\) if \(\theta>-1\), and \(T_3=\frac1k\) if \(\theta=-1\).
\end{itemize}
Inserting (a)--(c) into \eqref{G-three-terms}, we obtain the following.

If \(a<k-1\), then \(T_1\lesssim\tau^{\theta+1}\), so that \(\mathcal{G}_{a,k}\approx\tau^{a-k+1}\bigl(\frac{1}{k-1-a}+\frac1{a+1}\bigr)\).
For \(a<-1+\frac k2\) we have \(k-1-a>a+1\), whence \(\mathcal{G}_{a,k}\approx\frac{\tau^{a-k+1}}{a+1}\); for \(-1+\frac k2\leq a<k-1\) we have \(k-1-a\leq a+1\), whence \(\mathcal{G}_{a,k}\approx\frac{\tau^{a-k+1}}{k-1-a}\approx T_2\).
Thus \(-1+\frac k2\) is the transition value at which the two coefficients \(\frac1{a+1}\) and \(\frac1{k-1-a}\) coincide; it is not an integrability threshold.
Since the constants are allowed to depend on \(a\), for fixed \(a<k-1\) the two expressions are comparable to each other, both being \(\approx\tau^{a-k+1}\); the distinction is retained only to record which coefficient is the larger, and no uniformity in \(a\) is claimed (see Remark \ref{1217}).

If \(a=k-1\), then \(T_1\approx1\) and \(T_3=\frac1k\) are bounded by a constant multiple of \(\ln2\leq T_2\), so \(\mathcal{G}_{a,k}\approx\ln\frac{x_n^2}{\tau}\).

If \(k-1<a<k-\frac12\), then \(T_1\approx T_2\approx x_n^{2\theta+2}\) and \(T_3\lesssim T_2\), so \(\mathcal{G}_{a,k}\approx T_2=\frac{x_n^{2a-2k+2}-\tau^{a+1-k}}{a+1-k}\).
Together with the previous paragraphs, this gives the first two terms in (iv) and the logarithmic term for \(a=k-1\).

If \(a=k-\frac12\), then \(T_1\approx T_2\approx x_n\) and \(T_3\lesssim x_n\), so \(\mathcal{G}_{a,k}\approx x_n\).

If \(k-\frac12<a<k+\frac12\) and \(a\neq k\), then \(T_2+T_3\lesssim x_n^{2\theta+2}\lesssim T_1\), so \(\mathcal{G}_{a,k}\approx T_1=x_n^2\frac{2^{k-a}-x_n^{2a-2k}}{a-k}\).
If \(a=k\), then likewise \(\mathcal{G}_{a,k}\approx T_1=x_n^2|\ln(2x_n^2)|\).

Finally, if \(a\geq k+\frac12\), then \(T_2+T_3\lesssim x_n^{2\theta+2}\leq 2^{-2\theta}x_n^2\lesssim T_1\approx x_n^2\frac{2^{-\theta}}{\theta}\), and \(\frac{2^{-\theta}}{\theta}=\frac{2^{k}\,2^{-a}}{a-k}\approx\frac{2^{-a}}{a}\) because \(a-k\geq\frac12\).
Hence \(\mathcal{G}_{a,k}\approx\frac{x_n^2}{2^aa}\).
This completes the proof of item (iv).

\subsection{Proof of Lemma \ref{elementarylemma2}}\label{proofofelementarylemma2}\hfill\\

Before proving Lemmas \ref{elementarylemma2} and \ref{elementarylemma3}, we estimate the integral
\begin{equation}\label{integralI}
I(\alpha;x,y):=\int_{x}^y e^{-u}u^{\alpha}du.
\end{equation}

\begin{lemma}\label{twosidedgammaestimate}
Let \(0 < x < y\).
Then the integral \(I(\alpha;x,y)\) satisfies the following bounds: for fixed universal constants \(0<M_0<N_0\),
\begin{itemize}
\item[(i)] If \(y\leq N_0\), then
\begin{align*}
e^{-N_0}\left(\frac{y^{\alpha+1}-x^{\alpha+1}}{\alpha+1}\mathbf{1}_{\alpha\neq -1}+ \ln \frac{y}{x}\mathbf{1}_{\alpha=-1}\right)  \leq  I(\alpha;x,y)\leq \frac{y^{\alpha+1}-x^{\alpha+1}}{\alpha+1}\mathbf{1}_{\alpha\neq -1}+ \ln \frac{y}{x}\mathbf{1}_{\alpha=-1}.
\end{align*}
\item[(ii)] If \(y\geq N_0\), \(x\leq M_0\), then
\begin{align*}
& e^{-\frac{M_0+N_0}{2}}\left(\frac{\left(\frac{M_0+N_0}{2}\right)^{\alpha+1}-x^{\alpha+1}}{\alpha+1}\mathbf{1}_{\alpha\neq -1}  +  \ln \frac{M_0+N_0}{2x}\mathbf{1}_{\alpha=-1}\right)  +  \int_{\frac{M_0+N_0}{2}}^{N_0} e^{-u}u^{\alpha}du\\
&\leq I(\alpha;x,y)\leq \frac{\left(\frac{M_0+N_0}{2}\right)^{\alpha+1}-x^{\alpha+1}}{\alpha+1}\mathbf{1}_{\alpha\neq -1}  +  \ln \frac{M_0+N_0}{2x}\mathbf{1}_{\alpha=-1} +\int_{\frac{M_0+N_0}{2}}^{\infty}e^{-u}u^{\alpha}du.
\end{align*}
\item[(iii)] If \(y\geq N_0\), \(x\geq M_0\), \(y\geq \frac{N_0}{M_0}x\), then
\begin{align*}
\min\left\{1, \left(\frac{N_0}{M_0}\right)^{\alpha}\right\} (1-e^{M_0-N_0})e^{-x}x^\alpha \leq I(\alpha;x,y)\leq \max\left\{1, \int_0^{\infty}e^{-u}\left(1+\frac{u}{M_0}\right)^{\alpha}du\right\}e^{-x}x^{\alpha}.
\end{align*}
\item[(iv)] If \(y\geq N_0\), \(x\geq M_0\), \(y\leq \frac{N_0}{M_0} x\), then
\begin{align*}
\left(\mathbf{1}_{\alpha\geq0} + \left(\frac{N_0}{M_0}\right)^{\alpha}\mathbf{1}_{\alpha<0}  \right)x^{\alpha}(y-x)e^{-\frac{N_0x}{M_0}}   \leq I(\alpha;x,y)\leq  \left(\left(\frac{N_0}{M_0}\right)^{\alpha}\mathbf{1}_{\alpha\geq 0}+\mathbf{1}_{\alpha<0}\right)  x^\alpha (y-x)e^{-x}.
\end{align*}
\end{itemize}
\end{lemma}

\begin{proof}
{\bf (i). (\(y\leq N_0\)):}\quad We have \(e^{-N_0}\leq e^{-u}\leq 1\).
Thus 
\begin{align*}
  e^{-N_0}\int_x^{y}u^{\alpha}du\leq I(\alpha;x,y)\leq \int_x^y u^{\alpha}du . 
\end{align*}

{\bf (ii). (\(y\geq N_0\), \(x\leq M_0\)):} \quad We split the integral as
\begin{align*}
     I(\alpha;x,y)=\left(\int_x^{\frac{M_0+N_0 }{2}}+\int_{\frac{M_0+N_0}{2}}^{y}\right)e^{-u}u^{\alpha}du.
\end{align*}

The first integral is treated as in (i), while the second satisfies
\begin{align*}
\int_{\frac{M_0+N_0}{2}}^{N_0} e^{-u}u^{\alpha}du\leq \int_{\frac{M_0+N_0}{2}}^y e^{-u}u^{\alpha}du\leq \int_{\frac{M_0+N_0}{2}}^{\infty}e^{-u}u^{\alpha}du.
\end{align*}

{\bf (iii). (\(y\geq N_0\), \(x\geq M_0\), \(y\geq \frac{N_0}{M_0}x \)):}\quad For the upper bound, we see that
\begin{align*}
I(\alpha;x,y)\leq \int_x^{\infty}e^{-u}u^{\alpha}du\leq C_1 e^{-x}x^{\alpha},
\end{align*}
where
\begin{align*}
C_1:= \max_{x\geq M_0}\frac{\displaystyle  \int_{x}^{\infty}e^{-u}u^{\alpha}du}{e^{-x}x^{\alpha}} = \max_{x\geq M_0} \int_0^{\infty}e^{-u}\left(1+\frac{u}{x}\right)^{\alpha}du.
\end{align*}
Then the result follows using \((1+\frac{u}{x})^{\alpha}\leq (1+\frac{u}{M_0})^{\alpha}\) for \(\alpha\geq 0\), and \((1+\frac{u}{x})^{\alpha}\leq 1\) for \(\alpha<0\).

For the lower bound, we see that
\begin{align*}
I(\alpha;x,y)\geq \int_{x}^{\frac{N_0x}{M_0}}e^{-u}u^{\alpha}du\geq C_2e^{-x}x^\alpha,
\end{align*}
where
\begin{align*}
C_2:=\min_{x\geq M_0}\frac{\displaystyle\int_x^{\frac{N_0x}{M_0}}e^{-u}u^{\alpha}du}{e^{-x}x^{\alpha}} = \min_{x\geq M_0}\int_0^{\frac{N_0-M_0}{M_0}x}e^{-u}\left(1+\frac{u}{x}\right)^{\alpha}du.
\end{align*}
The result follows from the corresponding inequalities used for the upper bound.

{\bf (iv). (\(y\geq N_0\), \(x\geq M_0\), \(y\leq \frac{N_0}{M_0} x\)):} \quad For \(\alpha \geq 0\), using \(x\leq u\leq y\) and the mean-value theorem,
\begin{align*}
&I(\alpha;x,y)\leq y^\alpha (e^{-x}-e^{-y})\leq y^{\alpha}(y-x)e^{-x}\leq \left(\frac{N_0}{M_0}\right)^{\alpha}x^{\alpha}e^{-x}(y-x),\\
&I(\alpha;x,y)\geq x^{\alpha}(e^{-x}-e^{-y})\geq x^{\alpha}(y-x)e^{-y}\geq x^{\alpha}e^{-\frac{N_0x}{M_0}}(y-x).
\end{align*}
For \(\alpha<0\), we similarly have
\begin{align*}
&I(\alpha;x,y)\leq x^{\alpha}(e^{-x}-e^{-y})\leq x^{\alpha}e^{-x}(y-x),\\
&I(\alpha;x,y)\geq y^{\alpha}(e^{-x}-e^{-y})\geq \left(\frac{N_0}{M_0}\right)^{\alpha}x^{\alpha}e^{-y}(y-x)\geq \left(\frac{N_0}{M_0}\right)^{\alpha}x^{\alpha}e^{-\frac{N_0x}{M_0}}(y-x).
\end{align*}

\end{proof}
We now prove Lemma \ref{elementarylemma2}.

{\bf \(\bullet\) (Case \(t\geq\frac98\))}.
Using \(t-1<t-s<5(t-1)\) for \(t\geq\frac98\) and \(\frac12<s<1\), we have
\begin{align*}
\mathcal{H}_{a,k}(r,t)\approx\int_{\frac{1}{2}}^1\frac{(1-s)^a}{(t-1)^k}e^{-\frac{cr^2}{t-1}}ds=\frac{1}{a+1}\frac{e^{-\frac{cr^2}{t-1}}}{(t-1)^k}.
\end{align*}

{\bf \(\bullet\) (Case \(1<t<\frac98\))}.
In this case we have
\begin{align*}
\mathcal{H}_{a,k}(r,t) &\approx \int_{\frac{1}{2}}^{2-t}(1-s)^{a-k} e^{-\frac{cr^2}{1-s}}\,ds + (t-1)^{-k}e^{-\frac{cr^2}{t-1}} \int_{2-t}^1 (1-s)^{a}\,ds \\
&\approx r^{2(a-k+1)} I\!\left(k-a-2; 2cr^2,\frac{cr^2}{t-1}\right) + \frac{1}{a+1}(t-1)^{a+1-k} e^{-\frac{cr^2}{t-1}}.
\end{align*}

From now on, we choose \(x=2cr^2\), \(y= \frac{cr^2}{t-1}\), \(M_0=\frac{c}{2}\) and \(N_0=2c\).
Note that \(N_0>M_0>0\) is satisfied.

i) If \(r^2 \leq 2(t-1)\), then \(y\leq N_0\) and thus from Lemma \ref{twosidedgammaestimate}(i),
\begin{align*}
\mathcal{H}_{a,k}(r,t)&\approx r^{2(a-k+1)}\left((cr^2)^{k-a-1}\frac{(t-1)^{a+1-k}-2^{k-a-1}}{k-a-1}\mathbf{1}_{a\neq k-1}  +  |\ln(2(t-1))|\mathbf{1}_{a= k-1}  \right) \\
&\quad + \frac{1}{a+1}(t-1)^{a+1-k}e^{-\frac{cr^2}{t-1}}\\
&= c^{k-a-1}\frac{(t-1)^{a+1-k}-2^{k-a-1}}{k-a-1}\mathbf{1}_{a\neq k-1}  +  |\ln(2(t-1))|\mathbf{1}_{a=k-1}  +  \frac{1}{a+1}(t-1)^{a+1-k}e^{-\frac{cr^2}{t-1}}.
\end{align*}

ii) If \(2(t-1)\leq r^2\leq \frac14\), then \(y\geq N_0\) and \(x\leq M_0\) and thus from Lemma \ref{twosidedgammaestimate}(ii),
\begin{align*}
\mathcal{H}_{a,k}(r,t)&\approx r^{2(a-k+1)} \left(\left(\frac{5c}{4}\right)^{k-a-1}\frac{1-\left(\frac{8r^2}{5}\right)^{k-a-1}}{k-a-1}\mathbf{1}_{a\neq k-1}  +  \ln \frac{5}{8r^2}\mathbf{1}_{a=k-1} \right)  +  \frac{1}{a+1}(t-1)^{a+1-k}e^{-\frac{cr^2}{t-1}}\\
&\approx \frac{1-r^{2(a-k+1)}}{a-k+1}\mathbf{1}_{a\neq k-1}  + |\ln r|\mathbf{1}_{a=k-1} + \frac{1}{a+1}(t-1)^{a+1-k}e^{-\frac{cr^2}{t-1}}.
\end{align*}
Here the last integral in Lemma \ref{twosidedgammaestimate}(ii) is bounded above and below by positive constants depending on \(a\) and \(k\); multiplied by \(r^{2(a-k+1)}\), it is dominated by the first term (because \(r^2\leq\frac14\)), and we have absorbed it.

iii) If \(r\geq \frac12\), then \(x\geq M_0\), \(y\geq N_0\) and \(y\geq \frac{N_0}{M_0}x\) (the last condition is equivalent to \(1<t<\frac{9}{8}\)) and thus from Lemma \ref{twosidedgammaestimate}(iii),
\begin{align*}
\mathcal{H}_{a,k}(r,t)\approx r^{2(a-k+1)}e^{-2cr^2}r^{2(k-a-2)} + \frac{1}{a+1}(t-1)^{a+1-k}e^{-\frac{cr^2}{t-1}}\approx r^{-2}e^{-2cr^2} + \frac{1}{a+1}(t-1)^{a+1-k}e^{-\frac{cr^2}{t-1}}.
\end{align*}

\subsection{Proof of Lemma \ref{elementarylemma3}}\label{proofofelementarylemma3}\hfill\\

\(\bullet\) {\bf (Case \(t\geq\frac98\))}\, Noting that \( t -s \approx t-1 \) for \(t \geq \frac98\) and \( \frac12 < s < 1\), we obtain
\begin{align*}
\mathcal{K}_{a,k}(x,t)&\approx\int_{\frac{1}{2}}^1     \frac{(1-s)^a}{(t-1)^{k}}e^{-\frac{c|x|^2}{t-1}}   \min\left\{ 1,\frac{x_n^2}{t-1} \right\}    ds \approx   \frac{1}{(a+1)(t-1)^{k}}e^{-\frac{c|x|^2}{t-1}}   \min\left\{1,\frac{x_n^2}{t-1} \right\}.
\end{align*}

\(\bullet\) {\bf (Case \(1<t<\frac98\))}\, Since \(t-s\approx 1-s\) for \(\frac12<s<2-t\) and \(t-s\approx t-1\) for \(2-t<s<1\), it follows that
\begin{align*}
\mathcal{K}_{a,k}(x,t) & \approx\int_{\frac{1}{2}}^{ 2-t }  (1-s)^{a-k}  e^{-\frac{c|x|^2}{1-s}}   \min \left\{1,\frac{x_n^2}{1-s}\right\} ds +  (t-1)^{-k} e^{-\frac{c|x|^2}{t-1}} \min \left\{1,\frac{x_n^2}{t-1}\right\}\int_{2-t}^1  (1-s)^a ds\\
&  = \int_{t-1}^{\frac{1}{2}}  s^{a-k}  e^{-\frac{c|x|^2}{s}}  \min \left\{1,\frac{x_n^2}{s}\right\} ds + \frac1{1 +a} (t-1)^{-k+a+1} e^{-\frac{c|x|^2}{t-1}}  \min \left\{1,\frac{x_n^2}{t-1}\right\}.
\end{align*}

i) If \(x_n^2< 2(t-1)\), then we have \(\min\left\{1, \frac{x_n^2}{s}\right\}\approx \frac{x_n^2}{s}\) and thus by a change of variable and Lemma \ref{twosidedgammaestimate}(iii),
\begin{align*}
\int^\frac{1}{2}_{t-1}  s^{a-k}  e^{-\frac{c|x|^2}{s}}  \min \left\{1,\frac{x_n^2}{s}\right\} ds &\approx x_n^2|x|^{2a-2k}I\left(k-a-1;2c|x|^2, \frac{c|x|^2}{t-1}\right)\\
&\approx x_n^2|x|^{2a-2k} |x|^{2k-2-2a}e^{-2c|x|^2}=x_n^2|x|^{-2}e^{-2c|x|^2}.
\end{align*}

ii) If \(2(t-1)\leq x_n^2< \frac14\), then we have \(\min\left\{1,\frac{x_n^2}{s}\right\}=\frac{x_n^2}{s}\) for \(x_n^2\leq s\leq \frac12\) and \(\min\left\{1,\frac{x_n^2}{s}\right\}=1\) for \(t-1\leq s\leq x_n^2\) and thus we change variables.
Each of the two incomplete gamma integrals below falls under Lemma \ref{twosidedgammaestimate}(iii) or (iv), depending on whether the ratio of its limits is at least or less than \(\frac{N_0}{M_0}=4\); this ratio is \(\frac{x_n^2}{t-1}\geq2\) for the first integral and \(\frac{1}{2x_n^2}>2\) for the second.
In either case the first term is exponentially small.
When item (iv) is used, its extra factor \(y-x\) is a power of \(|x|\) and \(x_n^{-1}\), and the exponentials in its lower and upper bounds differ; both effects are absorbed by changing the constant \(c\) in the exponent, as allowed by the convention of Section \ref{sect3}.
Thus
\begin{align*}
\int^\frac{1}{2}_{t-1}  s^{a-k}  e^{-\frac{c|x|^2}{s}}  \min \left\{1,\frac{x_n^2}{s}\right\} ds &\approx |x|^{2a-2k+2}I\left(k-a-2;\frac{c|x|^2}{x_n^2}, \frac{c|x|^2}{t-1}\right)\\
&\quad+x_n^2|x|^{2a-2k}I\left(k-a-1;2c|x|^2,\frac{c|x|^2}{x_n^2}\right)\\
&\approx x_n^{2a-2k+4}|x|^{-2}e^{-\frac{c|x|^2}{x_n^2 }}+x_n^2|x|^{-2}e^{-2c|x|^2}\approx x_n^2|x|^{-2}e^{-2c|x|^2}.
\end{align*}
The last ``\(\approx\)'' can be proved by considering the cases \(a-k+1\geq 0\) and \(a-k+1<0\) separately.

iii) If \( x_n \geq \frac12\), then \(\frac12\leq\min\left\{1,\frac{x_n^2}{s}\right\}\leq1\) for \(t-1\leq s\leq \frac12\), and thus by a change of variable and Lemma \ref{twosidedgammaestimate}(iii),
\begin{align*}
\int^\frac{1}{2}_{t-1}  s^{a-k}  e^{-\frac{c|x|^2}{s}}  \min \left\{1,\frac{x_n^2}{s}\right\} ds &\approx |x|^{2a-2k+2}I\left(k-a-2; 2c|x|^2, \frac{c|x|^2}{t-1}\right)\approx |x|^{-2}e^{-2c|x|^2}.
\end{align*}
This completes the proof of Lemma \ref{elementarylemma3}.
\subsection{Proof of Lemma \ref{lemma0630}}\label{proofoflemma0630}\hfill\\
Since \(h'(\theta)=\theta^{b-1}(b\ln\theta+1)\), the function \(h\) is increasing on \((1,\infty)\).
Write \(\vartheta_A=A R^{1/b}(\ln R)^{-1/b}\).
Then
\begin{align*}
h(\vartheta_A)=A^bR\left(\frac{\ln A}{\ln R}+\frac1b-\frac1b\frac{\ln\ln R}{\ln R}\right).
\end{align*}
For \(R>2\) one has \(\frac{\ln\ln R}{\ln R}\leq \frac1e<\frac12\) and \(\frac{|\ln\ln R|}{\ln R}\leq1\).
Since \(A_-\leq4\) and \(\frac{2\ln 2}{\ln R}<2\),
\begin{align*}
h(\theta_-^*) \leq A_-^bR\left(\frac{2\ln2}{\ln R}+\frac1b+\frac1b\frac{|\ln\ln R|}{\ln R}\right) \leq A_-^bR\frac{2b+3}{b}\leq cR.
\end{align*}
(The assertion \(h(\theta)<cR\) is vacuous if \(\theta_-^*\leq2\).)
On the other hand, \(A_+^b\geq2bc\) and \(\ln A_+>0\), and the bound \(\frac{\ln\ln R}{\ln R}<\frac12\) gives
\begin{align*}
h(\theta_+^*)\geq A_+^bR\,\frac{1}{2b}\geq cR.
\end{align*}
Moreover \(\theta_+^*>1\).
Monotonicity on \((1,\infty)\) therefore yields the two stated inequalities.

\subsection{Proof of Lemma \ref{lemma0928}}\label{proofoflemma0928}\hfill\\
By the change of variables \(u=\frac{\theta^2}{4s}\), we have
\begin{align*}
A(\alpha, \theta, t)=4^{-\alpha-1}\theta^{2\alpha+2} \int_{\theta^2/2}^{\theta^2/(4(1-t))}u^{-\alpha-2}e^{-u}\,du,\quad B(\alpha,\theta,t)=4^{-\alpha-1}\theta^{2\alpha+2} \int_{\theta^2/(4(1-t))}^{\infty}u^{-\alpha-2}e^{-u}\,du.
\end{align*}

If \( \theta^2\leq 2(1-t)\), it follows that
\begin{align*}
A(\alpha, \theta, t) &   \approx 4^{-\alpha-1}\theta^{2\alpha +2} \int_{\theta^2/2}^{\frac{\theta^2}{4(1-t)}}  u^{-\alpha-2}   du   = \frac{2^{-\alpha-1}-(1-t)^{\alpha+1}}{\alpha+1}\mathbf{1}_{\al \neq -1} + |\ln ( 2 (1-t))|\mathbf{1}_{ \al =-1},\\
B(\alpha, \theta, t) & \approx 4^{-\alpha-1}\theta^{2\alpha +2}  \Big(   \int^{\frac12}_{\frac{\theta^2}{4(1-t)}}  u^{-\alpha-2 }du  +  1\Big)\approx\frac{(1-t)^{\alpha+1}-(\theta^2/2)^{\alpha+1}}{\alpha+1}\mathbf{1}_{\alpha\neq -1}+\ln \frac{2(1-t)}{\theta^2}\mathbf{1}_{\alpha=-1}+\theta^{2\alpha+2}.
\end{align*}
If \( 2(1-t) \leq \theta^2 \leq \frac14 \), we obtain
\begin{align*}
A(\alpha, \theta, t) & \approx 4^{-\alpha-1}\theta^{2\alpha +2} \Big( \int_{\theta^2/2}^{\frac{3}{16}}  u^{-\alpha-2}   du +  1 \Big)= \frac{2^{-\alpha-1}-(4\theta^2/3)^{\alpha+1}}{\alpha+1}\mathbf{1}_{\alpha\neq -1} + \ln\frac{3}{8\theta^2}\mathbf{1}_{\alpha=-1} + 4^{-\alpha-1}\theta^{2\alpha+2},\\
B(\alpha, \theta, t)&   = 4^{-\alpha-1}\theta^{2\alpha +2} \int^\infty_{\frac{\theta^2}{4(1-t)}}  u^{-\alpha-2 }e^{-u} du \approx  \theta^{-2} (1-t)^{\alpha +2}  e^{-\frac{\theta^2}{4(1-t)}}.
\end{align*}
Finally, if \( \theta^2 \geq \frac14\), we have
\begin{align*}
A(\alpha, \theta, t) & \approx 4^{-\alpha-1}\theta^{2\alpha +2} \int_{\theta^2/2}^{\frac{\theta^2}{4(1-t)}}  u^{-\alpha-2}   e^{-u} du \approx \theta^{-2} e^{-\frac{\theta^2}{2}},\\
B(\alpha, \theta, t) & = 4^{-\alpha-1}\theta^{2\alpha +2} \int^\infty_{\frac{\theta^2}{4(1-t)}}  u^{-\alpha -2 }e^{-u} du \approx \theta^{-2} (1-t)^{\alpha +2}  e^{-\frac{\theta^2}{4(1-t)}}.
\end{align*}
This completes the proof.

\section*{Statements and Declarations}
\textbf{Funding.} The work of T. Chang was supported by the National Research Foundation of Korea (NRF), grant no. RS-2023-00244630. The work of K. Kang and C. Min was supported by the National Research Foundation of Korea (NRF), grant nos. RS-2024-00336346 and RS-2024-00406821.

\textbf{Conflict of interests.} The authors declare that they have no competing interests.

\textbf{Data availability.} Data sharing is not applicable to this article as no datasets were generated or analyzed during the current study.

\textbf{Use of Artificial Intelligence.} During the preparation of this manuscript, the authors used ChatGPT and Claude solely for language editing and polishing of the manuscript. The authors take full responsibility for the correctness and validity of all mathematical results presented in this paper.

\end{document}